\documentclass[11pt,letterpaper]{article}
\pdfoutput=1 
\usepackage[
    letterpaper,
    textwidth=6.5in,   
    textheight=9in,    
    heightrounded
]{geometry}
\usepackage{amsmath, amssymb, amsthm, mathabx, mathrsfs, mathtools, bbm, dsfont}
\usepackage{graphicx, color}
\usepackage{array, multirow}
\usepackage[font=small,labelfont=bf]{caption} 
\usepackage{lipsum}
\usepackage{tabularx}
\usepackage{booktabs}
\usepackage{graphicx}
\usepackage{epstopdf}
\usepackage{algorithmic}
\usepackage{comment, todonotes}

\usepackage{tikz}

\AtBeginDocument{%
	\setlength{\oddsidemargin}{\dimexpr(\paperwidth-\textwidth)/2-1in}%
	\setlength{\evensidemargin}{\oddsidemargin}%
	\setlength{\topmargin}{%
		\dimexpr(\paperheight-\textheight)/2-\headheight-\headsep-1in}%
}

\DeclareMathOperator*{\argmin}{\arg\!\min}

\usepackage[shortlabels]{enumitem}

\usepackage[sort]{natbib}

\setcitestyle{numbers,square,comma}
\usepackage{hyperref}
\hypersetup{colorlinks=true,
			urlcolor=blue,
			linkcolor=blue,
			citecolor=blue,
			bookmarksdepth=paragraph}
\usepackage[nameinlink]{cleveref}

\usepackage[nameinlink]{cleveref}
\crefname{equation}{}{}
\crefname{section}{section}{sections}
\crefname{figure}{figure}{figures}
\crefname{table}{table}{tables}
\crefname{example}{example}{examples}
\crefname{proposition}{proposition}{propositions}
\Crefname{section}{Section}{Sections}
\Crefname{figure}{Figure}{Figures}
\Crefname{table}{Table}{Tables}
\Crefname{definition}{Definition}{Definitions}
\Crefname{theorem}{Theorem}{Theorems}
\Crefname{remark}{Remark}{Remarks}
\Crefname{example}{Example}{Examples}
\Crefname{proposition}{Proposition}{Propositions}
\numberwithin{equation}{section}

\newtheorem{definition}{Definition}[section]
\newtheorem{theorem}[definition]{Theorem}
\newtheorem{lemma}[definition]{Lemma}
\newtheorem{corollary}[definition]{Corollary}
\newtheorem{proposition}[definition]{Proposition}
\newtheorem{claim}[definition]{Claim}

\newtheorem{remark}[definition]{Remark}

\newcommand{\C}{\mathbb{C}}

\newcommand{\R}{\mathbb{R}}
\newcommand{\N}{\mathbb{N}}

\title{Reconstruction of frequency-localized functions from pointwise samples via least squares and deep learning}

\author{A. Martina Neuman$^1$, Andr\'es Felipe Lerma Pineda$^1$, \\
Jason J. Bramburger$^2$, Simone Brugiapaglia$^2$}
\date{\small $^1$Faculty of Mathematics, University of Vienna, Vienna, Austria\\
$^2$Department of Mathematics and Statistics, Concordia University, Montr\'eal, QC, Canada}

\begin{document}

\maketitle

\abstract{Recovering frequency-localized functions from pointwise data is a fundamental task in signal processing. We examine this problem from an approximation-theoretic perspective, focusing on least squares and deep learning-based methods. First, we establish a novel recovery theorem for least squares approximations using the Slepian basis from uniform random samples in low dimensions, explicitly tracking the dependence of the bandwidth on the sampling complexity. Building on these results, we then present a recovery guarantee for approximating bandlimited functions via deep learning from pointwise data. This result, framed as a practical existence theorem, provides conditions on the network architecture, training procedure, and data acquisition sufficient for accurate approximation. To complement our theoretical findings, we perform numerical comparisons between least squares and deep learning for approximating one- and two-dimensional functions. We conclude with a discussion of the theoretical limitations and the practical gaps between theory and implementation.}



\maketitle

\section{Introduction}

The study of frequency-localized functions is fundamental to signal processing and sampling theory. For example, reconstructing bandlimited functions from pointwise samples is the object of the classical Nyquist-Shannon sampling theorem, which establishes that a bandlimited signal can be reconstructed if the sampling rate is twice its bandwidth \citep{jerri1977shannon}. Frequency-localized and bandlimited signals are ubiquitous in engineering and signal processing, with far-reaching applications from geophysical signal estimation \citep{simons2010slepian}, radar imaging \citep{chen2002highly}, radio transmission \citep{cabric2004implementation}, and compressed sensing \citep{tropp2009beyond}.

Motivated by these applications, we study the approximation of multi-dimensional frequency-localized functions from uniform random samples using both least squares regression and deep learning. Specifically, we consider least squares reconstruction methods based on the \emph{Slepian basis}, or the \emph{prolate spheroidal wave functions (PSWFs)} \citep{osipov2013prolate}. 
Introduced by Slepian et al. at Bell Laboratories \citep{slepian1961prolate, slepian1964prolate, slepian1965some}, this orthogonal basis consists of functions optimally concentrated in both space (or time) and frequency \citep{moore2004prolate, simons2010slepian, walter2003sampling, xiao2001prolate}, making it particularly effective for approximating frequency-localized or bandlimited functions confined to a bounded spatial domain. Continued research on the Slepian basis has underscored its distinctive mathematical properties, as reviewed in \citep{wang2017review}. We include a brief discussion on the numerical benefits of using Slepian functions in applications at the end of Section~\ref{sec:bandlimited}.



Building on recent advances in deep learning, we further consider the approximation of smooth functions via \emph{neural networks (NNs)}. Here, linear expansions in the Slepian basis are replaced by nonlinear artificial NNs. This approach is motivated by the demonstrated capability of NNs to learn piecewise smooth functions \citep{eckle2019comparison,petersen2018optimal}, with established dimension-independent approximation rates in the Fourier \citep{barron1993universal} and Radon domains \citep{parhi2022near}. These results strongly suggest that deep learning can provide an efficient means of modeling frequency-localized functions in applications requiring high spatial resolution.

Our contribution to recovering frequency-localized functions via deep learning falls within the class of approximation theorems known as \emph{practical existence theorems (PETs)}, as conceived in \citep{adcock2021gap} and further developed in \citep{adcock2022deep}; see also \citep{adcock2024learning} for a review. 
Traditional universal approximation theorems establish the existence of NNs capable of approximating functions from a given target class with arbitrary accuracy, provided certain conditions on network architecture are met. By leveraging results from sampling and approximation theories, PETs deliver recovery guarantees for trained NNs, contingent on sufficient conditions related to the training data and optimization strategy. Recent developments in PETs include using NNs to approximate orthogonal polynomials for basis expansions in Hilbert spaces \citep{daws2019analysis, de2021approximation, opschoor2022exponential}, realizing spectral Fourier collocation approximations to PDE solutions through physics-informed NNs \citep{brugiapaglia2024physics}, and applying convolutional auto-encoders for reduced-order modeling \citep{franco2025practical}.


In this work, we prove a PET for approximating frequency-localized, bandlimited functions, using an intermediate step that represents Slepian basis functions through Legendre polynomials, which can in turn be modeled by NNs \citep{opschoor2022exponential}.
We build on results from \citep{adcock2022deep}, leveraging previously established findings on least squares approximation sample complexities (see, e.g., \citep[Chapter~5]{adcock2022sparse} and references therein). This means that theoretically our NNs are assumed to be partially pre-trained. A closely related field of study to these ideas is \emph{transfer learning} \citep{pan2009survey}, where some layers of a previously trained model are fine-tuned on new data to solve a similar task. Specifically, this method of fine-tuning (or retraining) the last NN layer has been shown to improve the accuracy of neural networks when dealing with spurious features in the data \citep{kirichenko2022last,strombergenhancing}.

\subsection{Main contributions}

Our main theoretical contributions in this paper can be found in Section~\ref{sec:main_results} and are summarized as follows. 

First, Theorem~\ref{thm:leastsquaresampling} provides a sufficient condition on the number of random samples required for least squares approximation methods to accurately recover frequency-localized functions defined over the hypercube $[-1,1]^d$, where $d = 1, 2, 3$, using an expansion in Slepian functions of bandwidth $\mathsf{w} \geq 1$. The linear expansion is supported on a \emph{hyperbolic cross}, a popular multi-index set used in multi-variate function approximation \citep{dung2018hyperbolic}. Specifically, we demonstrate that to recover the best approximation error in a hyperbolic cross of order $n$, it is sufficient to collect $n^{\gamma(\mathsf{w})}$ random samples, up to a constant and a logarithmic factor, where $\gamma(\mathsf{w}) = \lceil \log_2(24\mathsf{w}) \rceil$. 


Second, Theorem~\ref{thm:ApproxSpan} shows that assuming the same sample complexity bound as in Theorem~\ref{thm:leastsquaresampling} it is possible to accurately approximate frequency-localized functions using deep neural networks (NNs). In particular, this finding demonstrates the existence of a class of trainable neural networks that can be fitted to the given data by solving a least squares-type optimization problem. A key auxiliary result in our proof is Proposition~\ref{prop:ApproxSlepian}, which provides a NN construction of Slepian functions.

Third, in Section~\ref{sec:numerics}, we run a numerical comparison between least squares and deep learning for approximating frequency-localized functions in dimensions one and two. 
Our experiments show that both techniques exhibit convergence in approximation error as the number of samples increases while also underscoring the influence of the function's domain on their performance. Moreover, we propose a Slepian-based initialization strategy inspired by the proof of Theorem~\ref{thm:ApproxSpan} and show that it can outperform standard random initializations in NN training.

\subsection{Paper outline}

After providing some essential background notions and illustrating the problem setting in Section~\ref{sec:background}, we state our theoretical guarantees for least squares and deep learning in Section~\ref{sec:main_results}. Our theory is complemented by several numerical illustrations, carried out in Section~\ref{sec:numerics}. Section~\ref{sec:proofs} contains the proofs of the theoretical results stated in Section~\ref{sec:main_results} and constitutes the main technical core of the paper. We conclude by discussing the limitations of our results and directions of future work in Section~\ref{s:conclusions}. Appendix~\ref{appx:basicnetworks} contains auxiliary results on neural networks needed for some of the proofs in Section~\ref{sec:proofs}.

\section{Setup} \label{sec:background}

In this section we provide the appropriate background material to present our results. Precisely, in Section~\ref{sec:bandlimited} we provide an overview of the relevant literature on bandlimited functions, in turn leading up to the definition of Slepian functions. Then, in Section~\ref{sec:problem} we describe the framework for the least squares approximation and formally define the class of NNs considered in this paper. Before proceeding, we note that throughout this paper we denote $\N_0:=\N\cup\{0\}$ as the set of non-negative integers. For a vector $\vec{v} = (v_j)_{j=1}^n\in\mathbb{C}^n$, we will make use of the (Euclidean) $\ell^2$-norm and the so-called $\ell^0$-norm of $\vec{v}$, defined as $\|\vec{v}\|_2 := (\sum_{j=1}^n |v_j|^2)^{1/2}$ and $\|\vec{v}\|_0 := \texttt{\#} \{ j: v_j \neq 0  \}$, respectively. Moreover, by treating a matrix $A\in\mathbb{R}^{m\times n}$ as a vector, we write $\|A\|_2$ to denote its Frobenius norm.

\subsection{Bandlimited and Slepian functions} 
\label{sec:bandlimited}

Given a function $f:\R^{d}\to\mathbb{C}$, its Fourier transform is defined to be
\begin{equation*} 
    \hat{f}(\vec{v}):=\int_{\R^{d}} f(\vec{x})e^{-i\vec{v}\cdot\vec{x}}\,{\rm d}\vec{x}, \qquad \forall \vec{v}\in\R^d, 
\end{equation*}
with inverse transform 
\begin{equation} \label{inverseF}
    \widecheck{f}(\vec{x}):= \frac{1}{(2\pi)^d}\int_{\R^{d}} f(\vec{v})e^{i\vec{v}\cdot\vec{x}}\,{\rm d}\vec{v}, \qquad \forall \vec{x}\in\R^d,
\end{equation}
whenever either exists. 
The Fourier transform is known to be a unitary map on $L^2(\R^d)$.
Within this space lies a linear subspace of bandlimited functions, distinguished by their compactly supported Fourier transforms, as formalized in the following definition.

\begin{definition} Let $\mathsf{w}>0$. A function  $f\in L^2(\R^d)$ is said to be \emph{$\mathsf{w}$-bandlimited} if $\hat{f}=0$ for almost every $x\in\R^d\setminus [-\mathsf{w},\mathsf{w}]^d$. Equivalently,
\begin{equation} \label{eq:bandlimitedequiv}
    f(\vec{x}) = \frac{1}{(2\pi)^d} \int_{[-\mathsf{w},\mathsf{w}]^d} \hat{f}(\vec{v})e^{i\vec{v}\cdot\vec{x}}\,{\rm d}\vec{v}.
\end{equation}
\end{definition}

When $d = 1$ the space of $\mathsf{w}$-bandlimited functions is referred to as the \emph{Paley-Wiener space} \citep[Section~2.2]{zayed2018advances}, denoted by
\begin{equation*}
    {\rm PW}_\mathsf{w} := \{f\in L^2(\R): \mathrm{supp}(\hat{f})\subset [-\mathsf{w},\mathsf{w}]\}.
\end{equation*}
The uncertainty principle asserts that no nontrivial $f\in L^2(\R)$ can have both the set $\{x: f(x)\neq 0\}$ and the set $\{v: \hat{f}(v)\neq 0\}$ confined to finite measures (see, e.g., \citep[Proposition~8]{amrein1977support} or \citep[Theorem~2]{benedicks1985fourier}). 
Nevertheless, the ``Bell Labs theory'' presents a framework for obtaining functions that achieve optimal concentration in both the time and frequency domains, as we now describe in greater detail. 

\paragraph{The Slepian basis in one dimension.} 
For $T,\mathsf{w}>0$, we define the following two restriction operators on $L^2(\mathbb{R})$, using the notation in \eqref{inverseF}, as
\begin{equation*}
    P_\mathsf{w} f := (\hat{f}\mathbbm{1}_{[-\mathsf{w},\mathsf{w}]})\, \widecheck{} \quad \text{ and } \quad Q_{T} f := f\mathbbm{1}_{[-T,T]},
\end{equation*}
where $\mathbbm{1}_A$ denotes the indicator function for a set $A\subset\R$. It is readily verified that $P_\mathsf{w}(L^2(\R)) = {\rm PW}_{\mathsf{w}}$. 
Since no nontrivial $f \in L^2(\R)$ remains invariant under both $P_\mathsf{w}$ and $Q_T$, an alternative approach is to assess the extent to which the energy of a function in ${\rm PW}_{\mathsf{w}}$ is concentrated on $[-T,T]$.
Particularly, we consider the operator $P_\mathsf{w}Q_T$, that is,
\begin{equation*}
    P_\mathsf{w}Q_Tf(x) = \frac{1}{2\pi}\int_{-\mathsf{w}}^{\mathsf{w}} \int_{-T}^T f(t)e^{-itv}e^{ivx}\,{\rm d}t{\rm d}v, \qquad \forall x\in\R.
\end{equation*}
The operator $P_\mathsf{w}Q_T$ is a Hilbert-Schmidt, hence compact, operator on ${\rm PW}_{\mathsf{w}}$ (see, e.g., \citep[Section~3.3.3]{hogan2005time}).
A quick calculation shows that $P_\mathsf{w}Q_Tf = P_{T\mathsf{w}}Q_1 f(T\cdot)$ for all $f \in L^2(\R)$, and hence without loss of generality we may set $T=1$ and simply define $\mathcal{Q}_{\mathsf{w}} := P_\mathsf{w}Q_1$. From a straightforward calculation, we can equivalently write
\begin{equation} \label{eqdef:Qw}
     \mathcal{Q}_{\mathsf{w}}f(x) = P_\mathsf{w}Q_1f(x)= \int_{-1}^1 f(t)\,\frac{\sin\mathsf{w}(x-t)}{\pi(x-t)}\,{\rm d}t, \qquad \forall x\in\R.
\end{equation}
We are particularly interested in the eigenfunctions of $\mathcal{Q}_{\mathsf{w}}$, i.e., the solutions to the homogeneous Fredholm integral equation of the second kind, $\mathcal{Q}_{\mathsf{w}} f = \mu f$, on ${\rm PW}_{\mathsf{w}}$.
It follows that the corresponding eigenvalue roughly indicates ``how much'' of an eigenfunction is fixed by $\mathcal{Q}_{\mathsf{w}}$.
The spectrum of $\mathcal{Q}_{\mathsf{w}}$ is countable and accumulates at $0$; we denote these eigenvalues as
\begin{equation} \label{eq:muj}
    1> \mu_{\mathsf{w},0} >\mu_{\mathsf{w},1}>\dots> 0,
\end{equation}
where $\lim_{j\to\infty} \mu_{ \mathsf{w}, j} = 0$ \citep[Section~3.2.1]{osipov2013prolate}. 
An orthogonal basis for the space $L^2_{\bf u}([-1,1])$, which is the $L^2$-Lebesgue space on $[-1,1]$ with norm calibrated to the uniform measure ${\bf u}$ on $[-1,1]$, can be generated using the eigenfunctions of $\mathcal{Q}_{\mathsf{w}}$.
Specifically, for $j \geq 0$, let $\varphi_{\mathsf{w},j}\in{\rm PW}_{\mathsf{w}}\cap L_{{\bf u}}^2([-1,1])$ to be the \emph{normalized} eigenfunction associated with $\mu_{\mathsf{w},j}$, i.e., $\mathcal{Q}_{\mathsf{w}}\varphi_{\mathsf{w},j} = \mu_{\mathsf{w},j}\varphi_{\mathsf{w},j}$ and
\begin{equation*} 
    \|\varphi_{\mathsf{w},j}\|^2_{L^2_{{\bf u}}[-1,1]} = \int_{-1}^1 |\varphi_{\mathsf{w},j}(x)|^2 \, {\rm d}{\bf u}(x) = 1.
\end{equation*}
We refer to the set $\{\varphi_{\mathsf{w},j}\}_{j\in\N_0}$ of normalized eigenfunctions, the Slepian basis functions (the prolate spheroidal wave functions), illustrated in Figure~\ref{fig:PSAF} for $\mathsf{w}=16$.
From \eqref{eq:muj}, it is apparent that the function in ${\rm PW}_{\mathsf{w}}$ that minimizes the energy loss most effectively when first time-limited to $[-1,1]$ and subsequently bandlimited to $[-\mathsf{w},\mathsf{w}]$, is $\varphi_{\mathsf{w},0}$. 
Thus, heuristically, a function is interpreted as being $(1,\mathsf{w})$-time-frequency localized if it is a finite linear combination of eigenfunctions $\varphi_{\mathsf{w},j}$ of $\mathcal{Q}_{\mathsf{w}}$ corresponding to eigenvalues $\mu_{\mathsf{w},j}$ close to $1$, or equivalently, those with low indices $j$.

\begin{figure}[t]
    \centering
    \includegraphics[width = 0.49\textwidth]{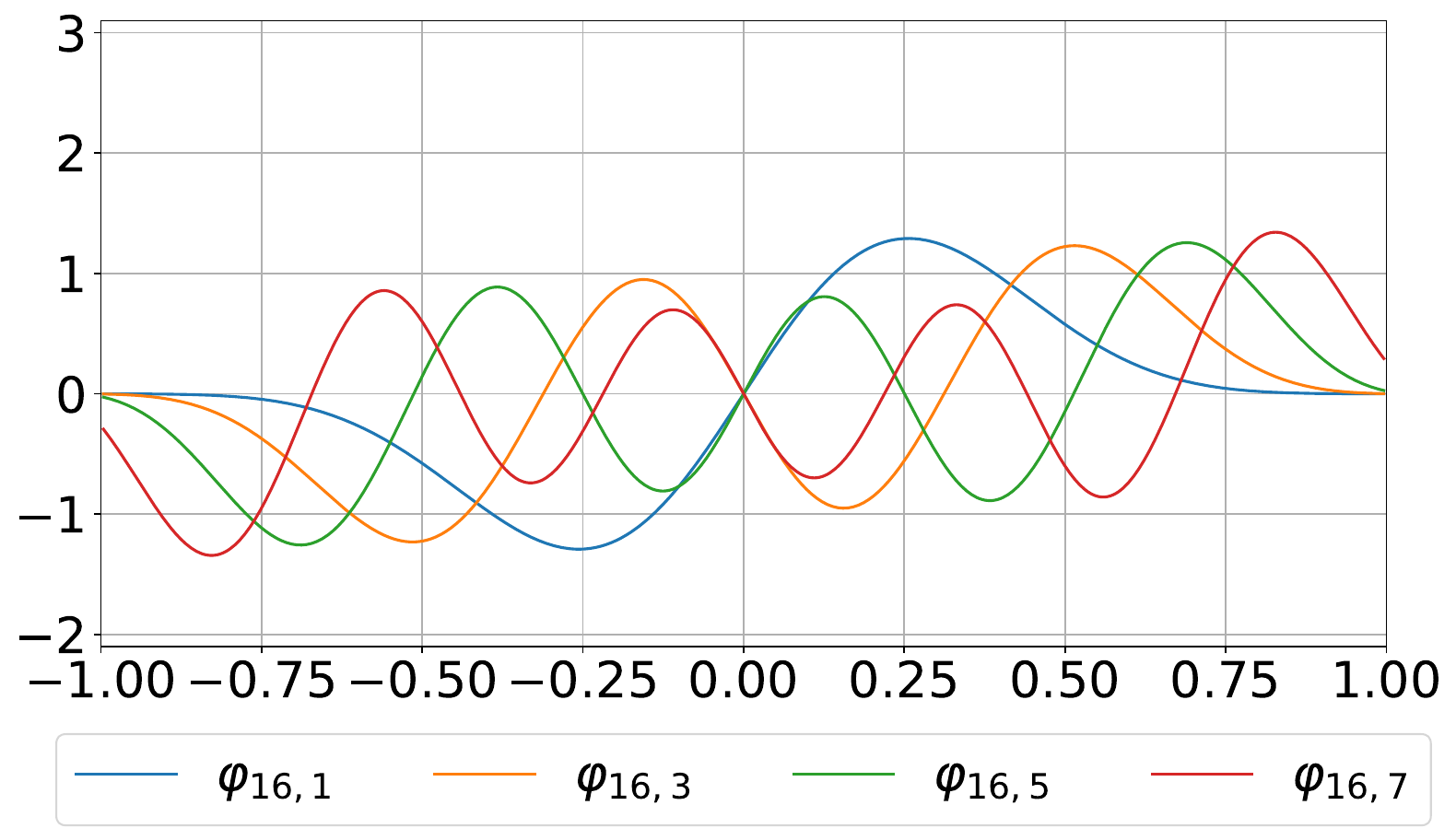} 
    \includegraphics[width = 0.49\textwidth]{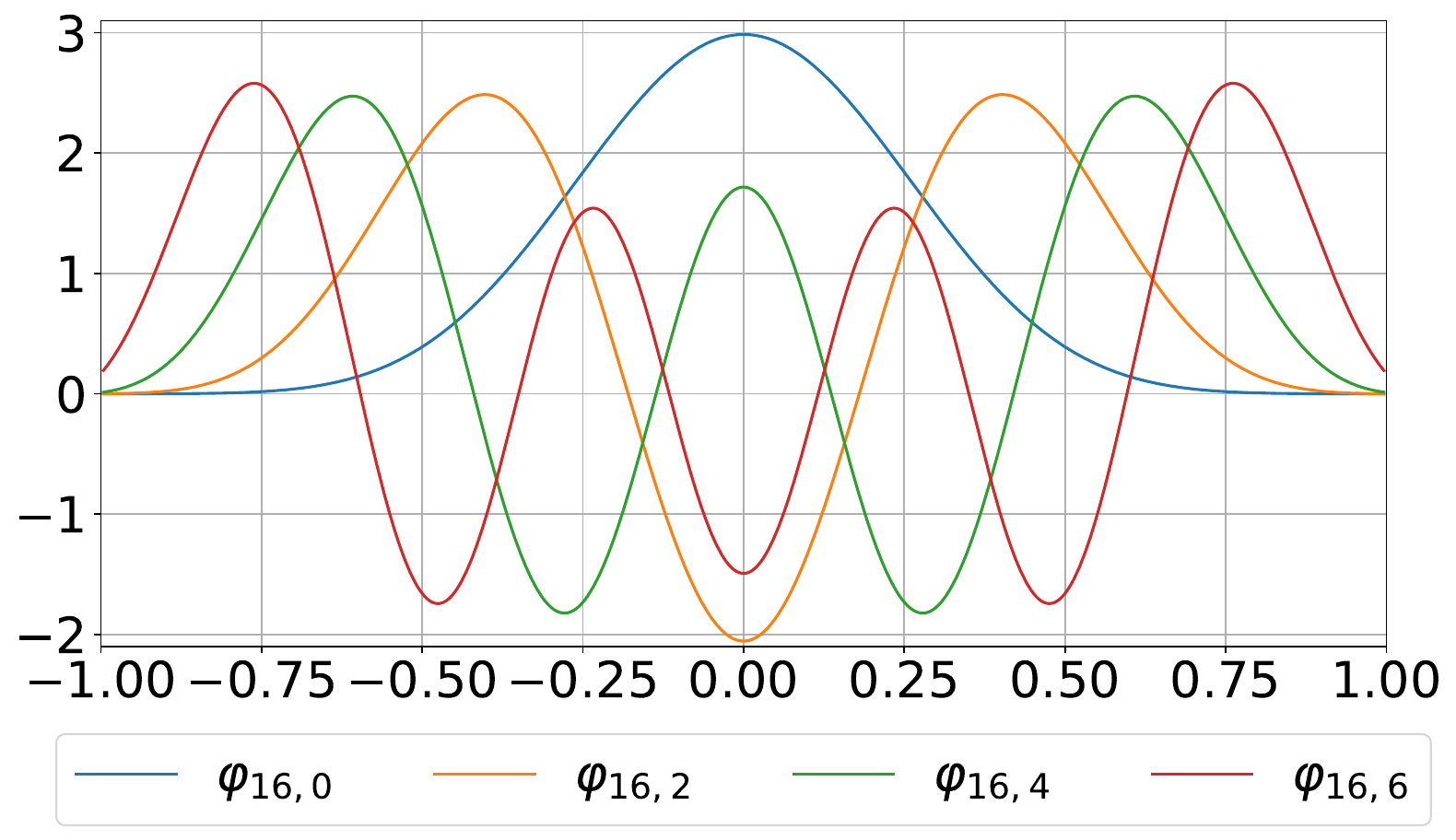}
    \caption{Plot of Slepian functions (PSWFs) $\varphi_{16,j}$ with $\mathsf{w} = 16$ for odd (left) and even (right) indices $j = 0,1,2,3,4,5,6,7$.} 
    \label{fig:PSAF} 
\end{figure}

An alternative way to arrive at the Slepian basis comes in the form of solutions to a Sturm-Louville equation. Precisely, consider the second-order boundary value problem
\begin{equation} \label{eqdef:Lw}
    \mathcal{L}_{\mathsf{w}}f(x) := -\frac{d}{dx}\bigg((1-x^2)\frac{df}{dx}(x)\bigg) + \mathsf{w}^2x^2f(x), \qquad \forall x\in (-1,1).
\end{equation}
The normalized eigenfunctions of this problem are exactly the PSWFs, $\mathcal{L}_{\mathsf{w}}\varphi_{\mathsf{w},j} = \chi_{\mathsf{w},j}\varphi_{\mathsf{w},j}$, but now with the eigenvalues satisfy the ordering
\begin{equation} \label{eq:spectrumchi}
    0< \chi_{\mathsf{w},0} < \chi_{\mathsf{w},1} <\cdots
\end{equation}
with $\lim_{j\to\infty} \chi_{\mathsf{w},j} = \infty$. The progression of the spectrum \eqref{eq:spectrumchi} will play a key role in our subsequent analysis of the $L^{\infty}$-norm of a Slepian basis function on $[-1,1]$, a prerequisite for the proofs of the main results.

\paragraph{The Slepian basis in $d$ dimensions.}
To produce a multi-dimensional Slepian basis, we will simply tensorize the one-dimensional basis. That is, for $d \geq 2$ we construct a $d$-dimensional Slepian basis function as the $d$-tensorization of its one-dimensional counterparts
\begin{equation} \label{tensor}
    \varphi_{\mathsf{w},\vec{\nu}} := \varphi_{\mathsf{w},\nu_1} \otimes \cdots \otimes \varphi_{\mathsf{w},\nu_d},
\end{equation}
for a multi-index $\vec{\nu}=(\nu_1,\dots,\nu_d)\in\N^d_0$. Continuing to denote ${\bf u}$ to be the uniform measure on $[-1,1]^d$, it follows from the preceding discussion that $\{\varphi_{\mathsf{w},\vec{\nu}}\}_{\vec{\nu}\in\N_0^d}$ forms an orthonormal basis of $L^2_{{\bf u}}([-1,1]^d)$, which we simply refer to as the Slepian basis in $d$ dimensions.

\begin{remark}[Practical advantages of Slepian representations]
Slepian functions are specifically designed for applications involving spatiospectral concentration, such as spectral estimation or problems constrained by geometry. This reflects a fundamental principle in approximation theory: an approximation is most accurate and efficient when the basis mirrors the intrinsic structure of the target space. The practical utility of Slepian functions and their variants (Slepian-like bases) is well documented. For instance, \citep{simons2006spatiospectral} employed them to obtain leakage-controlled spectral estimates of geophysical data localized to Earth’s continents and oceans, in contrast to spherical harmonics, whose energy is distributed globally.

Associated with a spatiospectral concentration problem is its \textit{Shannon number}, which quantifies the effective dimensionality of the space of functions simultaneously localized in space (or time) and frequency. This number is known to be roughly proportional to the product between the spatial (or temporal) and frequency domain areas. Remarkably, a defining feature of the corresponding Slepian basis is the rapid decay of their eigenvalues (see e.g. \eqref{eq:muj}) beyond a certain index, which in one dimension has been rigorously shown to closely align with the Shannon number (see \citep[Corollary 3]{karnik2021improved} or \eqref{eq:mu_bounds}). As a result, only the first few Slepian functions contribute significantly, while the remainder have negligible impact, enabling sparse and efficient representations of similarly localized functions. This behavior naturally yields a dimensionality reduction determined entirely by the domain, to which the Slepian basis is inherently adapted. In contrast, classical bases such as Chebyshev or Legendre polynomials provide no intrinsic mechanism for such compression.

Moreover, Slepian functions offer notable advantages in spectral and spectral-element methods for nearly bandlimited, highly oscillatory solutions, primarily due to their nearly uniform oscillatory behavior, which induces quasi-uniform spatial grids. This uniformity improves spatial resolution and allows for longer stable time steps, giving the Slepian basis a distinct advantage in applications such as weather prediction. By contrast, 
Chebyshev and Legendre polynomials have non-uniform oscillations that significantly cluster near the boundaries, limiting their efficiency in high-resolution computations \citep{boyd2004prolate}.
Similarly, for eigenvalue problems \citep[Section~5.3]{wang2017review}, Slepian-based collocation or Galerkin methods deliver superior resolution of oscillatory waves compared with traditional polynomial bases, requiring fewer points per wavelength to achieve a given accuracy.
\end{remark}

\subsection{Problem setting} \label{sec:problem}

To further initialize the work in this paper, in this subsection we present the main elements of our problem setting. To remain consistent with the relevant literature, the notation adopted here is analogous to that of \citep[Section~5]{adcock2022sparse}. 
Letting $m\in\mathbb{N}$, we seek to approximate a continuous function $f\in \mathcal{C}([-1,1]^d)$ 
from noisy samples
\begin{equation*}
    f(\vec{y}_1) + \eta_1,\  \dots,\ f(\vec{y}_m) + \eta_m,
\end{equation*}
where $\vec{y}_1,\dots, \vec{y}_m \in [-1,1]^d$ are random sample points drawn independently from the uniform measure ${\bf u}$. Note that the assumption $f \in \mathcal{C}([-1,1]^d)$ is only made for pointwise evaluations of $f$ to be well-defined, but it can be weakened to, for example, having $f$ be piecewise continuous. Defining the noise vector as $\vec{e}:=\frac{1}{\sqrt{m}}(\eta_j)_{j=1}^m$, we assume throughout a 
deterministic noise model, meaning that the $\eta_j$'s are arbitrary complex numbers.

First, we approximate $f$ using a least squares fit estimator from a hypothesis set composed of linear combinations of Slepian basis functions. 
Let $\Lambda\subset\N_0^d$ be a finite index set, whose cardinality is denoted as $\texttt{\#}\Lambda$. Then, for a fixed $\mathsf{w}>0$, we consider the orthonormal Slepian system $\{\varphi_{\mathsf{w},\vec{\nu}}\}_{\vec{\nu}\in\Lambda}$ whose linear span $\mathcal{S}_{\Lambda}$ constitutes a $(\texttt{\#}\Lambda)$-dimensional subspace in $L_{{\bf u}}^2([-1,1]^d)$.
Supposing $m\geq\texttt{\#}\Lambda$, we approximate $f$ through $f^{\natural}$, the least squares fit among $\mathcal{S}_{\Lambda}$, given by
\begin{equation} \label{leastsquaresproblem}
    f^{\natural} \in \argmin_{g\in\mathcal{S}_{\Lambda}} 
    \frac{1}{m} \sum_{j=1}^m |g(\vec{y}_j) - f(\vec{y}_j) - \eta_j|^2.
\end{equation}
The index set $\Lambda$ will be chosen strategically for our analysis, specifically utilizing the hyperbolic cross, as defined below.

\begin{definition} \label{def:HC}
Let $d,n\in\N$. The $d$-dimensional hyperbolic cross index set of order $n$ is defined to be
\begin{equation*}
    \Lambda^{\rm HC}_{n-1} := \bigg\{\vec{\nu}\in\N_0^{d}: \prod_{k=1}^{d}(\nu_k+1)\leq n\bigg\}.        
\end{equation*}
\end{definition}
To maintain a compact presentation, we omit the dimension $d$ in the notation $\Lambda^{\rm HC}_{n-1}$, allowing it to be inferred from the context.\footnote{Among its desirable properties, the hyperbolic cross has a well-understood cardinality: $\texttt{\#}\Lambda^{\rm HC}_{n-1}\sim n (\log n)^{d-1}/(d-1)!$ for $n \to \infty$ and any fixed $d \in \mathbb{N}$ (see \citep{dobrovol1998number}). A nonasymptotic bound can be found in Lemma~\ref{lem:HCsize}.} 
Continuing, we highlight two key properties of a hyperbolic cross index set. 
From \citep[Lemma~5.15]{adcock2022sparse}, we have the following \emph{slicing property}:\footnote{Note that in this union, the level $k$ stops at $n-1$. Moreover---although we do not use this explicitly---the sets $\Lambda_k$ are \emph{nested}, meaning $\Lambda_j\subset \Lambda_k$ if $j\geq k$. This property, in turn, implies the monotonicity property \eqref{nesting}.}
\begin{equation} \label{slicing}
    \Lambda^{\rm HC}_{n-1} = \bigsqcup_{k=0}^{n-1} \{k\}\times \Lambda_k, 
\end{equation}
where $\bigsqcup$ denotes the disjoint union, and $\Lambda_k := \{(\nu_2,\dots,\nu_d): (k,\nu_2,\dots,\nu_d) \in\Lambda^{\rm HC}_{n-1}\}$. Moreover, for $0\leq k\leq j\leq n-1$, we further have the \emph{monotonicity property}:
\begin{equation} \label{nesting}
    \texttt{\#} \Lambda_j \leq \texttt{\#}\Lambda_k.
\end{equation}

Second, we establish the existence of a class of NNs capable of approximating the Slepian basis functions.
To this end, we introduce the formal definition of an NN, as follows.

\begin{definition}[{\citep[Definition~2.1]{PetV2018OptApproxReLU}}] \label{def:NeuralNetworks}
Let $d, L \in \N$, and $d_1, \dots, d_{L} \in \N$, and $\varrho: \R \to \R$ be given, referred to as an \emph{activation function}. A \emph{neural network (NN)} $\Phi$ with input dimension $d$, output dimension $d_L$, and comprised of $L$ layers is a sequence of matrix-vector pairs
\begin{equation*}
    \Phi = \big((A_1,b_1), \dots, (A_L, b_L)\big), 
\end{equation*}
where $A_l \in \R^{d_l\times d_{l-1}}$ and $b_l \in \R^{d_l}$, for $l =1,\dots,L$, with $d_0 := d$.

Moreover, the \emph{realization} of $\Phi$ is a function ${\rm R}(\Phi): \R^d \to \R^{d_L}$, mapping an input $x\in\R^d$ to an output ${\rm R}(\Phi)(x)\in\R^{d_L}$, defined by the  sequence
\begin{align} \label{eq:NetworkScheme}
    \nonumber x_0 &\equiv x, \\
    x_{l} &:= \varrho(A_{l} \, x_{l-1} + b_l) \quad \text{ for }\quad l = 1, \dots, L-1,\\
    \nonumber x_L &:= A_{L} \, x_{L-1} + b_{L} \equiv {\rm R}(\Phi)(x).
\end{align}
In \eqref{eq:NetworkScheme}, $\varrho$ acts component-wise on vector inputs. We refer to $L(\Phi):= L$ as the \emph{depth} 
and $W(\Phi) := \sum_{j=1}^L \| A_j\|_{0} + \| b_j \|_{0}$ as the \emph{size} of $\Phi$. 
\end{definition}

In this paper, we fix the activation function to be the ReLU function $\varrho(x) = \max \{x,0\}$. The resulting neural networks are referred to as ReLU NNs.

\section{Main results}\label{sec:main_results}

In this section we provide the two main results of our paper: the least squares approximation result (Theorem~\ref{thm:leastsquaresampling}) and a PET (Theorem~\ref{thm:ApproxSpan}). In what follows, a universal constant 
refers to a constant that is independent of any other parameters or hyperparameters of the problem. Additionally, to streamline the presentation, for a bandwidth $\mathsf{w}\geq 1$ we define
\begin{equation} \label{eqdef:gammaw}
    \gamma(\mathsf{w}):= \lceil \log_2(24\mathsf{w}) \rceil,
\end{equation}
to be used in the ensuing statements. We begin in Section~\ref{subsec:LSresult} with our least squares approximation result. Then, our PET is presented in Section~\ref{subsec:LSresult}. We conclude in Section~\ref{ss:more_LS_results} with two additional least squares corollaries.

\subsection{Least squares approximation} \label{subsec:LSresult}

Our first result provides an error bound for the least squares problem \eqref{leastsquaresproblem} of approximating a continuous function 
using Slepian basis functions 
indexed over a hyperbolic cross.
A proof can be found in Section~\ref{sec:leastsquaresproof}.

\begin{theorem}[Recovery guarantee for least squares]  \label{thm:leastsquaresampling}
Let $\beta,\delta\in (0,1)$, $d=1,2,3$, and $\Lambda=\Lambda^{\rm HC}_{n-1}$ be the $d$-dimensional hyperbolic cross of order $n$, where $n\in\mathbb{N}$ if $d=1,2$, and $n\geq 26$ if $d=3$.
For $\mathsf{w}\geq 1$, let $\mathcal{S}_{\Lambda}$ be the linear span of $\{\varphi_{\mathsf{w},\vec{\nu}}\}_{\vec{\nu}\in\Lambda}$ and let $\vec{y}_1, \dots, \vec{y}_m$ be random sample points drawn independently from the uniform measure $\bf u$ on $[-1,1]^d$. 
If 
\begin{equation} 
\label{eq:sample_complexity_LS}
    m \geq ((1-\delta)\log(1-\delta)+\delta)^{-1} (2\texttt{\#}\Lambda)^{2\gamma(\mathsf{w})} \log(\texttt{\#}\Lambda/\beta),
\end{equation}
then, with probability at least $1-\beta$, for every $f \in \mathcal{C}([-1,1]^d)$ 
and every noise vector $\vec{e} = \frac{1}{\sqrt{m}} (\eta_j)_{j=1}^m \in \C^m$, the least squares problem \eqref{leastsquaresproblem} has a unique solution $f^{\natural}\in\mathcal{S}_{\Lambda}$ satisfying 
\begin{equation}\label{eq:AcFinal}
    \| f - f^{\natural} \|_{L_{ {\bf u}}^2([-1,1]^d)} 
    \leq 
    \left(1+ \dfrac{1}{\sqrt{1-\delta}} \right)\inf_{g \in \mathcal{S}_{\Lambda}} \| f - g \|_{L^{\infty}([-1,1]^d)} 
    + \dfrac{1}{\sqrt{1-\delta}} \| \vec{e} \|_2.
\end{equation}
\end{theorem}


The error estimate in Theorem~\ref{thm:leastsquaresampling} provides an example of a \emph{uniform recovery guarantee}, meaning that a single draw of sample points is sufficient to recover any function $f \in \mathcal{C}([-1,1]^d)$ 
with high probability \textit{up to the bound given in} \eqref{eq:AcFinal}. 
The number of samples required in \eqref{eq:sample_complexity_LS} scales as $(2\texttt{\#}\Lambda)^{2\gamma(\mathsf{w})}$, up to a constant and a logarithmic factor.
The corresponding $L^2_{\bf u}$-error is dominated by the $L^\infty$-distance of $f$ from the given Slepian basis and the $\ell^2$-norm of the noise vector, up to multiplicative constants. 
Consequently, while Theorem~\ref{thm:leastsquaresampling} applies for any continuous function $f$, it is most informative for frequency-localized functions that are well-approximated by the Slepian span of $\{\varphi_{\mathsf{w}, \vec{\nu}}\}_{\vec{\nu} \in \Lambda}$. For these functions, the approximation error term $\inf_{g \in \mathcal{S}_{\Lambda}} \|f - g\|_{L^{\infty}([-1,1]^d)}$ is minimal, making the bound in \eqref{eq:AcFinal} conclusive, and ensuring that pointwise samples suffice for accurate recovery.
Particularly, $\mathsf{w}$-bandlimited functions that are exactly linear combinations of elements in $\{\varphi_{\mathsf{w},\vec{\nu}}\}_{\vec{\nu}\in\Lambda}$ can be recovered with accuracy proportional to the noise level $\|\vec{e}\|_2$. 

Further recovery guarantees for least squares are presented in Section~\ref{ss:more_LS_results}.

\begin{remark}[Comparison with polynomial least squares approximation theory] \label{rem:compare}
Most of the literature on function approximation via least squares from random samples employs orthogonal polynomials (e.g., Legendre or Chebyshev) as the approximation basis---see the key contributions \citep{chkifa2015discrete, cohen2013stability, cohen2017optimal,migliorati2014analysis, migliorati2013polynomial} and the surveys \citep{adcock2022sparse, cohen2015approximation} for further references. The proof of Theorem~\ref{thm:leastsquaresampling} follows the same strategy used in the polynomial least squares literature, which relies on the study of the so-called Christoffel function---a key topic of Section~\ref{sec:leastsquaresproof}. These polynomial least squares approximation results yield sample complexity bounds that are, at first glance, more favorable than \eqref{eq:sample_complexity_LS}. For example, under the same assumptions as in Theorem~\ref{thm:leastsquaresampling} it can be shown that (up to a logarithmic factor) a number of random samples $m$ scaling like $(\texttt{\#}\Lambda)^2$ is sufficient to recover the best approximation $L^\infty$-error of $f$ with respect to Legendre polynomials. A similar result holds for Chebyshev polynomials and random samples distributed according to the arcsine measure with a scaling proportional to $(\texttt{\#}\Lambda)^{\log_3(2)} \approx (\texttt{\#}\Lambda)^{1.58}$ (up to a logarithmic factor). Moreover, resorting to weighted least squares and picking the random samples according to the so-called Christoffel sampling, one obtains a log-linear sampling complexity (scaling linearly in $\texttt{\#}\Lambda$ up to a logarithmic factor) \citep{cohen2017optimal}.

Despite their seemingly more favorable sampling complexity, recovery guarantees for orthogonal polynomials are not directly comparable to Theorem~\ref{thm:leastsquaresampling}. Indeed, the right-hand side of the error bound \eqref{eq:AcFinal} (which has an analogous form in the orthogonal polynomial setting---see, e.g., \cite[Corollary~5.9]{adcock2022sparse}) contains the best $L^\infty$-approximation error of $f$, which depends on the choice of basis. Consequently, the relevant question becomes: 
which basis is more suitable to efficiently approximate $f$ (e.g., with a potentially smaller set $\Lambda$)?
To the best of our knowledge, this approximation-theoretical question has no definitive answer in the literature.
An interesting theoretical comparison is carried out in, e.g., \citep[Section~7]{boyd2004prolate}, where the benefits of using the Slepian basis to approximate the simple analytic function $f(x) = \exp(-ik x)$ are shown asymptotically for $k \to \infty$. As we see in Section~\ref{sec:numerics}, there are functions (such as those considered in \citep{boyd2004prolate}) for which least squares approximation from random samples is more efficient when employing Slepian basis than orthogonal polynomial bases. 
\end{remark}

\subsection{Practical existence theorem} \label{subsec:PETresult}

Using the definition of $\gamma(\mathsf{w})$ in \eqref{eqdef:gammaw} for $\mathsf{w}\geq 1$, for $n\in\mathbb{N}$ we further define
\begin{equation} \label{eq:B}
    B(d,n) := 
    \begin{cases}
        n^{\gamma(\mathsf{w})} + 1 &\text{ if } d=1, \\
        3n^{2\gamma(\mathsf{w})} + 4n^{\gamma(\mathsf{w})} +2 &\text{ if } d=2, \\
        7n^{3\gamma(\mathsf{w})} + 12n^{2\gamma(\mathsf{w})} + 8n^{\gamma(\mathsf{w})} +3 &\text{ if } d=3,
    \end{cases}
\end{equation}
and 
\begin{equation} \label{eq:M}
    M(d,n) := 
    \begin{cases}
        1 &\text{ if } d=1, \\
        2n^{\gamma(\mathsf{w})} + 1 &\text{ if } d=2, \\
        4n^{2\gamma(\mathsf{w})} + 4n^{\gamma(\mathsf{w})} +2 &\text{ if } d=3.
    \end{cases}
\end{equation}

Our second result presents a PET for the least squares approximation of continuous functions, 
with a complexity rate derived from Theorem~\ref{thm:leastsquaresampling}.
A proof is given in Section~\ref{sec:ProofTheo3}. 

\begin{theorem}[Practical existence theorem]\label{thm:ApproxSpan} 
Let $\varepsilon,\delta,\beta \in (0,1)$, $d=1,2,3$, and $\Lambda=\Lambda^{\rm HC}_{n-1}$ be the $d$-dimensional hyperbolic cross of order $n$, where $n\in\mathbb{N}$ if $d=1,2$, and $n\geq 26$ if $d=3$. Let $\mathsf{w}\geq 1$ and $B(d,n)$, $M(d,n)$ be as in \eqref{eq:B}, \eqref{eq:M}, respectively.
Let $\vec{y}_1, \dots, \vec{y}_m$ be independent sample points drawn from the uniform measure $\bf u$ on $[-1,1]^d$, with 
\begin{equation*}
    m \geq ((1-\delta)\log(1-\delta)+\delta)^{-1} (2\texttt{\#}\Lambda)^{2\gamma(\mathsf{w})} \log(\texttt{\#}\Lambda/\beta).
\end{equation*}
Then for every $\varepsilon>0$ sufficiently small satisfying
\begin{equation} \label{epscondition}
    \varepsilon \leq \frac{\sqrt{1-\delta}}{2\sqrt{\texttt{\#}\Lambda} B(d,n)},
\end{equation}
there exists a class of NNs, denoted $\mathcal{N}_{\Lambda,\mathsf{w},\varepsilon}$, such that the following holds with probability at least $1-\beta$. For every $f \in \mathcal{C}([-1,1]^d)$ and every noise vector $\vec{e}=\frac{1}{\sqrt{m}}(\eta_j)_{j=1}^m \in \C^m$, the problem 
\begin{equation} \label{leastsquaresproblem2}
    \Psi^{\natural} \in \argmin_{\Psi \in \mathcal{N}_{\Lambda,\mathsf{w},\varepsilon}} 
    \frac{1}{m} \sum_{j=1}^m |f(\vec{y}_j) - {\rm R}(\Psi)(\vec{y}_j) - \eta_j|^2
\end{equation}
has a unique solution $\Psi^{\natural}\in\mathcal{N}_{\Lambda,\mathsf{w},\varepsilon}$ satisfying
\begin{align} \label{eq:PETFinal}
    \nonumber &\| f - {\rm R}(\Psi^{\natural}) \|_{L_{ {\bf u}}^2([-1,1]^d)} \\
    &\leq 
    \Big(1+ \dfrac{2}{\sqrt{1-\delta}} \Big) \Big(\inf_{g \in \mathcal{S}_{\Lambda}} \|f-g\|_{L^{\infty}([-1,1]^d)} + \sqrt{\texttt{\#}\Lambda} B(d,n)\|g\|_{L^2_{\bf u}([-1,1]^d)}\varepsilon \Big) + \dfrac{2\|\vec{e}\|_2}{\sqrt{1-\delta}}.
\end{align}
Moreover, it is guaranteed that
\begin{align*}   
    L(\Psi^{\natural}) 
    &
    \leq C \big((1+\log_2N_{\star})(N_{\star}+\log_2(N_{\star}/\varepsilon)) + (1+ \log_2 (M(d,n)/\varepsilon)\big),\\
    W(\Psi^{\natural}) 
    & 
    \leq C \texttt{\#}\Lambda \big((N_{\star}^2 + N_{\star})(N_{\star} + \log_2(N_{\star}/\varepsilon)) + (1+ \log_2 (M(d,n)/\varepsilon)\big),
\end{align*}
for a universal constant $C>0$ and $N_{\star}=N_{\star}(\Lambda,\mathsf{w},\varepsilon)\in\mathbb{N}$.
\end{theorem}

A few remarks are in order.
First, similar to Theorem~\ref{thm:leastsquaresampling}, Theorem~\ref{thm:ApproxSpan} states that functions well-approximated by the Slepian span of $\{\varphi_{\mathsf{w},\vec{\nu}}\}_{\vec{\nu}\in\Lambda}$ can be accurately recovered via deep learning using a training set of pointwise samples with size proportional to $(2\texttt{\#}\Lambda)^{2\gamma(\mathsf{w})}$, up a constant and a logarithmic factor. 
The theorem assumes training to be performed by solving a least squares-type optimization problem \eqref{leastsquaresproblem2} that closely resembles \eqref{leastsquaresproblem}, operating over a class of NNs $\mathcal{N}_{\Lambda,\mathsf{w},\varepsilon}$ with specific architectures, determined by $\Lambda$, $\mathsf{w}$, and $\varepsilon$. The class $\mathcal{N}_{\Lambda,\mathsf{w},\varepsilon}$ is defined in \eqref{eqdef:NLambda} and is composed by NNs whose first-to-second-last layers are explicitly constructed to emulate the Slepian basis and whose last layer is made of trainable weights, optimized through \eqref{leastsquaresproblem2}. 

Second, we briefly discuss the upper bounds on $L(\Psi^\sharp)$ and $W(\Psi^\sharp)$, i.e., the depth and width of the optimal network $\Psi^\sharp$, respectively. Note that since training occurs only in the last layer, these bounds are not directly tied to the computational difficulty of training $\Psi^\sharp$. However, they explicitly quantify how the network depth and width depend on $\Lambda$, $\mathsf{w}$, $\varepsilon$. In order to better understand these bounds, we unpack the definition of $N_{\star}=N_{\star}(\Lambda,\mathsf{w},\varepsilon)$. For a one-dimensional index set $\Lambda$, this is given in \eqref{eqdef:Nstar}. For the special case $\Lambda = \Lambda^{\rm HC}_{n-1}$ where $d=1,2,3$, the dependence on $\Lambda,\mathsf{w},\varepsilon$ can be simplified to that on $\mathsf{w},\varepsilon$, since---as explained in the proof of Proposition~\ref{prop:ApproxSlepian}---we have
\begin{equation} \label{theN}
    N_{\star} = \Big \lceil  \max \Big\{2\lfloor e  \mathsf{w} \rfloor +1, \frac{\log(3/(\varepsilon\mu_{\mathsf{w},n-1}))}{\log (3/2)} \Big\} \Big\rceil.
\end{equation}
Here, recall that $\mu_{\mathsf{w},n-1}$ denotes the $n$th-eigenvalue \eqref{eq:muj} of the operator $P_{\mathsf{w}}Q_T$. Moreover, for $\mathsf{w}\geq 1$, the following bounds, adapted from \citep[Corollary~3]{karnik2021improved}, hold for a universal constant $C>0$:
\begin{equation}\label{eq:mu_bounds}
\begin{alignedat}{2}
    \mu_{\mathsf{w},n-1} &\leq 10\exp\Big(-C\Big|\Big\lceil \frac{4\mathsf{w}}{\pi}\Big\rceil-(n-1)\Big|\Big) \qquad &&\text{ if } \qquad n-1 \geq \Big\lceil \frac{4\mathsf{w}}{\pi}\Big\rceil, \\
    \mu_{\mathsf{w},n-1} &\geq 1-10\exp\Big(-C\Big|\Big\lfloor \frac{4\mathsf{w}}{\pi}\Big\rfloor-n\Big|\Big) \qquad &&\text{ if } \qquad 0\leq n-1 \leq \Big\lfloor \frac{4\mathsf{w}}{\pi}\Big\rfloor - 1.
\end{alignedat}
\end{equation}
From \eqref{theN}, \eqref{eq:mu_bounds}, it follows that when $\Lambda = \Lambda^{\rm HC}_{n-1}$ with $d=1,2,3$, the quantity $N_{\star}$ grows linearly with respect to $\mathsf{w}$ and logarithmically with $1/\varepsilon$. Applying this to Theorem~\ref{thm:ApproxSpan} yields a network depth and width that scale at most polynomially with $\mathsf{w}$ and logarithmically with $1/\varepsilon$.

Third, the error bound \eqref{eq:PETFinal} in Theorem~\ref{thm:ApproxSpan} is similar to \eqref{eq:AcFinal} in Theorem~\ref{thm:leastsquaresampling}, so that besides minor changes in the constant factors, the main difference between \eqref{eq:PETFinal} and \eqref{eq:AcFinal} is the presence of an extra term in the best approximation error. 
The impact of this extra term is mild, though, as it scales linearly with an auxiliary parameter $\varepsilon \in (0,1)$ that can be made arbitrarily close to $0$ at the price of increasing the architecture bounds proportionally to $\log(1/\varepsilon)$. 

Finally, we note that Theorem~\ref{thm:ApproxSpan} demonstrates the existence of a class of neural networks that, when trained, can attain performance nearly matching least squares approximation in the Slepian basis. Owing to the nature of its proof (see Section~\ref{sec:ProofTheo3}), however, the theorem does not assert that such networks can outperform the Slepian basis---though this may occur empirically. Moreover, we stress that the main purpose of Theorem~\ref{thm:ApproxSpan} is of theoretical nature (i.e., it provides an improved version of a universal approximation theorem), and it is not intended to propose a new algorithmic recipe.

\subsection{Discussion on further recovery guarantees for least squares} \label{ss:more_LS_results}

We now turn to two additional recovery guarantees for least squares approximation. The reason we present these results is that the recovery guarantee in Theorem~\ref{thm:leastsquaresampling} comes with the drawback of bounding the error in terms of the best approximation of $f$ in the $L^{\infty}$-norm, which is a stronger and more restrictive norm compared to the $L_{ {\bf u}}^{2}$-norm used to measure the recovery error. Thus, here we present two results aimed at addressing this shortcoming. They can be seen as direct corollaries of Theorem~\ref{thm:LS} and Proposition~\ref{prop:Theta}, while their proofs are omitted as they are identical to other published results, which we will provide precise pointers to as we proceed through this subsection. 

The first result is an $L_{ {\bf u}}^{2}$-$L_{ {\bf u}}^{2}$ uniform recovery guarantee in probability. It is stated as follows.

\begin{corollary}[$L_{ {\bf u}}^{2}$-$L_{ {\bf u}}^{2}$ recovery guarantee in probability for least squares]\label{cor:least_squares_L2L2_prob}
In the same setting of Theorem~\ref{thm:leastsquaresampling}, let $\beta, \delta \in (0,1)$ and $f \in \mathcal{C}([-1,1]^d)$. 
Assume that the number of samples $m$ satisfies the condition $m \geq c (2\texttt{\#}\Lambda)^{2\gamma(\mathsf{w})} \log(2\texttt{\#}\Lambda/\beta)$,
where $c>0$ is a universal constant. Then for every noise vector $e\in \mathbb{C}^m$, with probability at least $1-\beta$, the least squares solution $f^{\natural}$ obtained in \eqref{leastsquaresproblem} satisfies
\begin{equation*}
    \| f - f^{\natural} \|_{L_{ {\bf u}}^2([-1,1]^d)} 
    \leq 
    \left(1+ \sqrt{\dfrac{2}{\beta}}\dfrac{1}{\sqrt{1-\delta}} \right)\inf_{g \in \mathcal{S}_{\Lambda}} \| f - g \|_{L_{\bf u}^{2}([-1,1]^d)} 
    + \dfrac{1}{\sqrt{1-\delta}} \| \vec{e} \|_2.
\end{equation*}
\end{corollary}

This result can be obtained from Theorem~\ref{thm:LS} and Proposition~\ref{prop:Theta} by applying the same argument as in \citep[Corollary 5.10]{adcock2022sparse}. The error bound now contains the best approximation error of $f$ with respect to the $L^2_{\bf u}$-norm as desired. However, as opposed to Theorem~\ref{thm:leastsquaresampling}, Corollary~\ref{cor:least_squares_L2L2_prob} is a \emph{nonuniform recovery guarantee}, i.e., the error bound holds with probability $1-\beta$ for a fixed $f$. 
In other words, one single draw of sample points is not sufficient for the validity of the recovery guarantee for all $f \in \mathcal{C}([-1,1]^d)$ simultaneously. Another limitation of Corollary~\ref{cor:least_squares_L2L2_prob} is the poor scaling of the error bound with respect to the failure probability $\beta$. The following result addresses these limitations by providing a recovery guarantee in expectation.

\begin{corollary}[$L_{ {\bf u}}^{2}$-$L_{ {\bf u}}^{2}$ recovery guarantee in expectation for least squares]
\label{cor:least_squares_L2L2_exp}
    Let $\delta, \beta \in (0,1)$ and $f \in \mathcal{C}([-1,1]^d)$ 
    be such that $\| f \|_{L_{ {\bf u}}^2([-1,1]^d)} \leq L$ for some $L>0$. Define the operator $\tau_L : L_{ {\bf u}}^2([-1,1]^d) \to L_{ {\bf u}}^2([-1,1]^d)$ as
    \begin{equation*}
        \tau_L(g) := \min \bigg \{1, \dfrac{L}{ \| g \|_{L_{ {\bf u}}^2([-1,1]^d)}} \bigg \}g.
    \end{equation*}
    Then, under the same assumptions as Theorem~\ref{thm:leastsquaresampling}, if $f^{\natural}$ is the least squares solution of \eqref{leastsquaresproblem}, for every noise vector $e \in \mathbb{C}^m$ we have 
    \begin{equation*}
        \mathbb{E}
        \| f - \tau_L(f^{\natural}) \|^2_{L_{ {\bf u}}^2([-1,1]^d)}
        \leq 
        \left(\dfrac{3- \delta}{1-\delta} \right)\inf_{g \in \mathcal{S}_{\Lambda}} \| f - g \|_{L^2_{\bf u}([-1,1]^d)}  
        + \dfrac{2}{1-\delta} \| \vec{e} \|_2^2 + 4 L^2 \beta.
    \end{equation*} 
\end{corollary}

This result can be obtained from Theorem~\ref{thm:LS} and Proposition~\ref{prop:Theta} by applying the same argument as in \citep[Corollary 5.11]{adcock2022sparse} (originally proposed in \citep[Theorem 2]{cohen2013stability}). Notably, the poor scaling with respect to $\beta$ in the error bound of Corollary~\ref{cor:least_squares_L2L2_prob} is not there anymore. However, this comes at the price of having a recovery guarantee in expectation, as opposed to high probability, while also requiring a priori knowledge of a constant $L$ such that $\| f \|_{L_{ {\bf u}}^2([-1,1]^d)} \leq L$. 

We conclude with a brief remark that the proof of Theorem~\ref{thm:ApproxSpan} can also be modified to incorporate $L^2_{\bf u}$-$L^2_{\bf u}$ recovery guarantees by employing Corollary~\ref{cor:least_squares_L2L2_prob} or \ref{cor:least_squares_L2L2_exp} in place of Theorem~\ref{thm:leastsquaresampling}. 
For the sake of conciseness, we omit the statements of the resulting corollaries.

\section{Numerical experiments}\label{sec:numerics}

In this section, we provide numerical examples illustrating the performance of least squares approximation and deep learning in reconstructing frequency-localized functions from pointwise samples. 
For simplicity, we refer to these methods as ``Least Squares'' and ``Deep Learning'' throughout. 
We describe our numerical setup and discuss the results obtained in dimensions one and two.
The Python code necessary to reproduce our experiments can be found in the GitHub repository: \url{https://github.com/andreslerma01/PSWF-}.

\paragraph{Training set, test set, and error metric.} For a given target function $f$ that we seek to approximate using either Least Squares or Deep Learning, we generate data sets of the form $\{ (x_j, f(x_j) \}_{j=1}^m$ with $x_j$ drawn independently from the uniform distribution on $[-1,1]^d$. 
For a set of $m$ training points, we also consider a test set $\{x_j^{\text{test}}\}_{j=1}^{m_{\text{test}}}$ of size $m_{\text{test}} = 0.2m $, sampled independently according to the uniform distribution on $[-1,1]^d$. The corresponding test error is given by the \emph{Root Mean Square Error (RMSE)}, defined as 
\begin{equation*} 
    \mathcal{E}_{{\text{test}}} := 
    \sqrt{\dfrac{1}{m_{\text{test}}} \sum_{j=1}^{m_{\text{test}}} |f(x_j^{{\text{test}}}) - \tilde{f}(x_j^{{\text{test}}}) |^2},
\end{equation*} 
where $f$ is the target function, and $\tilde{f}$ is the approximation obtained from either method.

\paragraph{Slepian basis functions.}
To construct one-dimensional Slepian basis functions $\varphi_{\mathsf{w}, j}$, we employ the Python package \texttt{scipy.special} \citep{2020SciPy-NMeth}, which is based on the angular solutions of the Helmholtz wave equation of the first kind. 
In fact, each function $\varphi_{\mathsf{w},j}$ is a scalar multiple of a corresponding angular solution \citep{moore2004prolate}. 
To construct two-dimensional Slepian basis functions, we utilize the tensor-product definition \eqref{tensor} $\varphi_{\mathsf{w},\vec{\nu}} = \varphi_{\mathsf{w},\nu_1} \otimes \varphi_{\mathsf{w},\nu_2}$. 
Finally, in both cases, we take the index set to be the hyperbolic cross $\Lambda^{\text{HC}}_{n-1}$ (Definition~\ref{def:HC}) for various orders $n\in\N$.

\subsection{The Slepian basis vs.\ orthogonal polynomials} 
To begin, we illustrate how the Least Squares method with the Slepian basis can outperform other recurrent orthogonal bases. More specifically, we compare the Slepian basis with the Chebyshev and Legendre orthogonal polynomial bases. Recall that Chebyshev polynomials (of the first kind) are defined on the interval $[-1,1]$ as $T_{k}(x) = \cos(k \arccos x)$, for $k \in \mathbb{N}$.
For a definition of Legendre polynomials, see Equation \eqref{eq:NormLegendr} in Section~\ref{sec:prelimformainthm2}, where they are used to prove Theorem~\ref{thm:ApproxSpan}. For more information about these families of orthogonal polynomials, we refer to, e.g., \citep[Section~2.2]{adcock2022sparse}. 
To perform our analysis, we consider the following functions. In one dimension, we take the Gaussian 
\[
g_1(x) := e^{-\pi  x^2}, \qquad \forall x\in [-1,1],
\]
and the complex exponential
\[
g_2(x) := e^{4 \pi i x}, \qquad \forall x\in [-1,1].
\]
In two dimensions, we consider
\[
g_3(x,y) := e^{-\pi(x^2+y^2+0.2xy)}, \qquad \forall (x, y)\in [-1,1]^2.
\]
Note that $g_3$ is non-separable, meaning it is not simply the product of two one-dimensional functions, thus making it purely two-dimensional. For each of these functions, we run two types of experiments. First, we analyze how the approximation error behaves when the number of samples varies while keeping the basis fixed as $\Lambda^{\text{HC}}_{n-1}$ (note that in the one-dimensional case $\Lambda^{\text{HC}}_{n-1} = \{0, \ldots, n-1\}$ corresponds to $n$ basis elements). For each of the functions $g_1, g_2, g_3$, the sample size $m$ takes values from $10$ up to $10000$.
For every experiment using a set of $m$ training points, we also consider a test set $\{x_j^{\text{test}}\}_{j=1}^{m_{\text{test}}}$ of size $m_{\text{test}} = 1000$, sampled independently from the uniform distribution on $[-1,1]^d$. 
Second, we observe the behavior of the error when the number of samples is fixed at $m=1000$ and the order $n$ varies.


\begin{figure}[t!]
    \centering
    \includegraphics[width = 0.40\textwidth]{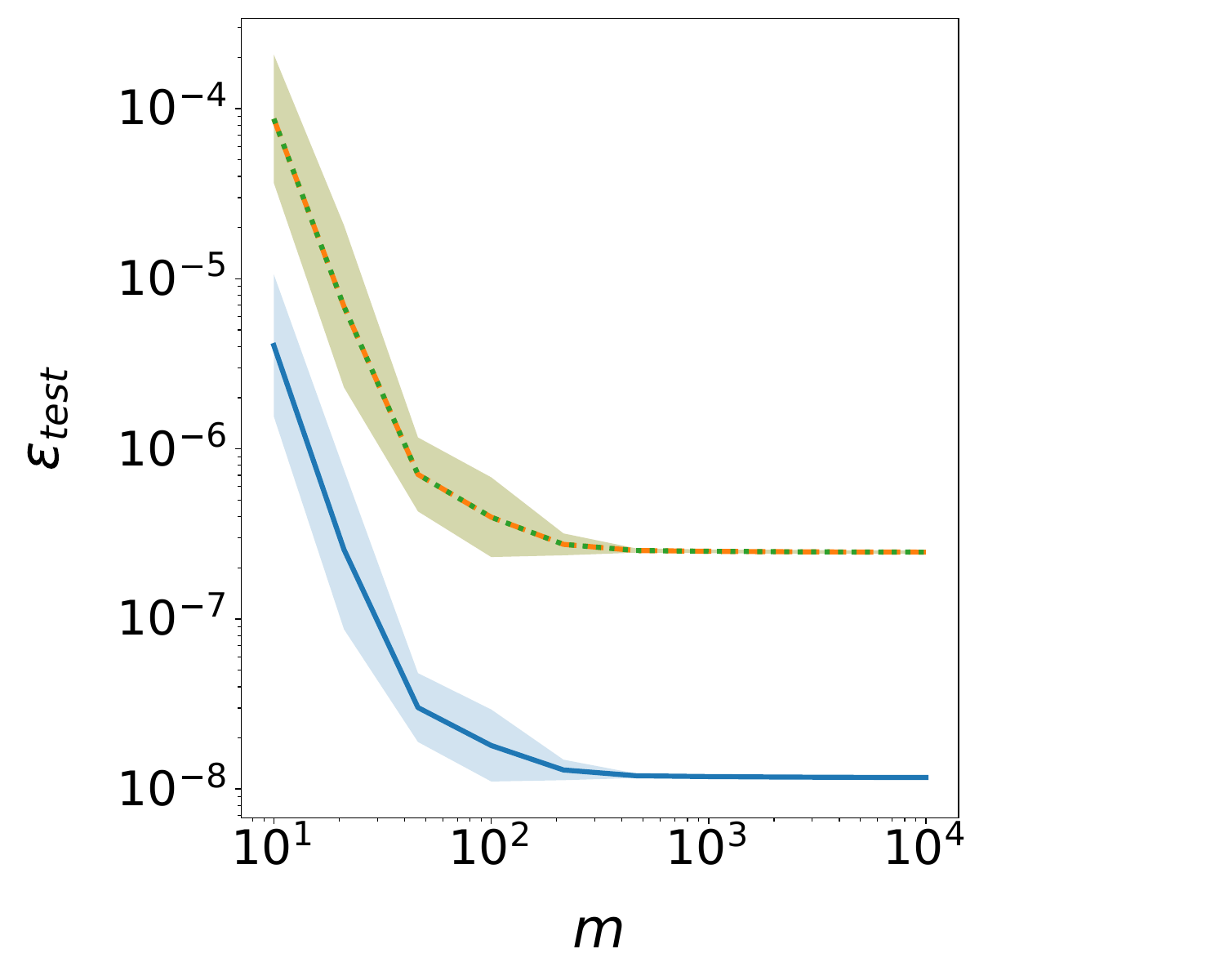} 
    \includegraphics[width = 0.49\textwidth]{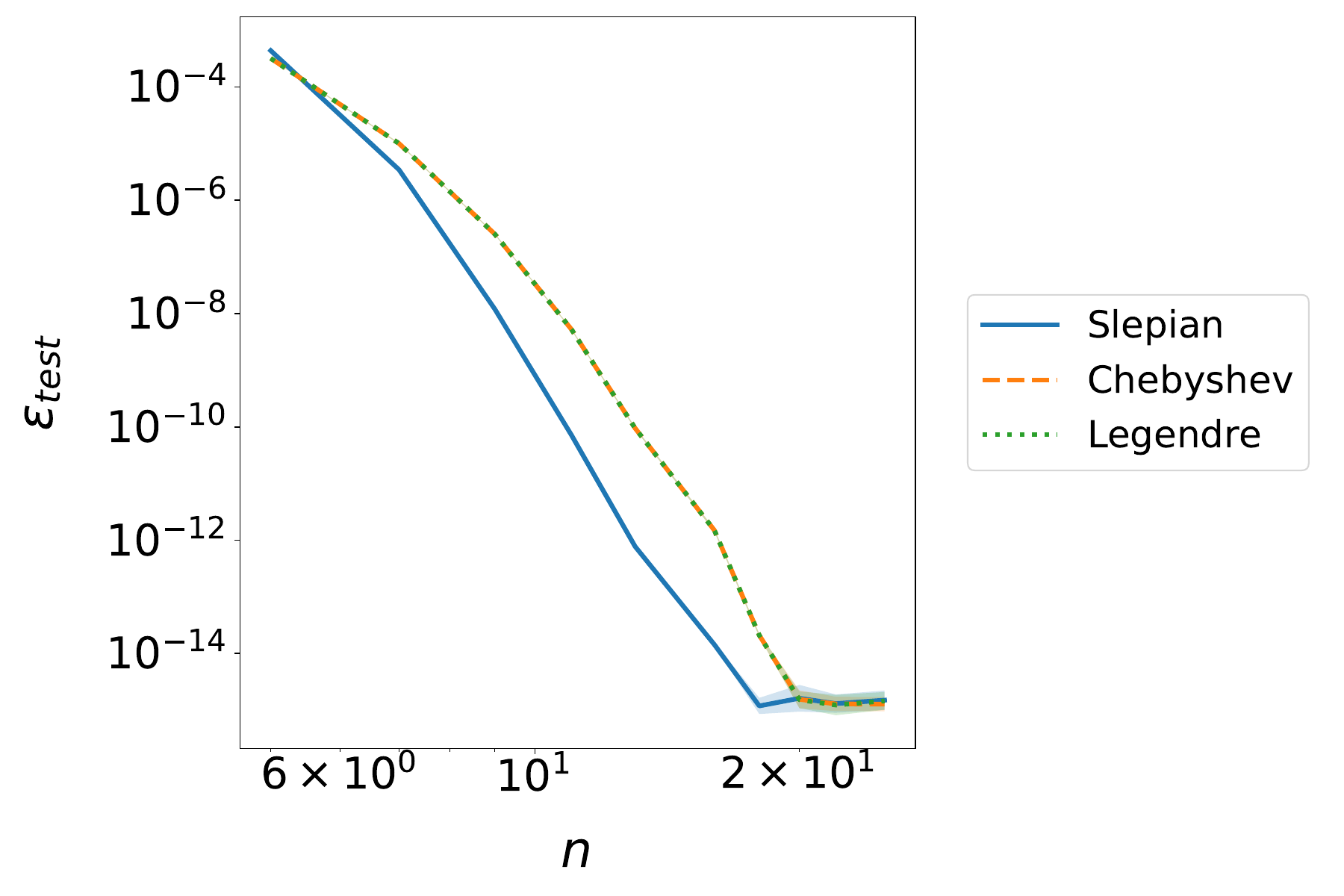}
  \caption{Comparison of Least Squares approximations using different bases for the function $g_1(x) = e^{-\pi x^2}$
  from $m$ random samples. The left panel shows results using $n=10$ basis functions and varying the number of samples, while the right panel shows results when the number of samples is fixed at $m=1000$ and the order parameter $n$ changes. Note that the results for Chebyshev and Legendre polynomials can be hard to distinguish because they are almost identical.}
  \label{fig:1DBases} 
\end{figure}

\begin{figure}[t!]
    \centering
    \includegraphics[width = 0.40\textwidth]{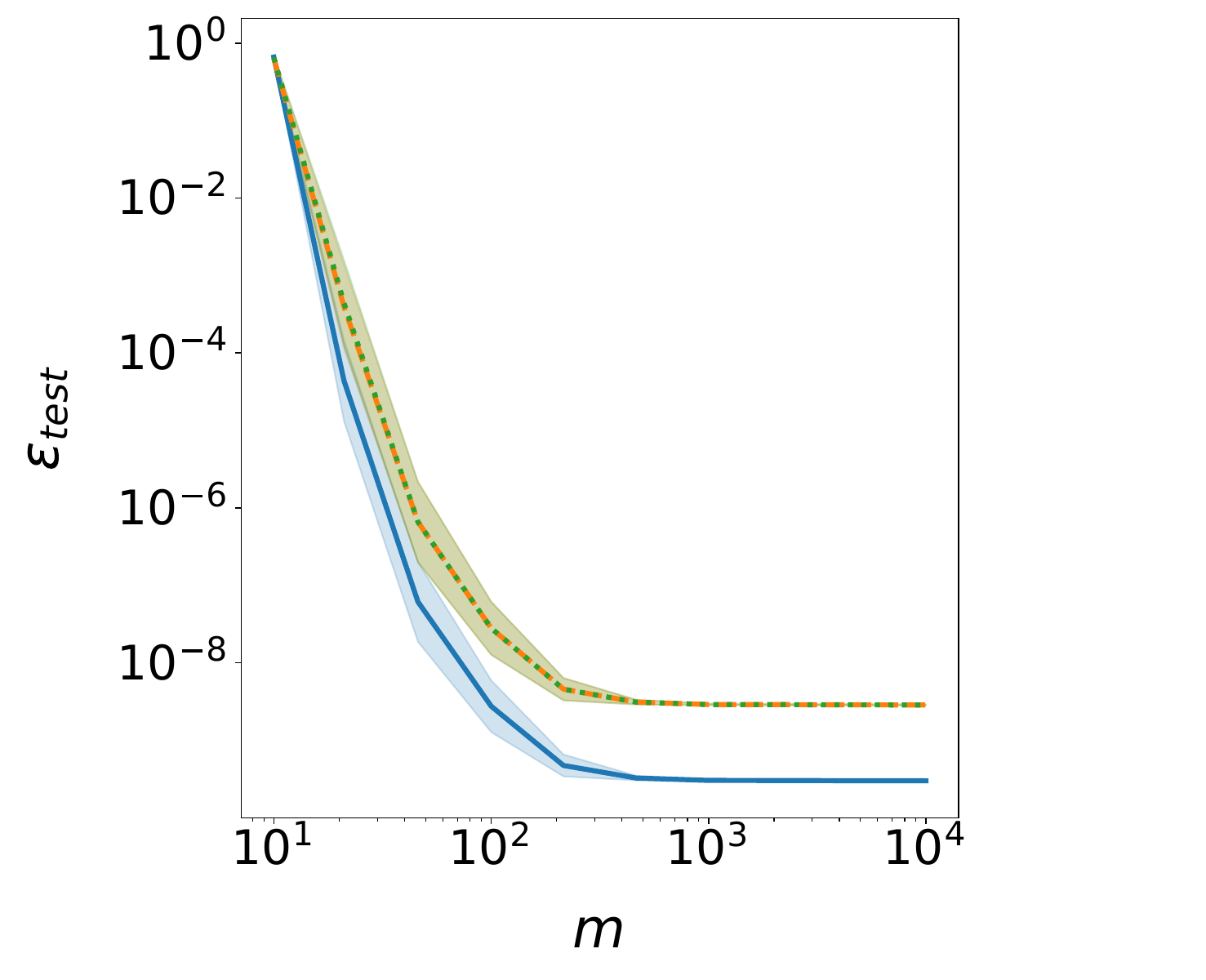} 
    \includegraphics[width = 0.49\textwidth]{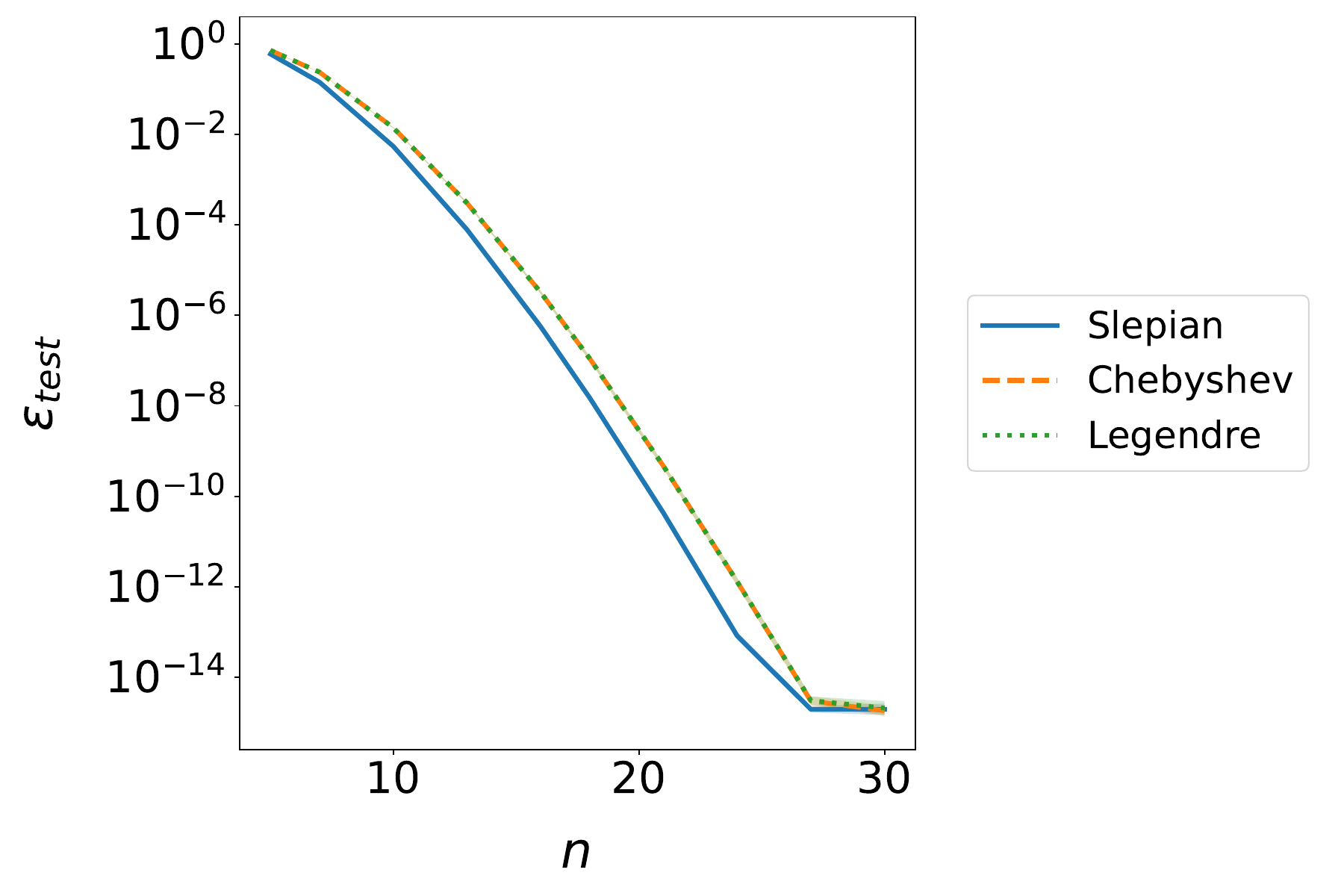}
  \caption{Comparison of least squares approximations using different bases for the function $g_2(x) = e^{ 4i \pi x}$ 
  from $m$ random samples. The left panel shows results using $n=20$ basis functions and varying the number of samples, while the right panel shows results when the number of samples is fixed at $m=1000$ and the order parameter $n$ changes. Note that the results for Chebyshev and Legendre polynomials can be hard to distinguish because they are almost identical.}
  \label{fig:1DBasesComplex} 
\end{figure}

\paragraph{One-dimensional results.}
Figures \ref{fig:1DBases} and \ref{fig:1DBasesComplex} display the RMSEs for the functions $g_1$ and $g_2$, respectively.
For moderate order $n$, using the Slepian basis for Least Squares reconstruction outperforms the Chebyshev and Legendre bases. (Note, when $d=1$ we simply have $\Lambda^{\text{HC}}_{n-1} = \{0, 1, \ldots n-1\}$, hence $n$ is also the number of basis functions.) 
In particular, when the number of samples is large enough, the RMSE for the Slepian basis is between one and two orders of magnitude smaller than for the two polynomial bases.
When the order $n=10$ is fixed, the RMSE converges as the number of samples reaches $m=1000$  and remains stable. At $m=1000$ samples, the RMSE decreases as the number of parameters increases. 
Reconstructions with the Chebyshev and Legendre bases yield very similar, though not identical, RMSEs, which are difficult to distinguish on the logarithmic scale. 


\paragraph{Two-dimensional results.} Figure \ref{fig:2Dtaketwo} shows that the Slepian basis is superior for reconstructing $g_3$. 
\begin{figure}[t!]
    \centering
    \includegraphics[width = 0.40\textwidth]{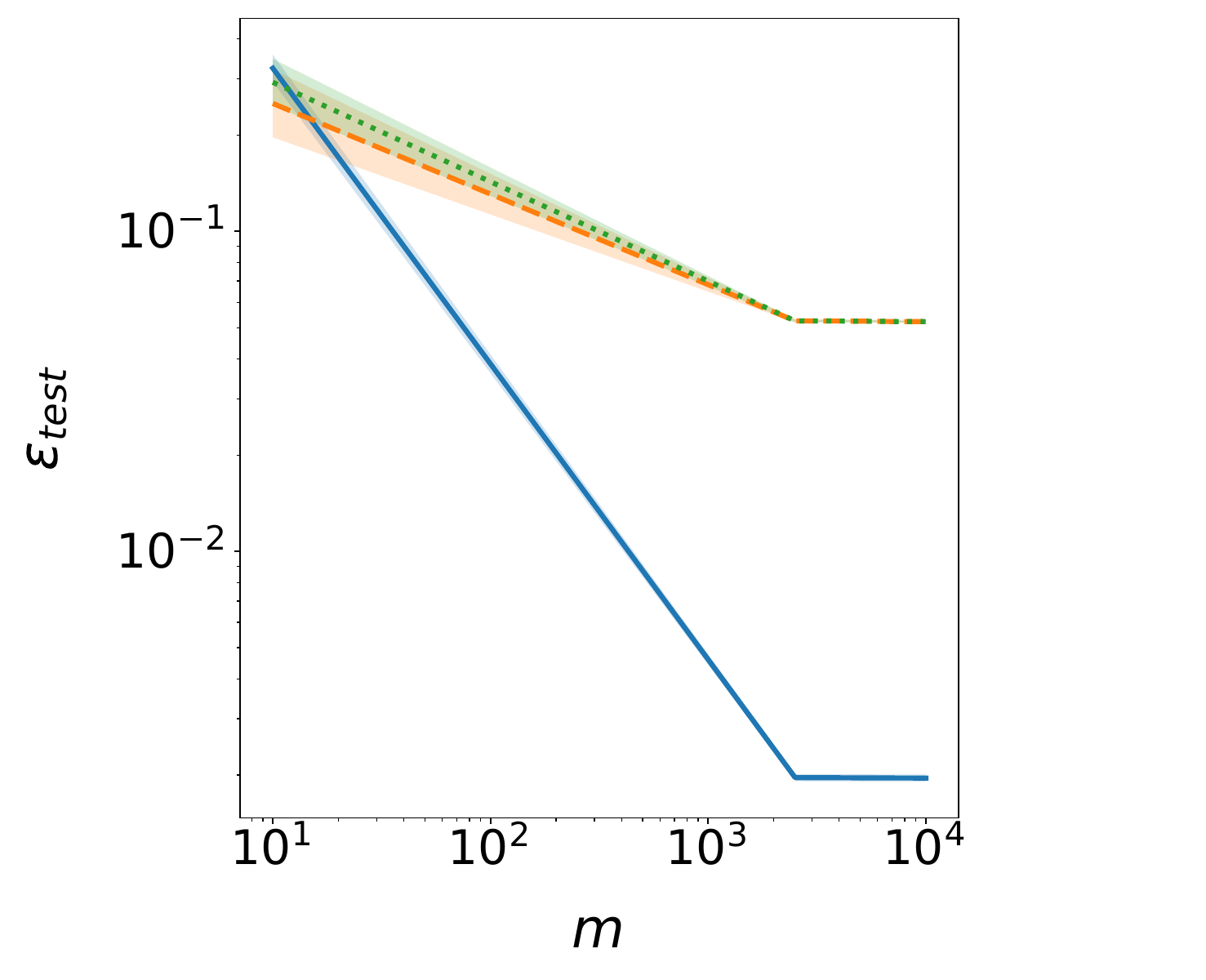} 
    \includegraphics[width = 0.49\textwidth]{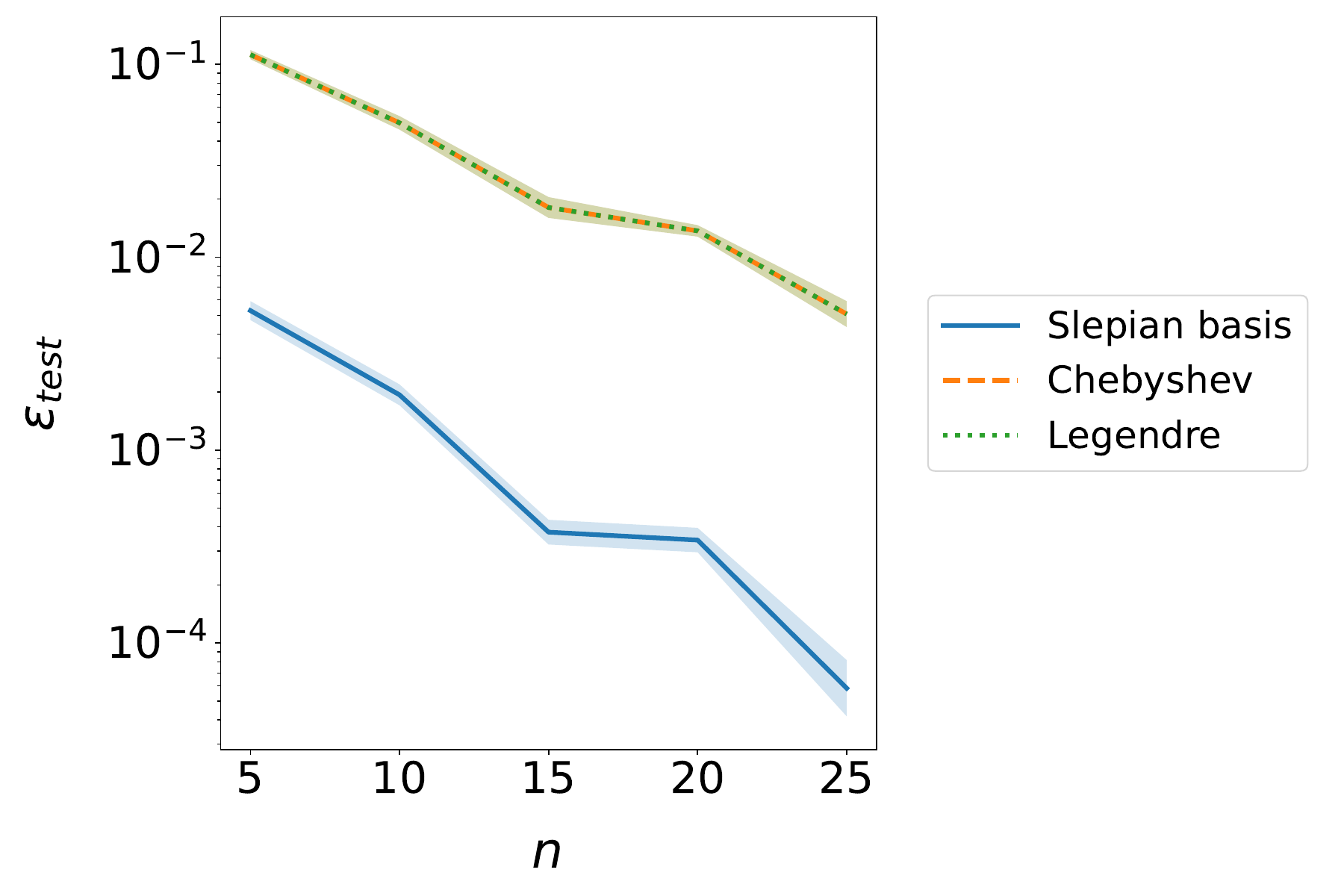}
  \caption{Comparison of Least Squares approximations using different bases for the function $g_3(x) = e^{-\pi(x^2+y^2+0.2xy)}$ 
  from random samples. The left panel shows results using a hyperbolic cross $\Lambda^{\text{HC}}_{n-1}$ with $n=10$ and varying the number of samples, while the right panel shows results when the number of samples is fixed at $m=1000$ and the order parameter $n$ changes. Note that the results for Chebyshev and Legendre polynomials can be hard to distinguish because they are almost identical.}
  \label{fig:2Dtaketwo} 
\end{figure}
In general, the RMSEs obtained in this case are several orders of magnitude higher than those in the one-dimensional case, reflecting the increased complexity of the approximation problem. In the left-hand plot, we consider $n=10$, which entails a large number of tensorized Slepian functions within the hyperbolic cross $\Lambda^{\text{HC}}_{n-1}$ (for a bound on its cardinality in all dimensions, see for example Lemma~\ref{lem:HCsize}). As in the one-dimensional experiments, the RMSE converges as the number of samples increases, for both the Slepian and the polynomial bases. Likewise, the right-hand plot, where the number of samples is fixed at $m=1000$, shows that the RMSE decreases as the number of parameters increases. 
In general, the benefits of Slepian over orthogonal functions appear more evident for $g_3$ and $g_1$ than $g_2$.

\subsection{Comparisons of different initializations for Deep Learning}
\label{sec:initializations}

Additionally, we conduct a set of experiments to compare different neural network initialization strategies. The main purpose of this section is to show that a novel Slepian-based initialization inspired by Theorem~\ref{thm:ApproxSpan} can lead to more accurate approximations with respect to other standard random initializations strategies for Deep Learning. In particular, we train networks with architecture $(1, 1000, 1000, 10, 1)$, where, in general, an architecture $(n_0, \ldots, n_{L+1})$ denotes a feedforward neural network with input layer of dimension $n_0$, output layer of dimension $n_{L+1}$ and $L$ hidden layers of size $n_1, \ldots, n_L$, respectively.
In our experiments, we consider the function  
\begin{equation}
\label{eq:def_f1}
f_1(x) = \cos(10x) e^{-\pi x^2}, \qquad \forall x \in [-1,1].
\end{equation}
This one-dimensional function, defined as the product of a trigonometric function and a Gaussian, was considered in \citep{moore2004prolate} as a simple example in which the Slepian basis achieves better approximation accuracy than the standard Fourier basis. In all cases, networks are trained using Adam \citep{diederik2015adam} with an exponential decay rate of $0.85$.

We start by describing the proposed Slepian-based initialization strategy. Inspired by the decomposition of a frequency-localized function \(f\) with respect to the Slepian basis, 
\begin{equation}\label{eq:ApproxSum}
f(x) \approx \sum_{j=0}^{J} \alpha_j \varphi_{\mathsf{w},j},
\end{equation}  
and by the construction of Theorem~\ref{thm:ApproxSpan} (see Section~\ref{sec:prelimformainthm2}), we first approximate each Slepian basis function \(\varphi_{\mathsf{w},j}\) using distinct neural networks \(\tilde{\varphi}_{\mathsf{w},j}\) with architecture $(1, 100, 100, 1)$.  
Each network is trained on $10{,}000$ samples \((\vec{x}_i, \varphi_{\mathsf{w},j}(\vec{x}_i)), i=1,\dots,10000\), where \(\vec{x}_i \in [-1,1]^d\) are drawn uniformly at random and for 1000 epochs.  
We train networks for the first ten Slepian basis functions (\(j = 0,1,\dots,9\)).  
The test errors for these networks are reported in Table~\ref{tab:1}.  
\begin{table}[t!]
\centering
\begin{tabular}{|c|c|c||c|c|c|}
\hline
\textbf{Index} $j$ & \textbf{Training error} & \textbf{Test error} & \textbf{Index} $j$ & \textbf{Training error} & \textbf{Test error}\\
\hline
0 & $9.320 \times 10^{-8}$ & $8.837 \times 10^{-8}$ & 5 & $1.571 \times 10^{-4}$ & $1.960 \times 10^{-4}$ \\
1 & $4.365 \times 10^{-7}$ & $3.578 \times 10^{-7}$ & 6 & $3.448 \times 10^{-3}$ & $3.270 \times 10^{-3}$ \\
2 & $8.744 \times 10^{-7}$ & $7.630 \times 10^{-7}$ & 7 & $2.218 \times 10^{-4}$ & $2.430 \times 10^{-4}$ \\
3 & $2.680 \times 10^{-7}$ & $2.766 \times 10^{-7}$ & 8 & $1.795 \times 10^{-5}$ & $1.845 \times 10^{-5}$ \\
4 & $1.877 \times 10^{-5}$ & $1.921 \times 10^{-5}$ & 9 & $1.307 \times 10^{-4}$ & $1.138 \times 10^{-4}$ \\
\hline
\end{tabular}
\caption{Test and training errors for networks $\tilde{\varphi}_{\mathsf{w},j}$ with architecture $(1,100,100,1)$ approximating the Slepian basis functions $\varphi_{\mathsf{w},j}$.}
\label{tab:1}
\end{table}
We observe that the approximation error is small for low-index basis functions and increases with \(j\); in particular, the functions with indices \(j=6,7\) yield the largest errors. 
Once each Slepian basis function is approximated, we combine the trained networks according to Equation~\eqref{eq:ApproxSum} and determine the coefficients \(\alpha_j\) via least-squares fitting using $10{,}000$ new random samples \((\vec{x}_i, f(\vec{x}_i))\).  
The networks and weights are then concatenated into a single network \(\tilde{\Phi}_1\) of architecture $(1,100 \times 10,100 \times 10, 10, 1) = (1, 1000, 1000, 10, 1)$ that realizes the corresponding linear combination.  
This construction is analogous to the neural network developed in Section~\ref{sec:ProofTheo3} for the proof of Theorem~\ref{thm:ApproxSpan}.  
The resulting network \(\Phi_1\) is subsequently used to initialize a network with identical architecture, denoted \(\tilde{\Phi}_1\). Note that the most expensive part of this construction is training the networks $\tilde{\varphi}_{\mathsf{w},j}$. However, this computation is independent of the particular function $f$ that one wants to approximate. Hence, it can be considered as an offline pre-computing step. On the other hand, the least square fitting step is function dependent but can be solved quite efficiently.

In parallel, we train three additional networks with the same $(1, 1000, 1000, 10, 1)$ architecture but different initialization schemes, i.e.,
\begin{itemize}
    \item \(\Phi_2\) with standard normal initialization (mean 0, standard deviation 0.1),  
    \item \(\Phi_3\) with He initialization \cite{he2015delving},  
    \item \(\Phi_4\) with Glorot initialization \cite{glorot2010understanding}.  
\end{itemize}
The corresponding training and test errors for \(\Phi_1\), \(\Phi_2\), \(\Phi_3\), and \(\Phi_4\) are shown in Figure~\ref{fig:TrainingVsTest}.  
\begin{figure}[t!]
    \centering
    \includegraphics[width = 0.49\textwidth]{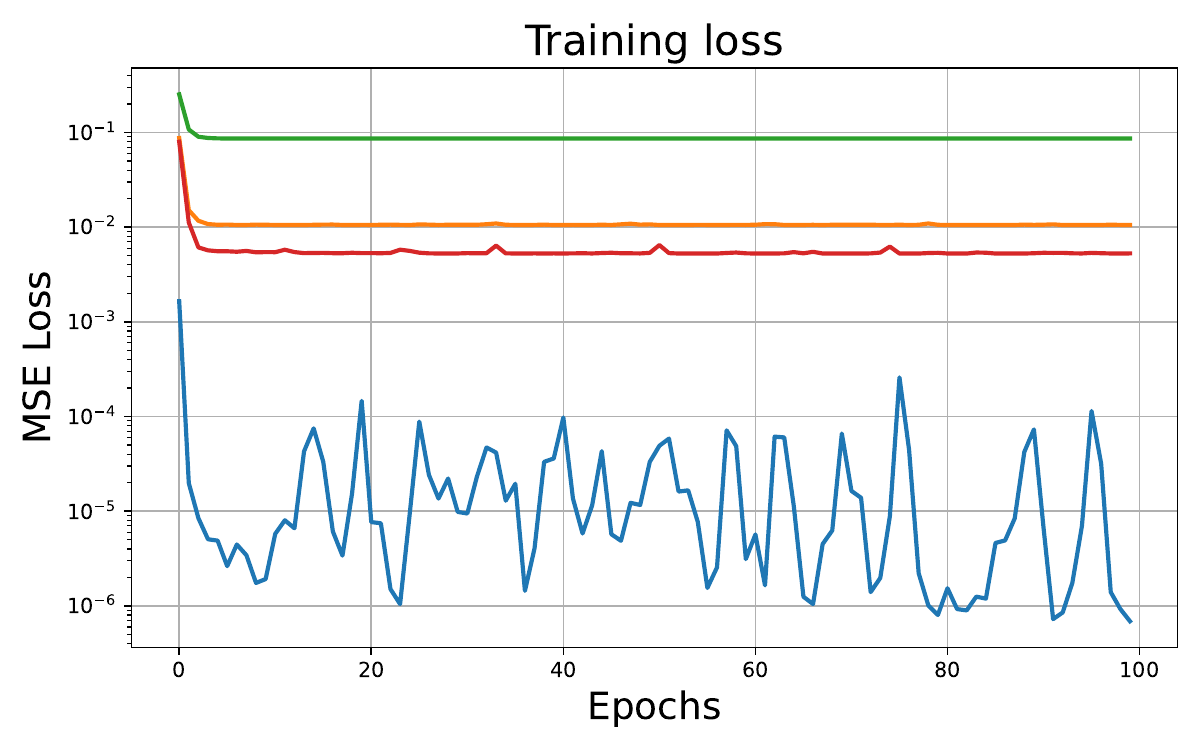}
    \includegraphics[width = 0.49\textwidth]{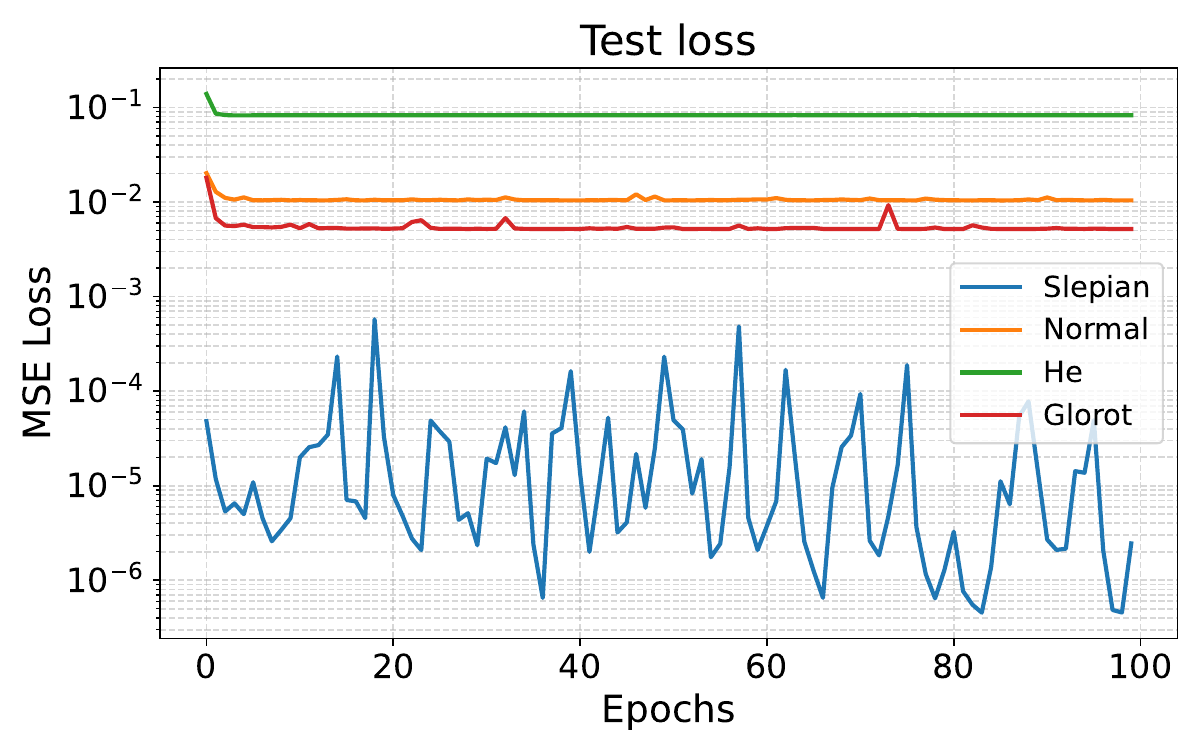} 
  \caption{Training and test errors for neural networks with architecture $(1,1000,1000,10,1)$ approximating the function $f_1(x) = \cos(10x)e^{-\pi x^2}$ using different initialization strategies. 
    The left panel shows the training error and the right panel shows the test error.}
  \label{fig:TrainingVsTest} 
\end{figure}
Our proposed initialization (\(\Phi_1\)) outperforms all others, achieving errors on the order of \(10^{-6}\) for both training and test sets (considering the best epoch among those shown). Glorot initialization yields errors around \(10^{-2}\), while the standard normal and He initializations result in errors of approximately \(10^{-1}\).  
Across all initialization methods, training and test errors converge within a few epochs.

\subsection{Least Squares vs.\ Deep Learning}
We assess the performance of Least Squares and Deep Learning for the reconstruction of the one-dimensional function $f_1$ defined in \eqref{eq:def_f1}
and the two-dimensional function 
\[
f_2(x,y) := \cos(0.2x)\cos(0.2y)e^{-\pi (x^2 + y^2)}, \qquad \forall (x,y)\in [-1,1]^2.
\]
The two-dimensional function $f_2$ generalizes the one-dimensional function $f_1$ 
while reducing the cosine frequency from 10 to 0.2. Both functions are frequency-localized since most of their energy is concentrated within a compact interval in the frequency domain. 

For each function $f$, we consider the sample size $m$ taking values up to $10000$.
For a set of $m$ training points, we also consider a test set $\{x_j^{\text{test}}\}_{j=1}^{m_{\text{test}}}$ of size $m_{\text{test}} = 0.2m $, sampled independently according to the uniform distribution on $[-1,1]^d$. 

\paragraph{Neural network architecture and training.} Following standard deep learning practice, we consider fully trained NNs. Here, we deviate from the theoretical setting of Theorem~\ref{thm:ApproxSpan}, where only the last layer is trained. NN architectures are built according to the ratio $r = L/N = 0.1$, where $L$ represents the number of hidden layers and $N$ the number of neurons per layer. Specifically, we conduct experiments for $L = 1,\dots, 10$, with $r = 0.1$, a choice empirically supported by \citep{adcock2021gap}. Weights and biases are initialized from a normal distribution with mean $0$ and standard deviation $0.1$. 
Training is performed using Adam \citep{diederik2015adam} with an exponential decay rate of $0.85$ for $150$ epochs. To ensure the robustness and reliability of the results, each experiment is repeated $20$ times.

\paragraph{One-dimensional results.} Figure~\ref{fig:1D} illustrates the results of our first experiment on approximating the one-dimensional function $f_1$.
\begin{figure}[t!]
    \centering
    \includegraphics[width = 0.49\textwidth]{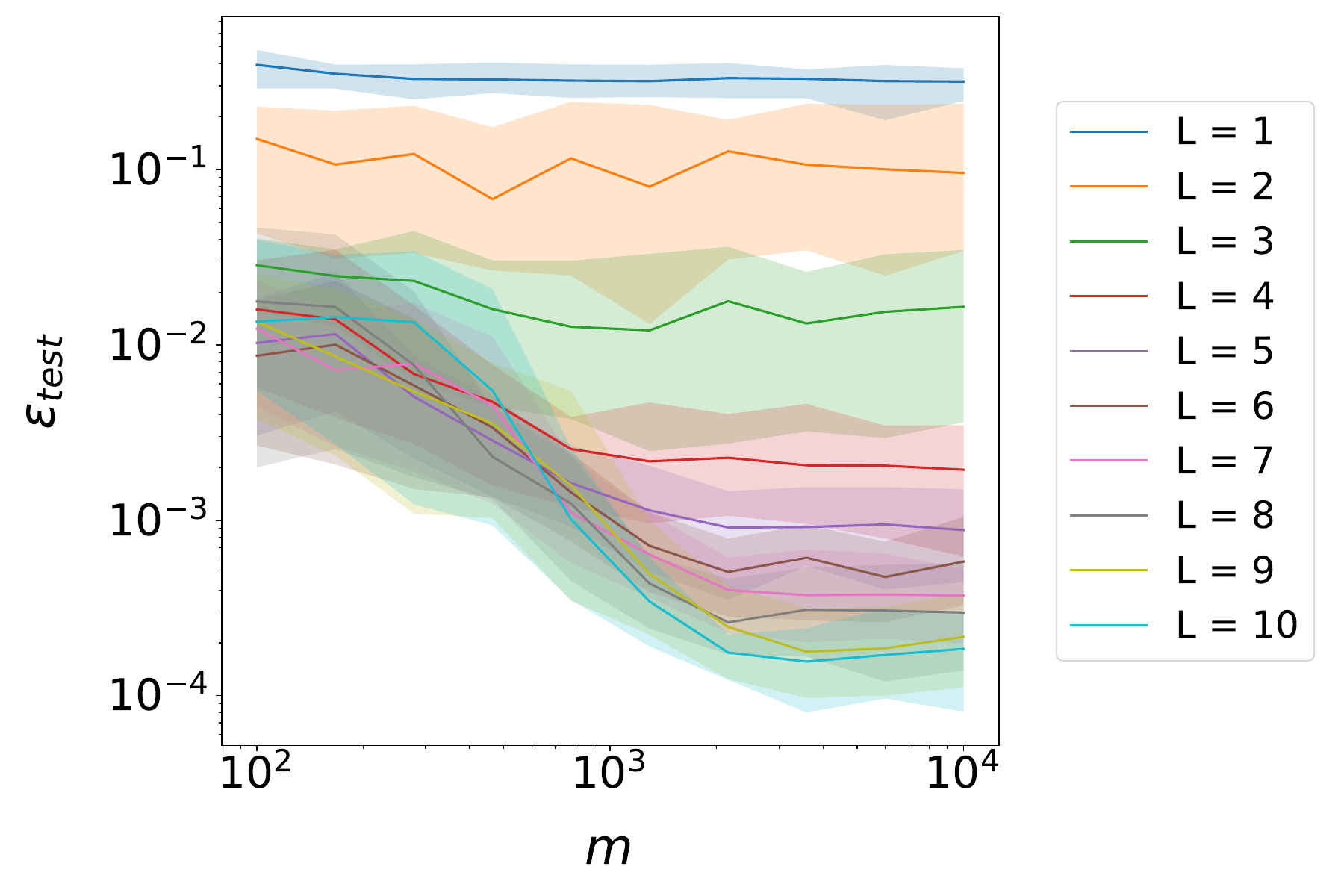} 
    \includegraphics[width = 0.49\textwidth]{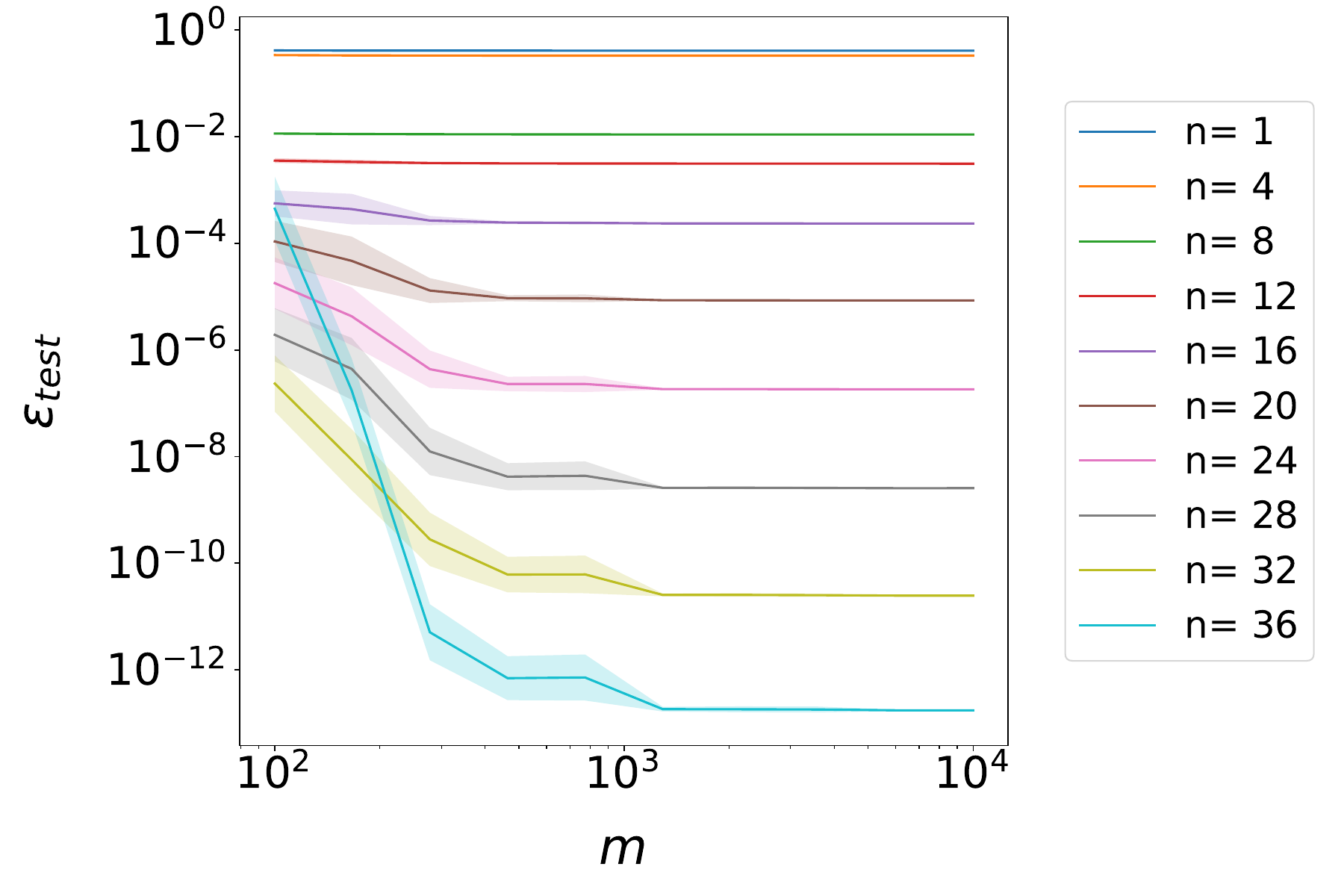}
  \caption{Comparison of reconstruction methods for $f_1(x) = \cos(10x)e^{-\pi x^2}$
  from samples. The left panel shows results using Deep Learning, while the right panel shows results using Least Squares. Here, $L$ denotes the number of layers in each NN, and $n$ denotes the order of the hyperbolic cross $\Lambda^{\rm HC}_{n-1}$.}
  \label{fig:1D} 
\end{figure}
In accordance with our theoretical guarantees, both methods successfully approximate this frequency-localized function, with the RMSE converging as the number of sampling points increases.
Moreover, the number of parameters in each reconstruction model significantly affects the RMSE. For Least Squares, increasing the number of Slepian basis functions (corresponding to increasing the order $n$) generally lowers test errors, with the only recorded exception being $n=36$. 
Similarly, for Deep Learning, increasing the number of network layers $L$ 
results in a noticeable error reduction. 
However, note that Least Squares typically requires fewer samples as well as significantly fewer parameters to achieve high accuracy. 
The best accuracy achieved by Least Squares (RMSE $\approx 10^{-12}$) is several orders of magnitude higher than that by Deep Learning (RMSE $\approx 10^{-4}$). In summary, in one dimension, Least Squares is much more effective than Deep Learning for the approximation of $f_1$.

\paragraph{Two-dimensional results.} 
Figure~\ref{fig:2D} illustrates the results of our second experiment on approximating the two-dimensional function $f_2$.
\begin{figure}[t]
    \centering
    \includegraphics[width = 0.49\textwidth]{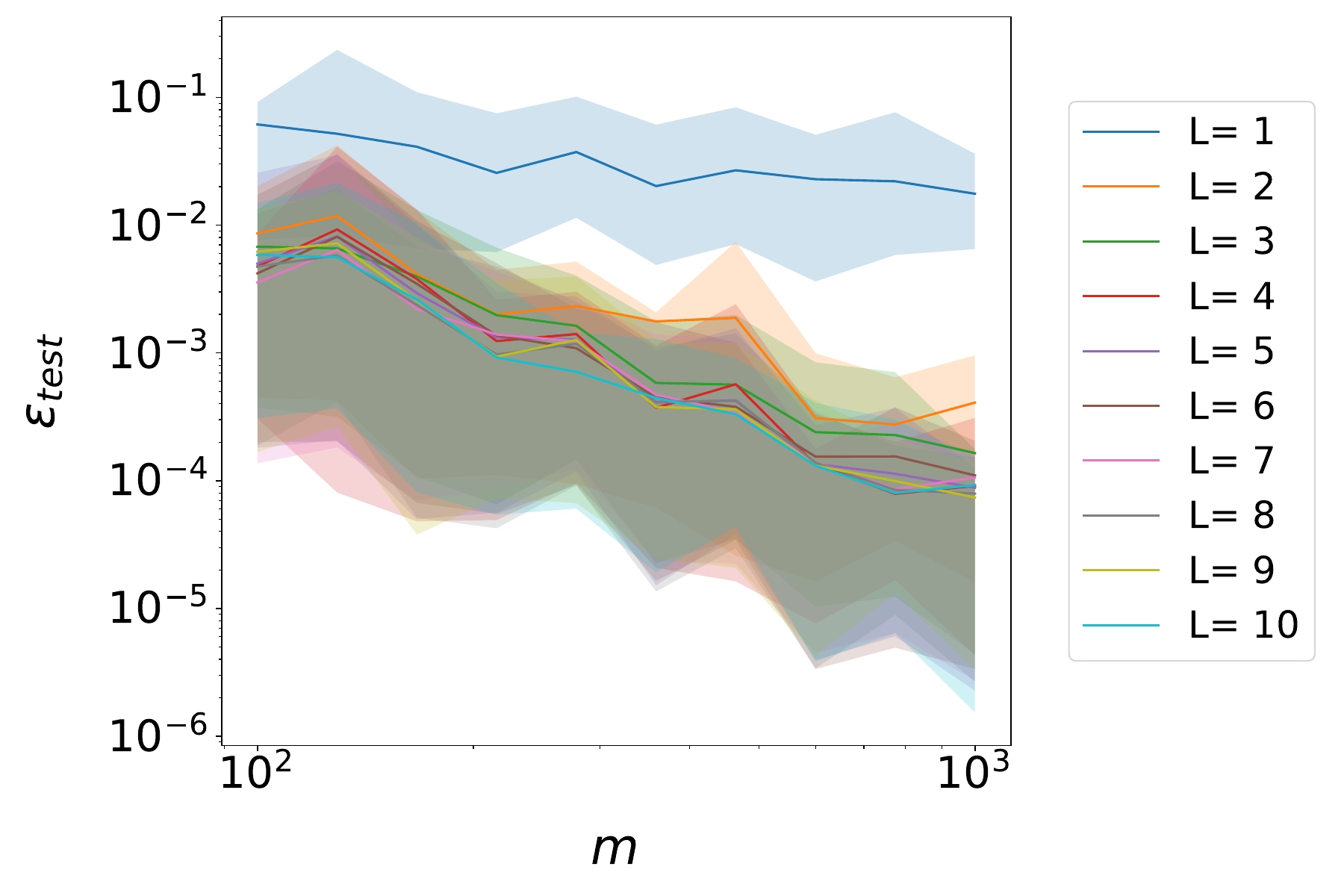}
    \includegraphics[width = 0.49\textwidth]{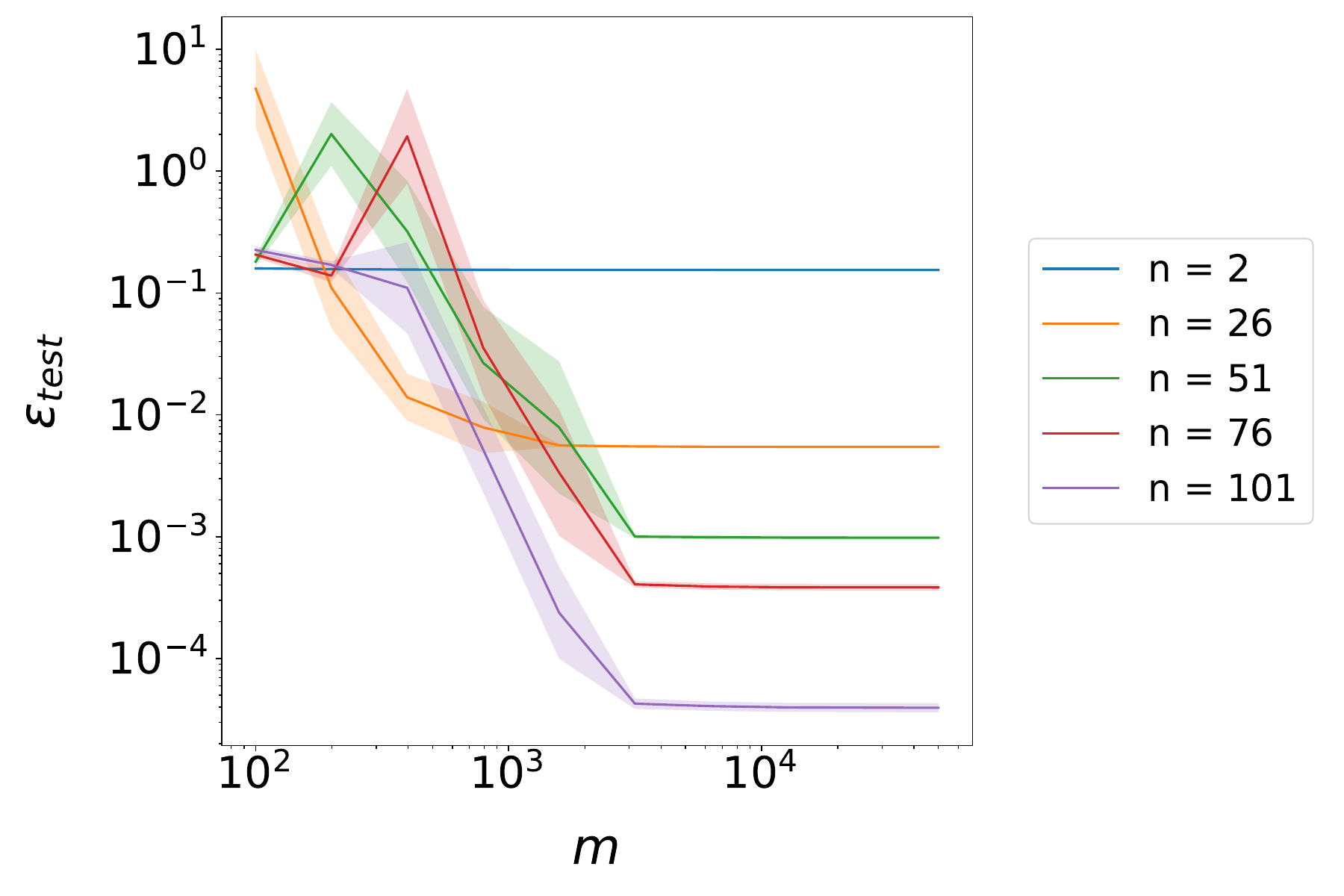} 
    \caption{Comparison of reconstruction methods for $f_2(x) = \cos(0.2x)\cos(0.2y)e^{-\pi (x^2+y^2)}$
    from samples. The left panel shows results using Deep Learning, while the right panel shows results using Least Squares. Here, $L$ denotes the number of layers in each NN, and $n$ denotes the order of the hyperbolic cross $\Lambda^{\rm HC}_{n-1}$.}
  \label{fig:2D} 
\end{figure}
Both Least Squares and Deep Learning successfully approximate this frequency-localized function, in line with our theory. 
As in the one-dimensional case, we observe RMSE convergence as a function of the number of samples $m$.
However, there are notable differences between the one- and two-dimensional cases. 
First, the best accuracy achieved by both methods is now much more similar. Both Least Squares and Deep Learning achieve a best mean RMSE of around $10^{-4}$. 
Interestingly, in some experiments, Deep Learning can attain a lower RMSE, as indicated by the size of the variance regions (the shaded regions in Figure~\ref{fig:2D}, left).
Hence, the accuracy gap between Least Squares and Deep Learning is much less pronounced in dimension two than in dimension one. Second, although the advantage of adding parameters to the model (for both Least Squares and Deep Learning) is still there, we observe some intriguing numerical phenomena. For Least Squares, considering a larger Slepian basis (i.e., larger order $n$) might not lead to a better accuracy (see the curve corresponding to $n=26$). Meanwhile, for Deep Learning, the benefit of increasing $L$ for accuracy becomes limited after $L\geq 6$.
In summary, dimensionality appears to play a crucial role when comparing Least Squares and Deep Learning. While a comprehensive numerical study is beyond the scope of this paper, our preliminary results suggest that 
the performance difference between the two reconstruction methods could increase in dimension three and higher.

\section{Proofs}
\label{sec:proofs}

To facilitate the proofs of Theorem~\ref{thm:leastsquaresampling} and Theorem~\ref{thm:ApproxSpan}, detailed in Section~\ref{sec:leastsquaresproof} and Section~\ref{sec:ProofTheo3} respectively, we first establish several key preliminary results in Section~\ref{sec:prelimformainthms} and Section~\ref{sec:prelimformainthm2}. 
Both theorems necessitate an understanding of the uniform norm of a normalized Slepian basis function within the centered $d$-dimensional unit cube. 
Additionally, Theorem~\ref{thm:ApproxSpan} relies on the polynomial approximability of neural networks and fundamental network operations, which are outlined in Appendix~\ref{appx:basicnetworks} for clarity and reference.

\subsection{A preliminary result for the main theorems} \label{sec:prelimformainthms}

In what follows, we recall from \eqref{eqdef:gammaw} the simplified notation $\gamma(\mathsf{w}) = \lceil \log_2(24\mathsf{w}) \rceil$.

\begin{proposition} \label{prop:gammaw} For any $\mathsf{w}\geq 1$, let $\{\varphi_{\mathsf{w},j}\}_{j\in\N_0}$ be the Slepian basis in $L_{ {\bf u}}^2([-1,1])$. 
Then we have
\begin{equation} \label{eq:varphijall}
    \|\varphi_{\mathsf{w},j}\|^2_{L^{\infty}([-1,1])}\leq \mathbbm{1}_{\{j=0\}} \Big(\frac{4\mathsf{w}}{\pi}\Big) + \mathbbm{1}_{\{j\in\mathbb{N}\}} (j+1)^{\gamma(\mathsf{w})}.
\end{equation}
\end{proposition}

In preparation, we present a key a priori result, with its proof deferred to the end of this subsection.

\begin{proposition} \label{prop:Phi_j} Let $\mathsf{w}\geq 1$. 
Let $j^{\star}(\mathsf{w}):= \lfloor 4\mathsf{w} \rfloor -1$. 
Then it holds for all $j\in\N_0$ that
\begin{equation*} 
    \|\varphi_{\mathsf{w},j}\|_{L^{\infty}([-1,1])} \leq 
    \begin{cases} 
        2\sqrt{\frac{\mathsf{w}}{\pi}} &\text{ if } j\leq j^{\star}(\mathsf{w}),\\
        {}\\
        \frac{12}{5} \sqrt{j+1} &\text{ if } j> j^{\star}(\mathsf{w}).
    \end{cases}
\end{equation*}
\end{proposition}

\begin{remark} \label{rem:firsteigen}
The distinction between Proposition~\ref{prop:Phi_j} and Proposition~\ref{prop:gammaw} lies in the latter providing a more encompassing upper bound for $\|\varphi_{\mathsf{w},j}\|_{L^{\infty}([-1,1])}$ as a function of the index $j$, without relying on a critical division $j^{\star}(\mathsf{w})$ tied to the bandwidth $\mathsf{w}$.
Such a broader bound is needed in the proof of Theorem~\ref{thm:leastsquaresampling}. 
A frequently cited emblematic estimate indicates that, for all $j\in\mathbb{N}_0$ 
\begin{equation} \label{eq:misinfo}
    \|\varphi_{\mathsf{w},j}\|_{L^{\infty}([-1,1])}\leq C\sqrt{j};
\end{equation}
see for example \citep[Remark~2.4]{shkolnisky2006approximation}.
However, \eqref{eq:misinfo} holds only asymptotically (and a rectified version of \eqref{eq:misinfo} is given in \citep[Theorem~2.7]{xiao2001prolate}). Specifically, for sufficiently large values of $j$, beyond the critical index $j^{\star}(\mathsf{w})$, the global maximum of $|\varphi_{\mathsf{w},j}|$ on $[-1,1]$ shifts to the boundary of the interval, becoming equal to $|\varphi_{\mathsf{w},j}(1)|$, which is of order $\mathcal{O}(\sqrt{j})$, as numerically demonstrated in \citep[Theorem~8.4]{xiao2001prolate}.
This shifting behavior, as discussed in both \citep[Theorem~3.38]{xiao2001prolate}, \citep[Theorem~3.1]{bonami2014uniform}, occurs when $\mathsf{w}^2/\chi_{\mathsf{w},j} \leq 1$. The second theorem will be later utilized in our proof of Proposition~\ref{prop:Phi_j}.
For low indices $j$, \citep[Remark~3.1]{bonami2014uniform} implies that $\|\varphi_{\mathsf{w},j}\|_{L^{\infty}([-1,1])}\leq C\sqrt{\mathsf{w}}$.
Notably, $\varphi_{\mathsf{w},0}$ is even, positive on the interval $(-1,1)$ and attains its global maximum at the origin. Yet, for low indices $j$, including $j=0$, the behavior of $\varphi_{\mathsf{w},j}(0)$ is poorly understood.
\end{remark}

\begin{proof}[Proof of Proposition~\ref{prop:gammaw}]
Leveraging Proposition~\ref{prop:Phi_j}, the claim $\|\varphi_{\mathsf{w},0}\|_{L^{\infty}([-1,1])}^2 \leq 4\mathsf{w}/\pi$ becomes clear, leaving only the proof of \eqref{eq:varphijall} to be addressed for $j\in\N$.
Toward this end, we write
\begin{equation} \label{eq:varphLinftybd}
    \|\varphi_{\mathsf{w},j}\|^2_{L^{\infty}([-1,1])} \leq \Big(\frac{4\mathsf{w}}{\pi}\Big) \mathbbm{1}_{\{1\leq j\leq j^{\star}(\mathsf{w})\}} + \Big(\frac{12}{5}\Big)^2 (j+1) \mathbbm{1}_{\{j>j^{\star}(\mathsf{w})\}}.
\end{equation}
Let $\tilde{c} = (\frac{12}{5})^2$, and let
\begin{equation*} 
    h(x) := \frac{\log\big(24\mathsf{w}\mathbbm{1}_{\{x\leq j^{\star}(\mathsf{w})\}} + \tilde{c}(x+1) \mathbbm{1}_{\{x>j^{\star}(\mathsf{w})\}} \big)}{\log(x+1)}.
\end{equation*}
We claim the following.
\begin{claim} \label{claim}
    \qquad $h(x)$ \text{ is nonincreasing on } $[1,\infty)$.
\end{claim}
Observe that the claim is evident for either the case $1\leq x\leq j^{\star}(\mathsf{w})$ or for the case $x > j^{\star}(\mathsf{w})$. 
Hence it suffices to show, if $1\leq x_1\leq j^{\star}(\mathsf{w})< x_2\leq j^{\star}(\mathsf{w})+1$, then $h(x_2)\leq h(x_1)$, or:
\begin{equation*}
    \frac{\log(24\mathsf{w})}{\log(x_1+1)} \geq \frac{\log(\tilde{c}(x_2+1))}{\log(x_2+1)}.
\end{equation*}
Since the right-hand side is a continuous and nonincreasing function, it also suffices if we show
\begin{equation} \label{eq:resolvent}
    \frac{\log(24\mathsf{w})}{\log(x_1+1)} \geq \lim_{x_2\uparrow j^{\star}(\mathsf{w})} \frac{\log(\tilde{c}(x_2+1))}{\log(x_2+1)} = \frac{\log(\tilde{c}(j^{\star}(\mathsf{w})+1))}{\log(j^{\star}(\mathsf{w})+1)} = \frac{\log(\tilde{c}\lfloor 4\mathsf{w}\rfloor)}{\log(\lfloor 4\mathsf{w}\rfloor)}.
\end{equation}
Note that $\log(\lfloor 4\mathsf{w}\rfloor)=\log(j^{\star}(\mathsf{w})+1)\geq \log(x_1+1)$, as $j^{\star}(\mathsf{w})\geq x_1$. Therefore \eqref{eq:resolvent} follows, once $\log(24\mathsf{w})\geq \log (\tilde{c}\lfloor 4\mathsf{w}\rfloor)$ holds.
However, this is immediate, since
\begin{equation*}
    \log(24\mathsf{w}) \geq \log (23.04\mathsf{w}) = \log(\tilde{c}4\mathsf{w}) \geq \log(\tilde{c}\lfloor 4\mathsf{w}\rfloor).
\end{equation*}
Thus, Claim~\ref{claim} holds, and we get:
\begin{equation} \label{eq:resolvent1}
    \frac{\log\big(24\mathsf{w}\mathbbm{1}_{\{x\leq j^{\star}(\mathsf{w})\}} + \tilde{c}(x+1) \mathbbm{1}_{\{x>j^{\star}(\mathsf{w})\}} \big)}{\log(x+1)} \leq \log_2(24\mathsf{w}) \leq \gamma(\mathsf{w}), \qquad \forall x\geq 1.
\end{equation}
Combining \eqref{eq:varphLinftybd}, \eqref{eq:resolvent1}, we infer that
\begin{equation*} \label{eq:caseN}
    \|\varphi_{\mathsf{w},j}\|^2_{L^{\infty}([-1,1])} \leq 24\mathsf{w}\mathbbm{1}_{\{1\leq j\leq j^{\star}(\mathsf{w})\}} + \tilde{c}(j+1)\mathbbm{1}_{\{j>j^{\star}(\mathsf{w})\}} \leq (j+1)^{\gamma(\mathsf{w})},
\end{equation*}
as wanted.
\end{proof}

To prove Proposition~\ref{prop:Phi_j}, we rely on the following two well-known lemmas. The first, Lemma~\ref{lem:spliteigenrecall} (see, for instance, \citep{landau1965eigenvalue, landau1993density}) concerns the spectral gap in the eigenvalues $\mu_{\mathsf{w},j}$ \eqref{eq:muj}, while the second, Lemma~\ref{lem:higheigncase}, pertains to the magnitude of $\varphi_{\mathsf{w},j}$ for sufficiently large index $j$.
The proof of Proposition~\ref{prop:Phi_j} will come after the presentation of these lemmas.

\begin{lemma} \label{lem:spliteigenrecall} Let $\mathcal{Q}_{\mathsf{w}}$ be defined in \eqref{eqdef:Qw}. Then the eigenvalues $\mu_{\mathsf{w},j}$ of $\mathcal{Q}_{\mathsf{w}}$ satisfy
\begin{equation} \label{crucialupperbound} 
    \mu_{\mathsf{w}, \lfloor 4\mathsf{w}\rfloor-1}\geq \frac{1}{2} \geq \mu_{\mathsf{w}, \lceil 4\mathsf{w}\rceil}. 
\end{equation}
\end{lemma}

Lemma~\ref{lem:higheigncase} below is a version of an estimate commonly referenced in the literature but seldom proved. 
For the convenience of the reader, we offer a brief proof sketch, supplemented with references for more rigorous details.

\begin{lemma} \label{lem:higheigncase} Let $j\in\N_0$. Then for every $j\geq \frac{2\mathsf{w}}{\pi}$, the following holds
\begin{equation} \label{eq:psijlarge}
    \|\varphi_{\mathsf{w},j}\|_{L^{\infty}([-1,1])} \leq \frac{12}{5} \sqrt{j+1}.
\end{equation}
\end{lemma}

\begin{proof}[Sketch of proof]
Recall that the Slepian basis functions $\varphi_{\mathsf{w},j}$ serve as eigenfunctions for the operator $\mathcal{L}_{\mathsf{w}}$ \eqref{eqdef:Lw}, with the corresponding eigenvalues denoted by $\chi_{\mathsf{w},j}$ \eqref{eq:spectrumchi}.
Let $j\geq \frac{2\mathsf{w}}{\pi}$. Then it follows from \citep[Theorem 4.4]{osipov2013prolate} that
\begin{equation} \label{eq:leg1}
    \chi_{\mathsf{w},j} > \mathsf{w}^2.
\end{equation}
On the one hand, when \eqref{eq:leg1} is satisfied, we can employ \citep[Theorem 3.39]{osipov2013prolate} to obtain
\begin{equation} \label{eq:leg2}
    \|\varphi_{\mathsf{w},j}\|_{L^{\infty}([-1,1])} = |\varphi_{\mathsf{w},j}(1)|.
\end{equation}
On the other hand, according to \citep[Theorem~3.1]{bonami2014uniform}, \eqref{eq:leg1} also yields 
\begin{equation} \label{eq:leg3}
    |\varphi_{\mathsf{w},j}(1)| \leq \frac{12}{5}\sqrt{j+1}.
\end{equation}
Combining \eqref{eq:leg1}, \eqref{eq:leg2}, \eqref{eq:leg3}, we acquire \eqref{eq:psijlarge}, as desired.
\end{proof}

\begin{proof}[Proof of Proposition~\ref{prop:Phi_j}]
Extend the uniform measure ${\bf u}$ on $[-1,1]$ to a measure on $\R$ by setting ${\rm d}{\bf u}(t) = \frac{1}{2} {\rm d}t$.
Under the standard energy concentration interpretation, $\mu_{\mathsf{w},j}$ represents the fraction of the total energy of $\varphi_{\mathsf{w},j}$ contained in $[-1,1]$ (see Section~\ref{sec:bandlimited}). Therefore, with the normalization $\|\varphi_{\mathsf{w},j} \|_{L_{\bf u}^2(\R)}=1$, we obtain
\begin{equation} \label{eq:decayvarphi}
    \mu_{\mathsf{w},j} = \frac{\|\varphi_{\mathsf{w},j} \|_{L^2([-1,1])}}{\|\varphi_{\mathsf{w},j} \|_{L^2(\R)}} = \frac{1}{\|\varphi_{\mathsf{w},j} \|_{L_{\bf u}^2(\R)}}.
\end{equation}
Since $\varphi_{\mathsf{w},j}\in{\rm PW}_{\mathsf{w}}$, \eqref{eq:bandlimitedequiv} holds. It follows that
\begin{equation} \label{eq:key}
    |\varphi_{\mathsf{w},j}(x)| = \bigg|\frac{1}{2\pi}\int_{-\mathsf{w}}^{\mathsf{w}} \widehat{\varphi_{\mathsf{w},j}}(t)e^{ixt}\,dt\bigg|
    \leq \sqrt{\frac{\mathsf{w}}{\pi}} \|\varphi_{\mathsf{w},j}\|_{L^2(\R)} = \sqrt{\frac{2\mathsf{w}}{\pi}}\|\varphi_{\mathsf{w},j}\|_{L_{\bf u}^2(\R)}, \qquad \forall x\in\R.
\end{equation}
For $j\leq \lfloor 4\mathsf{w}\rfloor-1$, applying Lemma~\ref{lem:spliteigenrecall} together with \eqref{eq:decayvarphi} and \eqref{eq:key} yields
\begin{equation*} \label{eq:varphibigjLinfty}
    \|\varphi_{\mathsf{w},j}\|_{L^{\infty}([-1,1])} 
    \leq \|\varphi_{\mathsf{w},j}\|_{L^{\infty}(\R)} 
    \leq \sqrt{\frac{2\mathsf{w}}{\pi}}\|\varphi_{\mathsf{w},j}\|_{L_{\bf u}^2(\R)}
    \leq 2\sqrt{\frac{\mathsf{w}}{\pi}}.
\end{equation*}
For the remaining case of $j\geq \lfloor 4\mathsf{w}\rfloor$, we recall from Lemma~\ref{lem:higheigncase} that \eqref{eq:psijlarge} holds whenever $j\geq  \frac{2\mathsf{w}}{\pi}$. 
The proof is now completed. 
\end{proof}

\subsection{Preliminary results for Theorem~\ref{thm:leastsquaresampling}} \label{sec:prelimforthm1}

The existence of a unique solution to \eqref{leastsquaresproblem}, for any finite index set $\Lambda\subset\mathbb{N}_0^d$ can be ensured using the \emph{discrete stability constant} associated with $\mathcal{S}_{\Lambda}$ and the sample size $m$; that is
\begin{equation*}
    \alpha 
    = \alpha(\mathcal{S}_{\Lambda}, \{\vec{y}_j\}_{j=1}^m) \\
    := \inf \bigg\{\frac{1}{m}\sum_{j=1}^m |g(y_j)|^2 : g\in\mathcal{S}_{\Lambda} \text{ and } \|g\|_{L_{{\bf u}}^2([-1,1]^d)}=1  \bigg\}.
\end{equation*}
Then, since $\mathcal{S}_{\Lambda}$ is a finite-dimensional space, with a basis formed by the vectors $\{\varphi_{\mathsf{w},\vec{\nu}}\}_{\vec{\nu}\in\Lambda}$, we apply the Courant-Fisher min-max theorem \citep[Theorem 4.2.6]{horn2012matrix} to get 
\begin{equation*} 
    \alpha = \sigma_{\rm min}(A) = \lambda_{\rm min}(A^*A),
\end{equation*}
where $A\in\mathbb{R}^{m\times\texttt{\#}\Lambda}$ is given by 
\begin{equation} \label{eqdef:A}
    A_{j,\vec{\nu}} := \dfrac{1}{\sqrt{m}} \varphi_{\mathsf{w},\vec{\nu}}(\vec{y}_j).
\end{equation}
Here, $\sigma_{\rm min}(A)$ denotes the minimum singular value of $A$ and $\lambda_{\rm min}(A^*A)$ the minimum eigenvalue of $A^*A$.
Note that to define $A^*$, we assume a fixed ordering of $\vec{\nu}\in\Lambda$.
A positive lower bound on $\alpha$, crucial for the existence of a unique solution to \eqref{leastsquaresproblem}, can be further guaranteed via Monte Carlo sampling if $m$ is sufficiently large. 
This is formalized in the following theorem, adapted from \citep[Theorem 5.7]{adcock2022sparse}.

\begin{theorem} \label{thm:LS}
Let $\beta,\delta\in (0,1)$ and set\footnote{The quantity $\kappa(\mathcal{S}_{\Lambda})$ in \eqref{coherence} is known as the $L^{\infty}$-norm of the \emph{Christoffel function} of $\mathcal{S}_{\Lambda}$ \citep[Definition~5.4]{adcock2022sparse}.}
\begin{equation} \label{coherence}
    \kappa(\mathcal{S}_{\Lambda}) := \max_{\vec{y}\in [-1,1]^d} \sum_{\vec{\nu}\in\Lambda} |\varphi_{\mathsf{w},\vec{\nu}}(\vec{y})|^2.
\end{equation}
If $m\geq ((1-\delta)\log(1-\delta)+\delta)^{-1}\kappa(\mathcal{S}_{\Lambda})\log(\texttt{\#}\Lambda/\beta)$, then with probability at least $1-\beta$, the discrete stability constant $\alpha=\alpha(\mathcal{S}_{\Lambda}, \{\vec{y}_j\}_{j=1}^m)$ satisfies
\begin{equation} \label{lowbdalpha}
    \alpha > \sqrt{1-\delta}.
\end{equation}
\end{theorem}

Evidently from \eqref{coherence}, the size of $\kappa(\mathcal{S}_{\Lambda})$ is impacted by the choice of $\Lambda$. 
When $\Lambda$ is a hyperbolic cross in low dimensions, an upper bound for $\kappa(\mathcal{S}_{\Lambda})$ can be provided by the next proposition. 

\begin{proposition} \label{prop:Theta} 
Let $n\in\N$, and let $\Lambda = \Lambda^{\rm HC}_{n-1}$ be of dimension $d = 1,2,3$. 
If $\mathsf{w}\geq 1$, and $\mathcal{S}_{\Lambda}$ denotes the linear span of $\{\varphi_{\mathsf{w},\vec{\nu}}\}_{\vec{\nu}\in\Lambda}$, then it holds that 
\begin{equation} \label{Theta}
    \kappa(\mathcal{S}_{\Lambda})\leq  (2\texttt{\#}\Lambda)^{2\gamma(\mathsf{w})},
\end{equation}
for the following cases: 
\begin{enumerate}
    \item for $d=1,2$ and all $n\in\N$,
    \item for $d=3$ and all $n\geq 26$. 
\end{enumerate}
\end{proposition}

Since the index set $\Lambda=\Lambda^{\rm HC}_{n-1}$ contains $\vec{0}$, the function $\varphi_{\mathsf{w},\vec{0}}$ will always be included in the span $\mathcal{S}_\Lambda$. 
Proposition~\ref{prop:gammaw} and Remark~\ref{rem:firsteigen} suggest that estimating $\|\varphi_{\mathsf{w},\vec{0}}\|_{L^{\infty}([-1,1]^d)}$ may pose complications, as an upper bound for this value could reach $\mathsf{w}^d$, leading to an exponential growth of $\kappa(\mathcal{S}_{\Lambda})$ in high dimensions. 
Consequently, we restrict the selection of $\Lambda=\Lambda^{\rm HC}_{n-1}$ to dimensions up to $d=3$.

The proof of Proposition~\ref{prop:Theta} utilizes an induction argument on dimension, a strategy employed in the proof of \citep[Proposition~5.13]{adcock2022sparse} (also \citep[Lemma 3.3]{chkifa2015discrete}). 
For the base case $d=1$, we make use of Proposition~\ref{prop:gammaw} to provide an upper bound for $\kappa(\mathcal{S}_{\Lambda})$.
For the induction step, we observe from the slicing property \eqref{slicing} that each $\Lambda_k$ is a $(d-1)$-dimensional hyperbolic cross of order $\big\lfloor \frac{n}{k+1}\big\rfloor$, i.e.
\begin{equation} \label{eq:LambdakHC}
    \Lambda_{k} = \bigg\{\nu'=(\nu_2,\dots,\nu_d) \in\N_0^{d-1}: \prod_{k=2}^{d}(\nu_k+1)\leq \big\lfloor \frac{n}{k+1}\big\rfloor\bigg\},
\end{equation}
and particularly, 
\begin{equation} \label{eq:Lambda0HC}
    \Lambda_0 = \bigg\{\nu'=(\nu_2,\dots,\nu_{d})\in\N^{d-1}_0: \prod_{k=2}^{d}(\nu_k+1)\leq n\bigg\}
\end{equation}
is a $(d-1)$-dimensional hyperbolic cross of exactly order $n$. 
From there, we rely on the auxiliary yet crucial result stated below, which compares the sizes of horizontal slices in two- and three-dimensional hyperbolic cross index sets. 

\begin{proposition} \label{prop:relating}
Let $n\in\N$ and $d=2,3$. Let $\Lambda=\Lambda^{\rm HC}_{n-1}$ be the $d$-dimensional hyperbolic cross of order $n$ where $\Lambda = \bigsqcup_{k=0}^{n-1} \, \{k\}\times\Lambda_k$ as in \eqref{slicing}, \eqref{eq:LambdakHC}.
Then for $d=2$ and all $n\geq 2$, or for $d=3$ and all $n\geq 26$, 
\begin{equation} \label{crucialinequality}
    \sum_{k=1}^{n-1} (2k+1)\,(\texttt{\#}\Lambda_k)^2 \geq \frac{(\texttt{\#}\Lambda_0)^2}{n}. 
\end{equation}
\end{proposition}

\begin{proof} 
We first prove the proposition for $d=2$.
In this case, $\Lambda=\Lambda^{\rm HC}_{n-1}$ is a two-dimensional hyperbolic cross where each slice $\Lambda_k$, for $k=0,\dots,n-1$, is a one-dimensional hyperbolic cross.
From \eqref{eq:Lambda0HC}, we have $\texttt{\#}\Lambda_0 = n$.
Thus, proving the proposition amounts to proving
\begin{equation*} 
    \sum_{k=1}^{n-1} (2k+1)\,(\texttt{\#}\Lambda_k)^2 \geq n. 
\end{equation*}
However, this is immediate, since
\begin{equation*}
    \sum_{k=1}^{n-1} (2k+1)\,(\texttt{\#}\Lambda_k)^2 \geq \sum_{k=1}^{n-1} (2k+1) = n^2-1\geq n,
\end{equation*}
when $n\geq 2$, and we are done with $d=2$. 
Turning to $d=3$, we observe that
\begin{equation} \label{L2embed}
    \sum_{k=1}^{n-1} (2k+1) (\texttt{\#}\Lambda_k)^2 \geq 3(\texttt{\#}\Lambda_1)^2 \geq 3\Big\lfloor \frac{n}{2} \Big\rfloor^2.
\end{equation}
Note that in this case, each $\Lambda_k$ is a two-dimensional hyperbolic cross.
Here, the last inequality in \eqref{L2embed} holds since $\Lambda_1$ contains the set
\begin{equation*}
    \Big\{(0,j): 0\leq j\leq \Big\lfloor \frac{n}{2} \Big\rfloor -1\Big\}.
\end{equation*}
At the same time, we can write
\begin{equation*}
    \Lambda_0 = \bigcup_{k=0}^{n-1} \Big\{(k,j), 0 \leq j \leq \frac{n}{k+1} -1 \Big\},
\end{equation*}
which gives
\begin{equation} \label{Lambda0sizefact}
    \texttt{\#}\Lambda_0 \leq \sum_{k=0}^{n-1} \frac{n}{k+1} \leq n\Big(1+\int_1^n \frac{1}{x} \, {\rm d}x  \Big) = n(1+\log (n)). 
\end{equation}
Therefore, if
\begin{equation} \label{eq:claimthreshold}
    3\Big\lfloor \frac{n}{2} \Big\rfloor^2 \geq n(1+\log (n))^2,
\end{equation}
then the desired conclusion \eqref{crucialinequality} follows from \eqref{L2embed} and \eqref{Lambda0sizefact}.
A direct check shows that \eqref{eq:claimthreshold} holds for all $n\geq 26$, which completes the proof.
\end{proof}

We now proceed with the proof of Proposition~\ref{prop:Theta}.

\begin{proof}[Proof of Proposition~\ref{prop:Theta}]
To begin, we note that $\gamma(\mathsf{w})>1$ for $\mathsf{w}\geq 1$. Moreover
\begin{equation} \label{eq:powerscompared}
    2^{\gamma(\mathsf{w})-1}\geq \frac{4\mathsf{w}}{\pi}.
\end{equation}
When $d=1$, $\Lambda = \Lambda^{\rm HC}_{n-1} = \{0,\dots,n-1\}$. 
On the one hand, if $n=1$, then it follows from Proposition~\ref{prop:gammaw} and \eqref{eq:powerscompared} that
\begin{equation} \label{eq:s1}
    \kappa(\mathcal{S}_{\Lambda}) = \|\varphi_{\mathsf{w},0}\|^2_{L^{\infty}([-1,1])} \leq \frac{4\mathsf{w}}{\pi} \leq 2^{\gamma(\mathsf{w})} \leq (2\texttt{\#}\Lambda)^{2\gamma(\mathsf{w})},
\end{equation}
which satisfies \eqref{Theta}.
On the other hand, if $n\geq 2$, then we can again apply Proposition~\ref{prop:gammaw} to obtain
\begin{equation} \label{eq:d1s2}
    \kappa(\mathcal{S}_{\Lambda}) \leq \sum_{k=0}^{n-1} \|\varphi_{\mathsf{w},k}\|^2_{L^{\infty}([-1,1])} \leq \frac{4\mathsf{w}}{\pi} + \sum_{k=1}^{n-1} (k+1)^{\gamma(\mathsf{w})} \leq \frac{4\mathsf{w}}{\pi} + \sum_{k=1}^{n-1} (2k+1)^{\gamma(\mathsf{w})}. 
\end{equation}
We claim that 
\begin{equation} \label{eq:claim}
   \frac{4\mathsf{w}}{\pi} + \sum_{k=1}^{n-1} (2k+1)^{\gamma(\mathsf{w})} \leq \Big(\sum_{k=0}^{n-1} (2k+1) \Big)^{\gamma(\mathsf{w})} =n^{2\gamma(\mathsf{w})} \leq (2\texttt{\#}\Lambda)^{2 \gamma(\mathsf{w})}.
\end{equation}
It is enough to validate the first inequality in \eqref{eq:claim}. 
Then on account of Pascal's identity \citep[Chapter~6.4, Theorem~2]{rosen1999discrete} and \eqref{eq:powerscompared}, we have
\begin{align*}
    \Big(\sum_{k=0}^{n-1} 2k+1 \Big)^{\gamma(\mathsf{w})} &= 
    \Big( \Big(\sum_{k=1}^{n-1} 2k+1 \Big) + 1\Big)^{\gamma(\mathsf{w})} \\
    &= \sum_{l=0}^{\gamma(\mathsf{w})-1} {\gamma(\mathsf{w}) \choose l} \Big(\sum_{k=1}^{n-1} 2k+1\Big)^{l} + \Big(\sum_{k=1}^{n-1} 2k+1 \Big)^{\gamma(\mathsf{w})} \\
    &\geq \sum_{l=0}^{\gamma(\mathsf{w})-1} {\gamma(\mathsf{w}) -1 \choose l} 3^l +  \sum_{k=1}^{n-1} (2k+1)^{\gamma(\mathsf{w})} \\
    &= 4^{\gamma(\mathsf{w})-1} +  \sum_{k=1}^{n-1} (2k+1)^{\gamma(\mathsf{w})} \\
    &\geq \frac{4\mathsf{w}}{\pi} + \sum_{k=1}^{n-1} (2k+1)^{\gamma(\mathsf{w})}.
\end{align*}
Therefore, \eqref{eq:claim} holds. By combining \eqref{eq:claim} and \eqref{eq:d1s2}, we deduce \eqref{Theta} for $d=1$ and $n\geq 2$. 

When $d=2$ and $n=1$, we have $\Lambda = \{(0,0)\}$.
Then 
a similar calculation as in \eqref{eq:s1} yields
\begin{equation*}
    \kappa(\mathcal{S}_{\Lambda}) = \|\varphi_{\mathsf{w},(0,0)}\|^2_{L^{\infty}([-1,1]^2)} = \|\varphi_{\mathsf{w},0}\|^4_{L^{\infty}([-1,1])} \leq \frac{16\mathsf{w}^2}{\pi^2} \leq 2^{2\gamma(\mathsf{w})} = (2\texttt{\#}\Lambda)^{2\gamma(\mathsf{w})},
\end{equation*}
which establishes \eqref{Theta}.
When $d=2$ and $n\geq 2$, we apply \eqref{Theta} to each one-dimensional hyperbolic cross $\Lambda_k$ \eqref{eq:LambdakHC}, \eqref{eq:Lambda0HC} and invoke Proposition~\ref{prop:gammaw}, obtaining
\begin{align} \label{eq:induction1}
    \nonumber \kappa(\mathcal{S}_{\Lambda}) &\leq \sum_{k=0}^{n-1} \|\varphi_{\mathsf{w},k}\|^2_{L^{\infty}([-1,1])} \sum_{j\in\Lambda_k} \|\varphi_{\mathsf{w},j}\|^2_{L^{\infty}([-1,1])} \\
    \nonumber &= \|\varphi_{\mathsf{w},0}\|^2_{L^{\infty}([-1,1])} \sum_{j\in\Lambda_0} \|\varphi_{\mathsf{w},j}\|^2_{L^{\infty}([-1,1])} + \sum_{k=1}^{n-1} \|\varphi_{\mathsf{w},k}\|^2_{L^{\infty}([-1,1])} \sum_{j\in\Lambda_k} \|\varphi_{\mathsf{w},j}\|^2_{L^{\infty}([-1,1])} \\
    &\leq \frac{4\mathsf{w}}{\pi}(2\texttt{\#}\Lambda_0)^{2\gamma(\mathsf{w})} + \sum_{k=1}^{n-1} (k+1)^{\gamma(\mathsf{w})}(2\texttt{\#}\Lambda_k)^{2\gamma(\mathsf{w})}. 
\end{align}
As with the case $d=1$, we aim to bound the right-hand side of \eqref{eq:induction1} 
by $(\sum_{k=0}^{n-1} (2k+1)\,(2\texttt{\#}\Lambda_k)^2)^{\gamma(\mathsf{w})}$. For brevity, we denote $A_k := 2\texttt{\#}\Lambda_k$. Then
\begin{align}
    \nonumber &\Big(\sum_{k=0}^{n-1} (2k+1)\,A_k^2\Big)^{\gamma(\mathsf{w})} \\
    \nonumber &= \Big(\sum_{k=1}^{n-1} (2k+1)\,A_k^2\Big)^{\gamma(\mathsf{w})} + \sum_{l=1}^{\gamma(\mathsf{w})} {\gamma(\mathsf{w}) \choose l} \Big(\sum_{k=1}^{n-1} (2k+1)\,A_k^2\Big)^{\gamma(\mathsf{w})-l} A_0^{2l} \\
    \label{eq:induction2} &\geq \sum_{k=1}^{n-1} (2k+1)^{\gamma(\mathsf{w})} A_k^{2\gamma(\mathsf{w})} + \sum_{l=1}^{\gamma(\mathsf{w})} {\gamma(\mathsf{w}) \choose l} \Big(\sum_{k=1}^{n-1} (2k+1)\,A_k^2\Big)^{\gamma(\mathsf{w})-l} A_0^{2l}.
\end{align}
An application of Proposition~\ref{prop:relating} gives
\begin{equation*}
    \Big(\sum_{k=1}^{n-1} (2k+1)\,A_k^2\Big)^{\gamma(\mathsf{w})-l} \geq \frac{A_0^{2(\gamma(\mathsf{w})-l)}}{n^{\gamma(\mathsf{w})-l}}.
\end{equation*}
Therefore
\begin{align}
    \nonumber \sum_{l=1}^{\gamma(\mathsf{w})} {\gamma(\mathsf{w}) \choose l} \Big(\sum_{k=1}^{n-1} (2k+1)\,A_k^2\Big)^{\gamma(\mathsf{w})-l} A_0^{2l} &\geq \frac{A_0^{2\gamma(\mathsf{w})}}{n^{\gamma(\mathsf{w})}} \sum_{l=1}^{\gamma(\mathsf{w})} {\gamma(\mathsf{w}) \choose l} n^{l} \\
    \nonumber &= A_0^{2\gamma(\mathsf{w})}\big(1+ n + n^2 + \dots + n^{\gamma(\mathsf{w})-1}\big)\\
    \nonumber &\geq A_0^{2\gamma(\mathsf{w})}2^{\gamma(\mathsf{w})-1} \\
    \label{eq:induction3} &\geq A_0^{2\gamma(\mathsf{w})} \frac{4\mathsf{w}}{\pi},
\end{align}
where the last inequality is due to \eqref{eq:powerscompared}.
Putting back $A_k = 2\texttt{\#}\Lambda_k$, and combining \eqref{eq:induction2}, \eqref{eq:induction3} together with \eqref{eq:induction1} 
yields
\begin{align} \label{eq:induction4}
    \nonumber \kappa(\mathcal{S}_{\Lambda}) &\leq \frac{4\mathsf{w}}{\pi}(2\texttt{\#}\Lambda_0)^{2\gamma(\mathsf{w})} + \sum_{k=1}^{n-1} (k+1)^{\gamma(\mathsf{w})}(2\texttt{\#}\Lambda_k)^{2\gamma(\mathsf{w})} \\
    &\leq 4^{\gamma(\mathsf{w})} \Big(\sum_{k=0}^{n-1} (2k+1)\,(\texttt{\#}\Lambda_k)^2\Big)^{\gamma(\mathsf{w})}.
\end{align}
Next, by utilizing the monotonicity property \eqref{nesting}
\begin{equation*}    
    k(\texttt{\#}\Lambda_k)^2 = \sum_{j=1}^{k} (\texttt{\#}\Lambda_k)^2 \leq \sum_{j=0}^{k-1} (\texttt{\#}\Lambda_k) (\texttt{\#}\Lambda_j),
\end{equation*}
we can conclude from \eqref{eq:induction4} that
\begin{align*}
    \kappa(\mathcal{S}_{\Lambda}) 
    &\leq 4^{\gamma(\mathsf{w})} \Big(\sum_{k=0}^{n-1} (2k+1)\,(\texttt{\#}\Lambda_k)^2\Big)^{\gamma(\mathsf{w})} \\
    &\leq 4^{\gamma(\mathsf{w})} \Big(\sum_{k=0}^{n-1}(\texttt{\#}\Lambda_k)^2 + 2\sum_{j<k} (\texttt{\#}\Lambda_k)(\texttt{\#}\Lambda_j) \Big)^{\gamma(\mathsf{w})} \\
    &= 4^{\gamma(\mathsf{w})} \Big(\sum_{k=0}^{n-1} \texttt{\#}\Lambda_k \Big)^{2\gamma(\mathsf{w})}
    = (2\texttt{\#}\Lambda)^{2\gamma(\mathsf{w})},
\end{align*}
which is \eqref{Theta} when $d=2$ and $n\geq 2$. 

Finally, when $d=3$ and $n \geq 26$, we adopt arguments similar to those used for $d=2$ and $n \geq 2$. Namely, we first apply \eqref{Theta} to each two-dimensional hyperbolic cross $\Lambda_k$ \eqref{eq:LambdakHC}, \eqref{eq:Lambda0HC}, and invoke Proposition~\ref{prop:gammaw}, to obtain
\begin{align} \label{eq:induction1d3}
    \nonumber \kappa(\mathcal{S}_{\Lambda}) &= \sum_{k=0}^{n-1} \|\varphi_{\mathsf{w},k}\|^2_{L^{\infty}([-1,1])} \sum_{\vec{\nu}\in\Lambda_k} \|\varphi_{\mathsf{w},\vec{\nu}}\|^2_{L^{\infty}([-1,1]^2)} \\
    &\leq \frac{4\mathsf{w}}{\pi}(2\texttt{\#}\Lambda_0)^{2\gamma(\mathsf{w})} + \sum_{k=1}^{n-1} (k+1)^{\gamma(\mathsf{w})}(2\texttt{\#}\Lambda_k)^{2\gamma(\mathsf{w})},
\end{align}
which matches the form of \eqref{eq:induction1}. From this point on, the argument proceeds exactly as in the case $d=2$ and $n\geq 2$, with the sole modification being the application of Proposition~\ref{prop:relating} for $d=3$ and $n \geq 26$. We therefore omit further details. 
\end{proof}

\subsection{Proof of Theorem~\ref{thm:leastsquaresampling}} \label{sec:leastsquaresproof}

Building on Proposition~\ref{prop:Theta}, if we take
\begin{equation*}
    m\geq ((1-\delta)\log(1-\delta)+\delta)^{-1}(2\texttt{\#}\Lambda)^{2\gamma(\mathsf{w})}\log(\texttt{\#}\Lambda/\beta),
\end{equation*}
then we guarantee \eqref{lowbdalpha} from Theorem~\ref{thm:LS}.
From there, the theory outlined in \citep[Section~5]{adcock2022sparse}, particularly \citep[Corollary~5.9]{adcock2022sparse}, demonstrates that every $f \in \mathcal{C}([-1,1]^d)$ 
has a unique least squares approximation $f^{\natural}$ in \eqref{leastsquaresproblem} that satisfies 
\begin{align} \label{punchline}
    \nonumber 
    \|f - f^{\natural}\|_{L_{ {\bf u}}^2([-1,1]^d)} 
    &\leq (1+ \alpha^{-1}) \inf_{g \in \mathcal{S}_{\Lambda}} 
    \|f-g \|_{L^{\infty}([-1,1]^d)} + \dfrac{1}{\alpha} \|\vec{e}\|_2 \\
    &< \Big(1+ \frac{1}{\sqrt{1-\delta}} \Big) \inf_{g \in \mathcal{S}_{\Lambda}} 
    \|f-g \|_{L^{\infty}([-1,1]^d)} + \frac{1}{\sqrt{1-\delta}} \|\vec{e}\|_2,
\end{align}
thereby completing the proof of Theorem~\ref{thm:leastsquaresampling}. \qed



\subsection{Preliminary results for Theorem~\ref{thm:ApproxSpan}} \label{sec:prelimformainthm2}

The proof of Theorem~\ref{thm:ApproxSpan}, is based on a three-step approximation approach using NNs: 
\begin{itemize}
    \item {\bf Step 1:} approximate Legendre polynomials using NNs;
    \item {\bf Step 2:} construct Slepian basis functions through linear combinations of Legendre polynomials and subsequently of the NN approximations derived in {\bf Step 1};
    \item {\bf Step 3:} approximate frequency-localized functions by linearly combining the appropriate Slepian basis functions, subsequently expressed in terms of the NN approximations obtained in {\bf Step 2}.
\end{itemize}
The culmination of {\bf Step 1} and {\bf Step 2} is Proposition~\ref{prop:ApproxSlepian}, the main preliminary result of this subsection, which will be presented shortly. 
The outcome of {\bf Step 3} will be our second main theorem, Theorem~\ref{thm:ApproxSpan}, with its proof provided in Section~\ref{sec:ProofTheo3}.
The initial steps rely on the fact that Legendre polynomials can effectively approximate Slepian basis functions.
Thus, we begin with an overview of Legendre polynomials.

For each $k \in \N_0$, the Legendre polynomial $\widetilde{P}_k$ in one dimension is defined as the solution to the Sturm-Liouville eigenvalue problem
\begin{eqnarray*}
    (1-x^2)\dfrac{d^2}{dx^2}\widetilde{P}_k - 2x \dfrac{d}{dx}\widetilde{P}_k+k(k+1)\widetilde{P}_k = 0, \qquad \forall x\in [-1,1], 
\end{eqnarray*}
such that $\max_{x\in [-1,1]} |\widetilde{P}_k(x)|=1$.
The family $\{\widetilde{P}_k\}_{k\in\mathbb{N}_0}$ forms an orthogonal basis in $L^2_{\bf u}([-1,1])$. Moreover, as shown in, e.g., \citep[Section~2.2.2]{adcock2022sparse}, 
\begin{equation*} 
    \|\widetilde{P}_k \|_{L^{2}_{ {\bf u}}([-1,1])} = 
    \dfrac{1}{\sqrt{2k+1}}.
\end{equation*} 
From this, we obtain a normalized\footnote{Recall here that ``normalized'' means having the $L^2_{\bf u}([-1,1])$-norm of $1$.} Legendre polynomial $P_k$ of degree $k$ by setting 
\begin{equation} \label{eq:NormLegendr}
    P_{k} := \sqrt{2k+1}  \widetilde{P}_k.
\end{equation} 
Similar to the approach used to build the $d$-dimensional Slepian functions in Section~\ref{sec:background}, for $d \geq 2$ and for each index $\vec{\nu}=(\nu_1,\dots,\nu_d)\in\N^d_0$ we define 
\begin{equation} \label{eq:highLeg}
    P_{\vec{\nu}} := P_{\nu_1} \otimes \cdots \otimes P_{\nu_d},
\end{equation}
where each $P_{\nu_j}$ is a normalized Legendre polynomial in one dimension. 
Evidently, $\{P_{\vec{\nu}}\}_{\vec{\nu}\in\N_0^d}$ forms an orthonormal basis of $L^2_{{\bf u}}([-1,1]^d)$. 

It is known that the normalized Legendre polynomials $\{P_k\}_{k \in \mathbb{N}_0}$ in one dimension can be effectively approximated by ReLU neural networks. 
A concrete example is provided in \citep[Proposition 2.11]{opschoor2022exponential} (see also Proposition~\ref{prop:AppLeg1D}). 
In the case of $d$ dimensions, we frequently rely on the tensorization definition \eqref{eq:highLeg} and the fundamental network calculus operations outlined in Appendix~\ref{appx:basicnetworks}, such as concatenation, parallelization, and product approximation, to develop approximating NNs for normalized Legendre polynomials. 
Although these constructions are well-established in the literature, we include them for completeness and convenient reference as the reader progresses through the material.

We are now ready to present Proposition~\ref{prop:ApproxSlepian}. 


\begin{proposition} \label{prop:ApproxSlepian} 
Let $\varepsilon\in (0,1)$, and $\mathsf{w}\geq 1$.
Let $\Lambda=\Lambda_{n-1}^{\rm HC}$ be the $d$-dimensional hyperbolic cross of order $n\geq 2$, with $d=1,2,3$. 
Let $B(d,n)$, $M(d,n)$ be given \eqref{eq:B}, \eqref{eq:M}, respectively.
Then there exist a universal constant $C>0$ and $N_{\star}=N_{\star}(\Lambda,\mathsf{w},\varepsilon)\in\N$ for which the following holds.
For every Slepian basis function $\varphi_{\mathsf{w}, \vec{\nu}}$ where $\vec{\nu}\in\Lambda_{n-1}^{\rm HC}$, there exists an NN $\Psi_{\varepsilon,N_{\star}}^{\vec{\nu}}$ such that 
\begin{equation*}
     \| \varphi_{\mathsf{w},\vec{\nu}} - {\rm R}(\Psi_{\varepsilon,N_{\star}}^{\vec{\nu}}) \|_{L^{\infty}([-1, 1]^d)}
    \leq B(d,n)\varepsilon,
\end{equation*}
with 
\begin{align*}   
    L(\Psi_{\varepsilon,N_{\star}}^{\vec{\nu}}) 
    &
    \leq C \big((1+\log_2N_{\star})(N_{\star}+\log_2(N_{\star}/\varepsilon)) + (1+ \log_2 (M(d,n)/\varepsilon)\big),\\
    W(\Psi_{\varepsilon,N_{\star}}^{\vec{\nu}}) 
    & 
    \leq C\big((N_{\star}^2 + N_{\star})(N_{\star} + \log_2(N_{\star}/\varepsilon)) + (1+ \log_2 (M(d,n)/\varepsilon)\big).
\end{align*}
\end{proposition}

To continue with the proof, we need the following two supporting results. The first, Lemma~\ref{lem:ApproxOfSlep1D}, establishes that every $\varphi_{\mathsf{w}, k}$ can be well-approximated by a linear combination of normalized Legendre polynomials, with a controllable number of terms.
The second, Lemma~\ref{lem:AppSumOfLeg}, establishes that any such linear combination can be efficiently modeled by an NN of manageable length and complexity.

\begin{lemma}\label{lem:ApproxOfSlep1D}
Let $\mathsf{w}\geq 1$. 
Let $\Lambda\subset\mathbb{N}$ be a finite set. 
For $k\in\Lambda$, let $\varphi_{\mathsf{w},k}$ be the $k$th Slepian basis function associated with the eigenvalue $\mu_{\mathsf{w},k}$ \eqref{eq:muj}. 
For every $\varepsilon\in (0,1)$, let
\begin{equation} \label{eqdef:Nstar}
    N_{\star} := \Big \lceil  \max \Big\{2\lfloor e  \mathsf{w} \rfloor +1, \frac{\log(3/(c_{\star}\varepsilon))}{\log (3/2)} \Big\} \Big\rceil,
\end{equation}
where $c_{\star} := \min_{k\in\Lambda} \mu_{\mathsf{w},k}$.
For every normalized Legendre polynomial $P_j$ of degree $j\in\mathbb{N}$, let
\begin{equation} \label{eqdef:betacoef}
    \beta_j^k := \frac{1}{2} \int_{-1}^1 \varphi_{\mathsf{w},k}(x) \overline{P_j}(x) \, {\rm d}x.
\end{equation}
Then for every $N\geq N_{\star}$ and for every $k \in\Lambda$, it holds that
\begin{equation}\label{eq:ApprSlepianUnif}
    \Big\| \varphi_{\mathsf{w}, k} - \sum_{j=0}^{N} \beta_j^k P_j \Big\|_{L^{\infty}([-1,1])} \leq \varepsilon.
\end{equation}
\end{lemma}

\begin{lemma} \label{lem:AppSumOfLeg}
Let $\varepsilon \in (0,1)$, and $N \in \N$. Let $\{\beta_j\}_{j=0}^N $ be a set of real numbers such that $|\beta_j| \leq c_0$ for some $c_0>0$.
Let $P_j$ be the normalized Legendre polynomial of degree $j$. 
Then there exists an NN $\widetilde{\Phi}_{\varepsilon, N}$ such that
\begin{equation*}
    \Big\| \sum_{j=0}^N \beta_j P_j -  {\rm R}( \widetilde{\Phi}_{\varepsilon, N}) \Big\|_{L^{\infty}([-1,1])} \leq c_0 \varepsilon.
\end{equation*}
Moreover, 
\begin{align*}
    L (\widetilde{\Phi}_{\varepsilon, N}) &\leq  C(1+ \log_2 N)(N + \log_2(N/\varepsilon) ), \\
    W ( \widetilde{\Phi}_{\varepsilon, N} ) &\leq C (N^2+N)(N+\log_2(N/\varepsilon)),
\end{align*}
for some universal constant $C>0$.
\end{lemma}

The proofs of Lemmas~\ref{lem:ApproxOfSlep1D},~\ref{lem:AppSumOfLeg} are given at the end of this subsection, in the order in which they are presented.

\begin{proof}[Proof of Proposition~\ref{prop:ApproxSlepian}]
We begin with the case $\Lambda=\Lambda^{\rm HC}_{n-1}\subset\mathbb{N}$, i.e. when $\{\varphi_{\mathsf{w}, k}\}_{k \in \Lambda}$ is a finite subset of the Slepian orthonormal basis of $L^2_{\bf u}([-1,1])$. 
Let $N_{\star}$ be as given in \eqref{eqdef:Nstar}, with $\varepsilon\in (0,1)$.
Note, when $\Lambda=\Lambda^{\rm HC}_{n-1}=\{0,\dots,n-1\}$,
\begin{equation} \label{eq:cstarreq}
    c_{\star} = \min_{k\in\Lambda} \mu_{\mathsf{w},k} = \mu_{\mathsf{w},n-1}.
\end{equation}
By Lemma~\ref{lem:ApproxOfSlep1D}, \eqref{eq:ApprSlepianUnif} holds with $N=N_{\star}$. 
To apply Lemma~\ref{lem:AppSumOfLeg} to our context, we need an estimate for $\max_{k \in \Lambda} \max_{j = 0, \dots, N_{\star}} |\beta_j^k|$, which will act as the value $c_0$.
Since $n\geq 2$, we obtain from Proposition~\ref{prop:gammaw} that 
\begin{equation} \label{eq:ngammaw}
    \|\varphi_{\mathsf{w},k}\|_{L^{\infty}([-1,1])} \leq \max\Big\{\frac{4\mathsf{w}}{\pi}, (k+1)^{\gamma(\mathsf{w})}\Big\} \leq n^{\gamma(\mathsf{w})},
\end{equation}
and subsequently, from \eqref{eqdef:betacoef}, 
\begin{multline*} 
    |\beta_j^k| 
    = \frac{1}{2} \Big|\int_{-1}^1 \varphi_{\mathsf{w},k}(x) \overline{P_j}(x) \, {\rm d}x \Big|
    \leq \|P_j\|_{L^2_{\bf u}([-1,1])} \|\varphi_{\mathsf{w},k}\|_{L^{\infty}_{\bf u}([-1,1])} \\
    = \|P_j\|_{L^2_{\bf u}([-1,1])} \|\varphi_{\mathsf{w},k}\|_{L^{\infty}([-1,1])}
    \leq n^{\gamma(\mathsf{w})}.
\end{multline*}
Thus, Lemma \ref{lem:AppSumOfLeg} ensures the existence of NN $\widetilde{\Phi}_{\varepsilon,N_{\star}}^k$ satisfying 
\begin{equation} \label{eq:NNuniformLambda}
    \Big\| \sum_{j=0}^{N_{\star}} \beta_j^k P_j -  {\rm R} (\widetilde{\Phi}_{\varepsilon,N_{\star}}^k) \Big\|_{L^{\infty}([-1,1])} \leq n^{\gamma(\mathsf{w})}\varepsilon,
\end{equation} 
with the following respective numbers of layers and weights
\begin{align}
    \label{eq:layers1d}
    L (\widetilde{\Phi}_{\varepsilon,N_{\star}}^k) &\leq  C(1+\log_2 N_{\star})(N_{\star} + \log_2(N_{\star}/\varepsilon)), \\
    \label{eq:weight1d}
    W (\widetilde{\Phi}_{\varepsilon,N_{\star}}^k) &\leq C(N_{\star}^2+N_{\star})(N_{\star}+\log_2(N_{\star}/\varepsilon)).
\end{align}
By setting $N = N_{\star}$ in \eqref{eqdef:Nstar} and combining it with \eqref{eq:NNuniformLambda}, we derive
\begin{equation} \label{eq:NNfor1d}
    \Big \| \varphi_{\mathsf{w}, k } - {\rm R} (\widetilde{\Phi}_{\varepsilon,N_{\star}}^k) \Big \|_{L^{\infty}([-1,1])} 
    \leq \big(1+ n^{\gamma(\mathsf{w})}\big)\varepsilon, 
\end{equation} 
and we conclude for the case $\Lambda^{\rm HC}_{n-1}\subset\mathbb{N}$ by letting $\Psi_{\varepsilon,N_{\star}}^k=\widetilde{\Phi}_{\varepsilon,N_{\star}}^k$.

For the case $\Lambda=\Lambda^{\rm HC}_{n-1}\subset\mathbb{N}^2$, we write, $\varphi_{\mathsf{w}, \vec{\nu}} = \varphi_{\mathsf{w}, \nu_1} \otimes \varphi_{\mathsf{w}, \nu_2}$, for each $\vec{\nu}=(\nu_1, \nu_2)\in\Lambda^{\rm HC}_{n-1}$. 
Note, $\nu_j\in \{0,\dots,n-1\}$.
Let $N_{\star}$ remain as previously defined with the choice $c_{\star}$ in \eqref{eq:cstarreq}.
Then by \eqref{eq:NNfor1d}, for $j=1,2$, there exists an NN $\widetilde{\Phi}_{\varepsilon,N_{\star}}^{\nu_j}$, such that 
\begin{equation} \label{eq:diffdimone}
    \| \varphi_{\mathsf{w}, \nu_j} - {\rm R}(\widetilde{\Phi}_{\varepsilon,N_{\star}}^{\nu_j}) \|_{L^{\infty}([-1,1])} \leq (1+n^{\gamma(\mathsf{w})})\varepsilon.
\end{equation}
This, together with \eqref{eq:ngammaw}, implies
\begin{equation} \label{eq:NNinftydimone}
    \| {\rm R}(\widetilde{\Phi}_{\varepsilon,N_{\star}}^{\nu_j}) \|_{L^{\infty}([-1,1])} 
    \leq 
    \|\varphi_{\mathsf{w}, \nu_j}\|_{L^{\infty}([-1,1])} + (1+n^{\gamma(\mathsf{w})})\varepsilon \leq 2n^{\gamma(\mathsf{w})} + 1.
\end{equation}
Combining \eqref{eq:ngammaw}, \eqref{eq:diffdimone}, \eqref{eq:NNinftydimone}, we derive for $\vec{\nu}=(\nu_1, \nu_2)\in\Lambda^{\rm HC}_{n-1}$ that
\begin{equation} \label{eq:splittingsum}
    \| \varphi_{\mathsf{w}, \nu_1}\otimes\varphi_{\mathsf{w},\nu_2}  - {\rm R}(\widetilde{\Phi}_{\varepsilon,N_{\star}}^{\nu_1}) \otimes {\rm R} (\widetilde{\Phi}_{\varepsilon, N_{\star}}^{\nu_2}) \|_{L^{\infty}([-1,1]^2)} 
    \leq a + b,
\end{equation}
where
\begin{align*}
    a &= \| \varphi_{\mathsf{w}, \nu_1} \otimes \varphi_{\mathsf{w}, \nu_2} - \varphi_{\mathsf{w}, \nu_1}\otimes {\rm R}(\widetilde{\Phi}_{\varepsilon,N_{\star}}^{\nu_2}) \|_{L^{\infty}([-1,1]^2)} \\
    & \leq \| \varphi_{\mathsf{w}, \nu_1} \|_{L^{\infty}([-1,1])} \| \varphi_{\mathsf{w}, \nu_2}  - {\rm R}(\widetilde{\Phi}_{\varepsilon,N_{\star}}^{\nu_2}) \|_{L^{\infty}([-1,1])} 
    \leq n^{\gamma(\mathsf{w})}(1+n^{\gamma(\mathsf{w})})\varepsilon,
\end{align*}
and 
\begin{align*}
    b &= \| \varphi_{\mathsf{w}, \nu_1}\otimes {\rm R}(\widetilde{\Phi}_{\varepsilon,N_{\star}}^{\nu_2}) - {\rm R}(\widetilde{\Phi}_{\varepsilon,N_{\star}}^{\nu_1}) \otimes {\rm R}(\widetilde{\Phi}_{\varepsilon,N_{\star}}^{\nu_2}) 
    \|_{L^{\infty}([-1,1]^2)} \\
    &\leq \| {\rm R}(\widetilde{\Phi}_{\varepsilon,N_{\star}}^{\nu_2}) \|_{L^{\infty}([-1,1])}  
    \| \varphi_{\mathsf{w}, \nu_1} - {\rm R}(\widetilde{\Phi}_{\varepsilon,N_{\star}}^{\nu_1})
    \|_{L^{\infty}([-1,1])} 
    \leq (2n^{\gamma(\mathsf{w})} + 1)(1+n^{\gamma(\mathsf{w})})\varepsilon.
\end{align*}
Therefore, \eqref{eq:splittingsum} yields
\begin{equation}\label{eq:Esti1}
    \| \varphi_{\mathsf{w}, \nu_1} \otimes \varphi_{\mathsf{w}, \nu_2} - {\rm R}(\widetilde{\Phi}_{\varepsilon,N_{\star}}^{\nu_1}) \otimes {\rm R} (\widetilde{\Phi}_{\varepsilon,N_{\star}}^{\nu_2}) \|_{L^{\infty}([-1,1])} 
    \leq (3n^{2\gamma(\mathsf{w})} + 4n^{\gamma(\mathsf{w})} +1) \varepsilon.
\end{equation}
Next, we utilize the network construction $\tilde{\times}_{\varepsilon, B}$ in  Proposition \ref{prop:Multiplication}, with $B=2n^{\gamma(\mathsf{w})}+1$ (obtained from \eqref{eq:NNinftydimone}), as well as the concatenation and parallelization operations, denoted as $\odot$ and ${\rm P}$, from Definitions~\ref{def:concat} and \ref{def:Parallelization} respectively, to construct
\begin{equation*}
    \Psi_{\varepsilon,N_{\star}}^{\vec{\nu}} :=
    \tilde{\times}_{\varepsilon, B} \odot {\rm P}(\widetilde{\Phi}_{\varepsilon,N_{\star}}^{\nu_1},\widetilde{\Phi}_{\varepsilon,N_{\star}}^{\nu_2}).
\end{equation*}
It follows from the construction that
\begin{equation} \label{eq:ProdMat}
    \| {\rm R}(\widetilde{\Phi}_{\varepsilon,N_{\star}}^{\nu_1}) \otimes {\rm R}(\widetilde{\Phi}_{\varepsilon,N_{\star}}^{\nu_2}) - {\rm R}(\Psi_{\varepsilon,N_{\star}}^{\vec{\nu}}) \|_{L^{\infty}([-1,1]^2)} 
    \leq \varepsilon. 
\end{equation}
Combining \eqref{eq:Esti1}, \eqref{eq:ProdMat}, we obtain
\begin{equation} \label{eq:Approx2D}
    \| \varphi_{\mathsf{w}, \nu_1} \otimes \varphi_{\mathsf{w}, \nu_2}  - {\rm R}(\Psi_{\varepsilon,N_{\star}}^{\vec{\nu}}) \|_{L^{\infty}([-1,1]^2)}
    \leq (3n^{2\gamma(\mathsf{w})} + 4n^{\gamma(\mathsf{w})} +2)\varepsilon.
\end{equation}
Furthermore, by applying \eqref{eq:layers1d}, \eqref{eq:weight1d}, Proposition~\ref{prop:Multiplication}, and Remark~\ref{rem:NNdefsprops}, we conclude 
\begin{align*}
    L(\Psi_{\varepsilon,N_{\star}}^{\vec{\nu}}) &\leq C\big((1+\log_2N_{\star})(N_{\star}+\log_2(N_{\star}/\varepsilon)) + (1+ \log_2 ((2n^{\gamma(\mathsf{w})}+1)/\varepsilon)\big) \\
    W(\Psi_{\varepsilon,N_{\star}}^{\vec{\nu}}) &\leq C\big((N_{\star}^2 + N_{\star})(N_{\star} + \log_2(N_{\star}/\varepsilon)) + (1+ \log_2 ((2n^{\gamma(\mathsf{w})}+1)/\varepsilon)\big),
\end{align*}
and thus, we are done with the case $\Lambda^{\rm HC}_{n-1}\subset\mathbb{N}^2$.

For the case $d=3$, let $\vec{\nu} = (\nu_1, \nu_2, \nu_3) \in \Lambda_{n-1}^{\rm HC}$ and denote $\varphi_{\mathsf{w}, \vec{\nu}} = \varphi_{\mathsf{w}, \nu_1}\otimes \varphi_{\mathsf{w}, \nu_2}\otimes \varphi_{\mathsf{w}, \nu_3}$.
By \eqref{eq:NNfor1d}, there exists an NN  $\widetilde{\Phi}_{\varepsilon,N_{\star}}^{\nu_1}$ such that
\begin{equation} \label{eq:NNfor1drecall}
    \| \varphi_{\mathsf{w}, \nu_1} - {\rm R} (\widetilde{\Phi}_{\varepsilon,N_{\star}}^{\nu_1}) \|_{L^{\infty}([-1,1])} \leq (1+n^{\gamma(\mathsf{w})})\varepsilon.
\end{equation}
Denote $\vec{\nu}^*= (\nu_2, \nu_3)$.
Then by \eqref{eq:Approx2D}, there also exists an NN $\Psi_{\varepsilon,N_{\star}}^{\vec{\nu}^*}$ satisfying
\begin{equation} \label{eq:diffdimtwo}
    \| \varphi_{\mathsf{w}, \nu_2} \otimes \varphi_{\mathsf{w}, \nu_3}  - {\rm R}(\Psi_{\varepsilon,N_{\star}}^{\vec{\nu}^*}) \|_{L^{\infty}([-1,1]^2)} \leq (3n^{2\gamma(\mathsf{w})} + 4n^{\gamma(\mathsf{w})} +2)\varepsilon.
 \end{equation}
Altogether, using \eqref{eq:ngammaw}, \eqref{eq:NNinftydimone}, \eqref{eq:NNfor1drecall}, \eqref{eq:diffdimtwo}, via similar arguments to those leading to \eqref{eq:Esti1}, we conclude 
\begin{multline} \label{eq:NNdiffdimtwo}
    \| \varphi_{\mathsf{w}, \nu_1} \otimes \varphi_{\mathsf{w}, \nu_2}\otimes \varphi_{\mathsf{w}, \nu_3}  - {\rm R} (\widetilde{\Phi}_{\varepsilon,N_{\star}}^{\nu_1}) \otimes {\rm R}(\Psi_{\varepsilon,N_{\star}}^{\vec{\nu}^*}) \|_{L^{\infty}([-1,1]^3)}
    \\
    \leq (7n^{3\gamma(\mathsf{w})} + 12n^{2\gamma(\mathsf{w})} + 8n^{\gamma(\mathsf{w})} +2)\varepsilon.
\end{multline}
Let $\Phi_I$ be the NN formulation in Proposition \ref{prop:ApproxIdent} (with $d=1$) such that $\Psi_{\varepsilon,N_{\star}}^{\nu_1}:= \Phi_I \odot \widetilde{\Phi}_{\varepsilon,N_{\star}}^{\nu_1}$ and $\Psi_{\varepsilon,N_{\star}}^{\vec{\nu}^*}$ have the same number of layers. 
We construct
\begin{equation*}
    \Psi_{\varepsilon,N_{\star}}^{\vec{\nu}} := \tilde{\times}_{\varepsilon, B'}\odot {\rm P}(\Psi_{\varepsilon,N_{\star}}^{\nu_1}, \Psi_{\varepsilon,N_{\star}}^{\vec{\nu}^*})
\end{equation*}
where $B'=4n^{2\gamma(\mathsf{w})} + 4n^{\gamma(\mathsf{w})} +2$.
Then from Proposition~\ref{prop:Multiplication} again and \eqref{eq:NNdiffdimtwo}, 
\begin{equation*}\label{eq:Approx3D}
    \| \varphi_{\mathsf{w},\nu_1} \otimes \varphi_{\mathsf{w},\nu_2} \otimes \varphi_{\mathsf{w},\nu_3}  - {\rm R} (\Psi_{\varepsilon,N_{\star}}^{\vec{\nu}}) \|_{L^{\infty}([-1,1]^3)} \leq (7n^{3\gamma(\mathsf{w})} + 12n^{2\gamma(\mathsf{w})} + 8n^{\gamma(\mathsf{w})} +3)\varepsilon.
\end{equation*}
Since it is readily verified from the construction of $\Psi_{\varepsilon,N_{\star}}^{\nu_1}$ and \eqref{eq:ngammaw}, \eqref{eq:NNfor1drecall}, \eqref{eq:diffdimtwo} that
\begin{equation*}
    \max\big\{|{\rm R}(\Psi_{\varepsilon,N_{\star}}^{\nu_1})|, |{\rm R}(\Psi_{\varepsilon,N_{\star}}^{\vec{\nu}^*})|\} \leq 4n^{2\gamma(\mathsf{w})} + 4n^{\gamma(\mathsf{w})} +2,
\end{equation*}
employing similar reasoning and computations analogous to those before, delivers
\begin{align*}
    L(\Psi_{\varepsilon,N_{\star}}^{\vec{\nu}}) &\leq C\big((1+\log_2N_{\star})(N_{\star}+\log_2(N_{\star}/\varepsilon)) + (1+ \log_2 ((4n^{2\gamma(\mathsf{w})} + 4n^{\gamma(\mathsf{w})} +2)/\varepsilon)\big) \\
    W(\Psi_{\varepsilon,N_{\star}}^{\vec{\nu}}) &\leq C\big((N_{\star}^2 + N_{\star})(N_{\star} + \log_2(N_{\star}/\varepsilon)) + (1+ \log_2 ((4n^{2\gamma(\mathsf{w})} + 4n^{\gamma(\mathsf{w})} +2)/\varepsilon)\big).
\end{align*}
The proof is now complete.
\end{proof}

\begin{proof}[Proof of Lemma \ref{lem:ApproxOfSlep1D}]
The set of all normalized Legendre polynomials is an orthonormal basis for $L_{\bf u}^2([-1,1])$. 
Therefore, for every Slepian basis function $\varphi_{\mathsf{w}, k}$, there exist $\{\beta_j^k\}_{j\in\mathbb{N}_0} \in \ell^2(\mathbb{N}_0)$ satisfying \eqref{eqdef:betacoef} such that
\begin{equation}\label{eq:L2ConvSlep}
    \lim_{N\to\infty} \Big\| \sum_{j=0}^N \beta_j^k P_j - \varphi_{\mathsf{w}, k} \Big\|_{L^2_{\bf u}([-1,1])} = 0;
\end{equation} 
see also \citep[Equation (2.47)]{osipov2013prolate}.
Since $L^{\infty}([-1,1])$ is a Banach space, the absolute convergence of the series $\sum_{j=0}^{\infty} \beta_j^k P_j$, i.e.
\begin{equation} \label{eq:tailconvergence}
    \lim_{N\to\infty} \sum_{j=N+1}^{\infty} |\beta_j^k| \|P_j\|_{L^{\infty}([-1,1])} = 0,
\end{equation}
will guarantee the existence of the limit $\sum_{j=0}^{\infty} \beta_j^k P_j$ as a function in $L^{\infty}([-1,1])$ \citep[Theorem~5.1]{folland1999real}.
Once held, \eqref{eq:tailconvergence} together with \eqref{eq:L2ConvSlep}, and the continuity of $\sum_{j=0}^N \beta_j^k P_j$ and $\varphi_{\mathsf{w},k}$, yields
\begin{equation} \label{eq:Linftylimit}
    \sum_{j=0}^{\infty} \beta_j^k P_j = \varphi_{\mathsf{w}, k},
\end{equation}
on $[-1,1]$.
We proceed to demonstrate \eqref{eq:tailconvergence} for all $k\in\Lambda$. 
By \citep[Theorem 7.2]{osipov2013prolate}, if $j \geq 2(\lfloor e \mathsf{w} \rfloor +1)$, then 
\begin{equation}\label{eq:beta}
    | \beta_j^k | 
    = \frac{1}{2} \Big| \int_{-1}^1 \varphi_{\mathsf{w},k}(x) \overline{P_j}(x) \, {\rm d}x \Big| \leq \dfrac{1}{\mu_{\mathsf{w}, k}} \Big( \dfrac{1}{2}\Big)^j 
    \leq \frac{1}{c_{\star}} \Big( \dfrac{1}{2}\Big)^j.
\end{equation} 
Using \eqref{eq:NormLegendr}, \eqref{eq:beta}, we find for $N\geq 2\lfloor e \mathsf{w} \rfloor +1$
\begin{align} \label{eq:TailBound1}
    \nonumber \sum_{j=N+1}^{\infty} |\beta_j^k| \| P_j \|_{L^{\infty}([-1,1])} 
    &\leq \dfrac{1}{c_{\star}} \sum_{j=N+1}^{\infty} \dfrac{\sqrt{2j+1}}{2^j}   \\
    \nonumber &\leq \dfrac{1}{c_{\star}} \sum_{j=N}^{\infty} \dfrac{\sqrt{j}}{2^j}\\
    \nonumber &\leq \dfrac{1}{c_{\star}} \sum_{j=N}^{\infty}  \Big(\dfrac{2}{3} \Big)^j \\
    &= \frac{3}{c_{\star}} \Big(\dfrac{2}{3} \Big)^N,
\end{align}
which is \eqref{eq:tailconvergence}, valid for all $k\in\Lambda$.
Hence, \eqref{eq:Linftylimit} holds, and we conclude from \eqref{eq:TailBound1} that, 
\begin{equation} \label{eq:TailBound2}
    \Big\| \sum_{j=0}^N \beta_j^k P_j - \varphi_{\mathsf{w}, k} \Big\|_{L^{\infty}([-1,1])} 
    = \Big \| \sum_{j=N+1}^{\infty}\beta_j^k P_j \Big \|_{L^{\infty}([-1,1])} \leq \frac{3}{c_{\star}} \Big(\dfrac{2}{3} \Big)^N,
\end{equation}
whenever $N\geq 2\lfloor e \mathsf{w} \rfloor -1$. 
For the final step, we require in addition that 
\begin{equation*}
    N\geq \max\Big\{1, \frac{\log(3/(c_{\star}\varepsilon))}{\log (3/2)}\Big\}
\end{equation*}
in \eqref{eq:TailBound2}, to get \eqref{eq:ApprSlepianUnif}, as desired. 
\end{proof}


\begin{proof}[Proof of Lemma~\ref{lem:AppSumOfLeg}]
Let $j=0,\dots,N$. 
A result derived from \citep[Proposition 2.11]{opschoor2022exponential}, and stated as Proposition~\ref{prop:AppLeg1D} in Appendix~\ref{appx:basicnetworks}, guarantees for $j=1,\dots,N$, the existence of an NN $\Phi^j_{\varepsilon}$ with both input dimension and output dimension $d^j_{L_j}$ equal to $1$, such that
\begin{equation} \label{eq:indLeg}
    \| P_j - {\rm R}(\Phi^j_{\varepsilon})\|_{L^{\infty}([-1,1])} \leq \varepsilon,
\end{equation}
and for $j=0$, an NN $\Phi^0_{\varepsilon}$ satisfying ${\rm R}(\Phi^0_{\varepsilon})= P_0=1$ on $[-1,1]$.
Moreover,
\begin{equation} \label{eq:lwpoly}
    L(\Phi^j_{\varepsilon}) \leq C(1+ \log_2 j)(j+ \log_2 (1/\varepsilon)) \quad\text{ and }\quad
    W(\Phi^j_{\varepsilon}) \leq Cj(j+\log_2(1/\varepsilon)),
\end{equation}
and $L(\Phi^0_{\varepsilon}) = 2$, $W(\Phi^0_{\varepsilon}) = 1$.
Therefore, by Proposition~\ref{prop:NNLinComb}, there exists an NN $\widetilde{\Phi}_{\varepsilon, N}$ that fulfills
\begin{equation} \label{eq:lincombrealized}
    {\rm R}(\widetilde{\Phi}_{\varepsilon, N}) = \sum_{j=0}^N \beta_j {\rm R}(\Phi^j_{\varepsilon}),
\end{equation}
with, from \eqref{eq:lwpoly}, the number of layers being at most
\begin{eqnarray} \label{eq:giantlayer}
    L(\widetilde{\Phi}_{\varepsilon, N}) = \max_{0\leq j \leq N} L (\Phi^j_{\varepsilon}) \leq C(1+ \log_2 N)(N+ \log_2 (1/\varepsilon)),
\end{eqnarray}
and the number of weights being at most, 
\begin{align} \label{eq:giantweight}
    \nonumber W (\widetilde{\Phi}_{\varepsilon, N}) 
    &\leq 2 \sum_{j=0}^N W( \Phi^j_{\varepsilon}) + 2\sum_{j=0}^N d^j_{L_j} \Big(\max_{0\leq j\leq N} L (\Phi^j_{\varepsilon})\Big)\\
    \nonumber &\leq C\sum_{j=1}^N j(j+\log_2(1/\varepsilon)) + CN(1+ \log_2 N)(N+ \log_2 (1/\varepsilon))\\
    \nonumber &\leq C (N^2+N)(N+\log_2(1/\varepsilon)) + CN(1+ \log_2 N)(N+ \log_2 (1/\varepsilon))\\
    &\leq C (N^2+N)(N+\log_2(1/\varepsilon)).
\end{align}
Now, from \eqref{eq:indLeg}, \eqref{eq:lincombrealized}, we obtain
\begin{align} \label{eq:sumLeg}
    \nonumber \Big\| \sum_{j=0}^N \beta_j P_j - {\rm R}(\widetilde{\Phi}_{\varepsilon, N}) \Big\|_{L^{\infty}([-1,1])} &= 
    \Big\| \sum_{j=0}^N \beta_j P_j - \sum_{j=0}^N \beta_j {\rm R}(\Phi^j_{\varepsilon}) \Big\|_{L^{\infty}([-1,1])} \\
    \nonumber &\leq \sum_{j=0}^N |\beta_j| \Big\| P_j - {\rm R}(\Phi^j_{\varepsilon})\Big\|_{L^{\infty}([-1,1])} \\
    &\leq c_0N\varepsilon.
\end{align}
By replacing $N\varepsilon$ with $\varepsilon$ in \eqref{eq:giantlayer}, \eqref{eq:giantweight}, \eqref{eq:sumLeg}, we arrive at the desired conclusion. 
\end{proof}

\subsection{Proof of Theorem \ref{thm:ApproxSpan}} \label{sec:ProofTheo3}

Let $g \in \mathcal{S}_{\Lambda}$, i.e.
\begin{equation} \label{eq:lincombSlep}
    g = \sum_{\vec{\nu}\in\Lambda} b_{\vec{\nu}} \varphi_{\mathsf{w},\vec{\nu}},
\end{equation}
for $b_{\vec{\nu}} := \frac{1}{2}\int_{-1}^1 g(x)\varphi_{\mathsf{w},\vec{\nu}}(x) \, {\rm d}x$. 
According to Proposition~\ref{prop:ApproxSlepian}, for a given $\varepsilon>0$ and each Slepian basis function $\varphi_{\mathsf{w},\vec{\nu}}$, there exists an NN $\Psi_{\varepsilon}^{\vec{\nu}}$ such that\footnote{Given that $N_{\star}$, in \eqref{eqdef:Nstar}, depends on $\varepsilon$, we will omit its reference in the subscript and simply write $\Psi^{\vec{\nu}}_{\varepsilon}$ going forward.} 
\begin{equation} \label{eq:indSlepapp}
     \|\varphi_{\mathsf{w}, \vec{\nu}} - {\rm R}(\Psi_{\varepsilon}^{\vec{\nu}}) \|_{L^{\infty}([-1, 1]^d)}
    \leq B(d,n)\varepsilon. 
\end{equation}
In turn, by Proposition~\ref{prop:NNLinComb}, there exists an NN $\Psi_{g;\varepsilon}$ satisfying 
\begin{equation} \label{eq:lincombreal}
    {\rm R}(\Psi_{g;\varepsilon}) = \sum_{\vec{\nu}\in\Lambda} b_{\vec{\nu}} {\rm R}(\Psi_{\varepsilon}^{\vec{\nu}}).
\end{equation}
Thus we obtain
\begin{align} \label{eq:approxg}
    \nonumber \|g - {\rm R}(\Psi_{g;\varepsilon})
    \|_{L^{\infty}([-1, 1]^d)}
    &\leq \sum_{\vec{\nu}\in\Lambda} |b_{\vec{\nu}}| \| \varphi_{\mathsf{w}, \vec{\nu}} - {\rm R}(\Psi_{\varepsilon}^{\vec{\nu}}) \|_{L^{\infty}_{\bf u}([-1, 1]^d)} \\
    &\leq \sqrt{\texttt{\#}\Lambda}\,B(d,n)\|g\|_{L^2_{\bf u}([-1,1]^d)}\varepsilon,
\end{align}
from \eqref{eq:indSlepapp} and the fact that $\|g\|_{L^2_{\bf u}([-1,1]^d)} = (\sum_{\vec{\nu}\in\Lambda} |b_{\vec{\nu}}|^2)^{1/2}$.
It also follows from Proposition~\ref{prop:ApproxSlepian} and another application of Proposition~\ref{prop:NNLinComb} that
\begin{equation} \label{eq:LWSLambda}
    \begin{split}
        L(\Psi_{g;\varepsilon}) 
        &
        \leq C \big((1+\log_2N_{\star})(N_{\star}+\log_2(N_{\star}/\varepsilon)) + (1+ \log_2 (M(d,n)/\varepsilon)\big),\\
        W(\Psi_{g;\varepsilon}) 
        & 
        \leq C \texttt{\#}\Lambda\big((N_{\star}^2 + N_{\star})(N_{\star} + \log_2(N_{\star}/\varepsilon)) + (1+ \log_2 (M(d,n)/\varepsilon)\big).
    \end{split}
\end{equation}
Now, let 
\begin{equation} \label{eqdef:NLambda}
    \mathcal{N}_{\Lambda,\mathsf{w},\varepsilon}:=\{\Psi_{g;\varepsilon}: \Psi_{g;\varepsilon} \text{ satisfies } \eqref{eq:lincombreal} \text{ where } g \text{ takes the form of } \eqref{eq:lincombSlep}\}.
\end{equation}
We will show in what follows that, given $\mathcal{N}_{\Lambda,\mathsf{w},\varepsilon}$ in \eqref{eqdef:NLambda}, the optimization problem \eqref{leastsquaresproblem2} admits a unique solution fulfilling \eqref{eq:PETFinal}.

Recall the matrix $A\in\mathbb{R}^{m\times\texttt{\#}\Lambda}$, introduced in \eqref{eqdef:A}, where $A_{j,\vec{\nu}} = \dfrac{1}{\sqrt{m}} \varphi_{\mathsf{w},\vec{\nu}}(\vec{y}_j)$.
Since the Slepian basis is an orthonormal set, the columns of $A$ are linearly independent.
Consider $\tilde{A}\in\mathbb{R}^{m\times\texttt{\#}\Lambda}$, such that $\tilde{A}_{j,\vec{\nu}} = \dfrac{1}{\sqrt{m}} {\rm R}(\Psi_{\varepsilon}^{\vec{\nu}})(\vec{y}_j)$.
Then it follows from \eqref{eq:indSlepapp} that the columns of $\tilde{A}$ are also linearly independent for sufficiently small $\varepsilon$.
Indeed, let $\varepsilon\in (0,1/(\sqrt{\texttt{\#}\Lambda}\, B(d,n)))$.
If $0 = \sum_{\vec{\nu}\in\Lambda} b_{\vec{\nu}} {\rm R}(\Psi^{\vec{\nu}}_{\varepsilon})$ for some not-all-zero scalars $\{b_{\vec{\nu}}\}_{\vec{\nu}\in\Lambda}$, then
\begin{equation*}
    \sum_{\vec{\nu}\in\Lambda} b_{\vec{\nu}} (\varphi^{\vec{\nu}}_{\varepsilon} - {\rm R}(\Psi^{\vec{\nu}}_{\varepsilon})) = \sum_{\vec{\nu}\in\Lambda} b_{\vec{\nu}} \varphi^{\vec{\nu}}_{\varepsilon},
\end{equation*}
which leads to 
\begin{align*}
    \Big(\sum_{\vec{\nu}\in\Lambda} |b_{\vec{\nu}}|^2\Big)^{1/2} 
    = \Big\| \sum_{\vec{\nu}\in\Lambda} b_{\vec{\nu}} \varphi^{\vec{\nu}}_{\varepsilon} \|_{L^2_{\bf u}([-1,1]^d)} 
    &= \Big\| \sum_{\vec{\nu}\in\Lambda} b_{\vec{\nu}} (\varphi^{\vec{\nu}}_{\varepsilon}-{\rm R}(\Psi^{\vec{\nu}}_{\varepsilon})) \Big \|_{L^2_{\bf u}([-1,1]^d)} \\
    &\leq \sum_{\vec{\nu}\in\Lambda} |b_{\vec{\nu}}| \|\varphi^{\vec{\nu}}_{\varepsilon}-{\rm R}(\Psi^{\vec{\nu}}_{\varepsilon})\|_{L^2_{\bf u}([-1,1]^d)} \\
    &\leq \varepsilon\sqrt{\texttt{\#}\Lambda}\, B(d,n) \Big(\sum_{\vec{\nu}\in\Lambda} |b_{\vec{\nu}}|^2\Big)^{1/2} 
    < \Big(\sum_{\vec{\nu}\in\Lambda} |b_{\vec{\nu}}|^2\Big)^{1/2},
\end{align*}
a contradiction.
Since, due to \eqref{eq:lincombreal}, every element in $\mathcal{N}_{\Lambda,\mathsf{w},\varepsilon}$ is realized as a linear combination of ${\rm R}(\Psi_{\varepsilon}^{\vec{\nu}})$, the problem \eqref{leastsquaresproblem2} is equivalent to finding $(\vec{z})^{\natural} \in \C^{\texttt{\#}\Lambda}$ such that
\begin{equation*}
    (\vec{z})^{\natural} \in \argmin_{\vec{z} \in \C^{\texttt{\#}\Lambda}} 
    \|\tilde{A}\vec{z} - \vec{b}\|_2^2
\end{equation*}
where $\vec{b} = \frac{1}{\sqrt{m}}(f(\vec{y}_j) + \eta_j)_{j=1}^m$. 
We look at the smallest singular value of $\tilde{A}$ next.
By Weyl's Lemma \citep[Theorem 4.3.1]{horn2012matrix}, we have
\begin{equation}\label{eq:Weyls}
    \max_{j =1,\dots,m} | \sigma_j(A) - \sigma_j( \tilde{A}) | \leq \| A - \tilde{A}\|_2
\end{equation}
where $\sigma_j$ denotes the $j$th singular value of $A$. 
Hence, for all unit vector $\vec{z}\in\C^{\texttt{\#}\Lambda}$, 
\begin{align*}
    \| (A-\tilde{A})\vec{z} \|_2^2 
    &= \sum_{j=1}^m | \sum_{\vec{\nu}\in\Lambda} (A_{j,\vec{\nu}} - \tilde{A}_{j,\vec{\nu}}) z_{\vec{\nu}} |^2 \\
    &\leq \sum_{j=1}^m \bigg(\sum_{\vec{\nu}\in\Lambda} |A_{j,\vec{\nu}} - \tilde{A}_{j,\vec{\nu}}| |z_{\vec{\nu}}|\bigg)^2 \\
    &\leq \sum_{j=1}^m \bigg(\sum_{\vec{\nu}\in\Lambda} \frac{1}{\sqrt{m}} \| \varphi_{\mathsf{w}, \vec{\nu}} - {\rm R}(\Psi^{\vec{\nu}}_{\varepsilon}) \|_{L^{\infty}([-1,1]^d)} |z_{\vec{\nu}}|\bigg)^2\\
    &= \bigg(\sum_{\vec{\nu}\in\Lambda} \| \varphi_{\mathsf{w}, \vec{\nu}} - {\rm R}(\Psi^{\vec{\nu}}_{\varepsilon}) \|_{L^{\infty}([-1,1]^d)} |z_{\vec{\nu}}|\bigg)^2\\
    &\leq \sum_{\vec{\nu}\in\Lambda}\| \varphi_{\mathsf{w}, \vec{\nu}} - {\rm R}(\Psi^{\vec{\nu}}_{\varepsilon}) \|_{L^{\infty}([-1,1]^d)}^2 \|\vec{z} \|_2^2 \\
    &\leq \texttt{\#}\Lambda \, B(d,n)^2 \varepsilon^2, 
\end{align*}
where we have used the Cauchy-Schwarz inequality in the third inequality and \eqref{eq:indSlepapp} in the fourth.
Therefore,
\begin{equation}\label{eq:MatricesDif}
    \| A - \tilde{A} \|_2 \leq \sqrt{\texttt{\#}\Lambda}\, B(d,n) \varepsilon.
\end{equation}
We have demonstrated in the proof of Theorem~\ref{thm:leastsquaresampling} that $\sigma_{\rm min}(A)> \sqrt{1-\delta}$ with probability at least $1-\beta$, as long as 
\begin{equation*}
    m \geq ((1-\delta)\log(1-\delta)+\delta)^{-1}(2\texttt{\#}\Lambda)^{2\gamma(\mathsf{w})} \log(\texttt{\#}\Lambda/\beta).
\end{equation*}
Hence, from \eqref{eq:Weyls}, \eqref{eq:MatricesDif}, we get 
\begin{equation} \label{eq:lowbdtildesigma}
    \sigma_{\rm min}(\tilde{A}) \geq \sigma_{\rm min}(A) - \|A - \tilde{A}\|_2
    \geq \sigma_{\rm min}(A) -\sqrt{\texttt{\#}\Lambda}\, B(d,n) \varepsilon
    > \sqrt{1- \delta} - \sqrt{\texttt{\#}\Lambda}\, B(d,n) \varepsilon,
\end{equation}
with the same probability.
Thus, by choosing $\varepsilon\in (0,1/(\sqrt{\texttt{\#}\Lambda}\,B(d,n)))$ additionally sufficiently small such that
\begin{equation*}
    \sqrt{\texttt{\#}\Lambda}\, B(d,n) \varepsilon \leq \frac{1}{2} \sqrt{1 - \delta},
\end{equation*}
which is \eqref{epscondition}, we further obtain from \eqref{eq:lowbdtildesigma}
\begin{equation} \label{eq:minsing}
   \sigma_{\rm min}(\tilde{A}) > \frac{1}{2} \sqrt{1 - \delta}.
\end{equation}
By applying the least squares approximation theory from Section~\ref{sec:problem} and Section~\ref{sec:leastsquaresproof}, specifically leveraging \eqref{lowbdalpha}, \eqref{punchline} and the lower bound given in \eqref{eq:minsing}, we conclude that \eqref{leastsquaresproblem2} has, with probability at least $1-\beta$, a unique solution $\Psi^{\natural}$ such that
\begin{equation*}    
    \| f - {\rm R}(\Psi^{\natural}) \|_{L_{ {\bf u}}^2([-1,1]^d)} 
    \leq 
    \Big(1+ \dfrac{2}{\sqrt{1-\delta}} \Big)\inf_{\Psi\in \mathcal{N}_{\Lambda,\mathsf{w},\varepsilon}} \| f - {\rm R}(\Psi) \|_{L^{\infty}([-1,1]^d)} 
    + \dfrac{2}{\sqrt{1-\delta}} \| \vec{e} \|_2.
\end{equation*}
Additionally, by \eqref{eq:approxg}, for any $\Psi\in\mathcal{N}_{\varepsilon}$ and any $g \in\mathcal{S}_{\Lambda}$,
\begin{align*}
    \| f - {\rm R}(\Psi) \|_{L^{\infty}([-1,1]^d)} 
    &\leq \| f - g\|_{L^{\infty}([-1,1]^d)} + \| g - {\rm R}(\Psi) \|_{L^{\infty}([-1,1]^d)} \\
    &\leq \| f - g\|_{L^{\infty}([-1,1]^d)} + \sqrt{\texttt{\#}\Lambda}\,B(d,n)\|g\|_{L^2_{\bf u}([-1,1]^d)}\varepsilon,
\end{align*}
and we arrive at \eqref{eq:PETFinal}. 
Finally, we note that the necessary upper bounds on the layers and weights for $\Psi^{\natural}$ have already been specified in \eqref{eq:LWSLambda}. \qed

\section{Conclusion}\label{s:conclusions}


We conclude by discussing some limitations of our work and indicating possible directions for future research. First, our theoretical results hold in dimensions $d=1,2, 3$. In principle, our arguments can be applied to higher dimensions, but the results would be affected by the curse of dimensionality. This arises due to the following reasons.
The sample complexity bound in $m$ of Theorem~\ref{thm:leastsquaresampling} relies on the estimation of an upper bound on $\kappa(S_{\Lambda})$, which is given in Proposition~\ref{prop:Theta}. This upper bound estimate follows from Proposition~\ref{prop:relating}, whose proof provides \textit{implicit} dimension-dependent bounds for the cardinalities of the slices of the hyperbolic cross index set.
An earlier version of the proof of Proposition~\ref{prop:relating} relied on the following estimate for the cardinality of the hyperbolic cross, for which an asymptotic version can be found in \citep{dobrovol1998number}.

\begin{lemma}[{See \citep[Theorem 3.5]{chernov2016new} and \citep[Lemma~B.3]{adcock2022sparse}}]
\label{lem:HCsize}
Let $d,n\in\N$, and suppose $\Lambda^{\rm HC}_{n-1}$ is the $d$-dimensional hyperbolic cross of order $n$. Then there exists $n_{\star}(d)\in\mathbb{N}$, depending only on $d$, such that if $n\geq n_{\star}(d)$, 
\begin{equation*}
    \texttt{\#} \Lambda^{\rm HC}_{n-1} < \frac{n(\log n + d\log 2)^{d-1}}{(d-1)!}.
\end{equation*}
\end{lemma}

An estimate for $n_{\star}(d)$ can be derived, for example, from the calculations presented in the proof of \citep[Theorem~2.4]{chernov2016new}. In particular, we can take $n_{\star}(2)=1$.
However, for $d \geq 3$, $n_{\star}(d)$ is no longer guaranteed to be $1$. Moreover, \eqref{crucialinequality} may only hold above a separate threshold in $n$ for each $d$.
This prevents a straightforward inductive argument across dimensions that involves the intricate geometry of the hyperbolic cross in higher dimensions.

While Proposition~\ref{prop:relating} is integral to our proof strategy for circumventing the dimensional impact in the upper bound of $\kappa(\mathcal{S}_{\Lambda})$, the potential for the curse of dimensionality to emerge remains evident, as $\kappa(\mathcal{S}_{\Lambda}) \geq \|\varphi_{\mathsf{w},\vec{0}}\|_{L^{\infty}([-1,1]^d)}^2$, and the latter quantity is of order $\mathcal{O}(\mathsf{w}^d)$, by Proposition~\ref{prop:Phi_j}.
A potential fix to mitigate this problem is to consider a \emph{preconditioning scheme} (see \citep[Chapter 5]{adcock2022sparse}) where samples are drawn from an orthogonal measure on $[-1,1]^d$, alongside leveraging hyperbolic crosses of anisotropic type. 
Alternatively, recovery strategies based on compressed sensing (see \citep[Chapter 7]{adcock2022sparse} and the references therein) could be explored. These considerations fall outside the scope of this paper, but they present promising directions for future research.

The error bounds in Theorems~\ref{thm:leastsquaresampling} and \ref{thm:ApproxSpan} depend on the best approximation error of the target function $f$ with respect to the Slepian basis $\{\varphi_{\mathsf{w},\vec{\nu}}\}_{\vec{\nu}\in\Lambda}$. It is possible to obtain approximation rates that depend on the number of samples $m$ for specific function classes, similar to the approach in \citep{adcock2022deep} for analytic functions. 
To address this approximation-theoretical challenge, one would need to define a suitable class of (approximately) bandlimited functions for which the best approximation error $\inf_{g \in \mathcal{S}_\Lambda} \| f - g \|_{L^{\infty}([-1,1]^d)}$ can be explicitly estimated in terms of $\texttt{\#}\Lambda$. This type of bound is not readily available in the literature and so would need to be derived. Another interesting open question is whether the error bound in \eqref{eq:PETFinal} holds when the inf is taken over the neural networks in $\mathcal{N}_{\Lambda, \mathsf{w}, \varepsilon}$ as opposed to $\mathcal{S}_{\Lambda}$.

Additionally, as noted in Section~\ref{sec:main_results}, the NNs in the class $\mathcal{N}_{\Lambda, \mathsf{w}, \varepsilon}$ of Theorem~\ref{thm:ApproxSpan} are not fully trained like those considered in our numerical experiments. In fact, their first-to-second-last layers are explicitly constructed to emulate the Slepian basis, and the last layer is trained via least squares. Addressing this gap between theory and practice is an important direction for future work. 
One possible approach is given by the proof strategy of \citep[Theorem 3.2]{adcock2024optimal}, which shows that training problems based on any family of fully connected NNs possess uncountably many minimizers that achieve accurate approximations. 

Finally, our numerical experiments in dimensions one and two demonstrate that both reconstruction methods using least squares and deep learning 
improve as the number of training samples increases. 
Although least squares approximation is far more effective in dimension one, deep learning can match and slightly surpass it in dimension two.
However, identifying an NN architecture that consistently achieves superior performance across different benchmark functions is an open question that deserves a more comprehensive numerical investigation. Additionally, our experiments in Section~\ref{sec:initializations} show that a novel Slepian-based initialization---based on the construction of Theorem~\ref{thm:ApproxSpan}---can outperform standard random initializations when approximating a one-dimensional frequency localized test function. Carrying out a more extensive experimentation to see if this advantage can be seen for more functions and in higher dimensions is an interesting avenue of future work.

\section*{Acknowledgments}

AMN is supported by the Austrian Science Fund (FWF) Project P-37010. SB and JJB were partially supported by the Natural Sciences and Engineering Research Council of Canada (NSERC) through grants RGPIN-2020-06766 and RGPIN-2023-04244, respectively, and the Fonds de Recherche
du Qu\'ebec Nature et Technologies (FRQNT) through grants 313276, 340894 and 359708, respectively.

\bibliographystyle{plain}
\bibliography{biblio}

\appendix

\section{Basic neural network calculus} \label{appx:basicnetworks}

We first introduce a classical result that constructively demonstrates the existence of an NN for approximating products within a compact region, as detailed below.

\begin{proposition}[{\citep[Proposition 2.5]{opschoor2022exponential}}] \label{prop:Multiplication}
Let $\varepsilon \in (0,1)$, $B>0$, and $\times: \mathbb{R}^2 \to \R$ be defined by $\times(x,y)=xy$. 
Then, there exists a ReLU NN $\widetilde{\times}_{\varepsilon,B}$ such that ${\rm R}(\widetilde{\times}_{\varepsilon,B}): [-B,B]^2\to \mathbb{R}$ and that
\begin{equation*}
   \| {\rm R}(\widetilde{\times}_{\varepsilon,B})- \times \|_{L_{\bf u}^{\infty}([-B,B]^2)} \leq \varepsilon.
\end{equation*}
Moreover, there is a universal constant $C>0$ such that 
\begin{equation*}
   L (\widetilde{\times}_{\varepsilon,B}) \leq C(1+ \log_2(B/\varepsilon)) 
\quad\text{ and }\quad
    W (\widetilde{\times}_{\varepsilon,B}) \leq C(1+ \log_2(B/\varepsilon)).
\end{equation*}
\end{proposition}

Continuing, we present two foundational network operations; they are, concatenation, as described in Definition~\ref{def:concat}, and parallelization, as described in Definition~\ref{def:Parallelization}.

\begin{definition}[{\citep[Definition~2.2]{PetV2018OptApproxReLU}}]\label{def:concat}
Let $L_1, L_2 \in \N$ and suppose 
\begin{equation*}
    \Phi^1 = ((A_1^1,b_1^1), \dots, (A_{L_1}^1,b_{L_1}^1)) 
    \quad\text{ and }\quad
    \Phi^2 = ((A_1^2,b_1^2), \dots, (A_{L_2}^2,b_{L_2}^2))
\end{equation*}
are two NNs with the number of layers $L_1$ and $L_2$, respectively, such that the input layer of $\Phi^1$ has the same dimension as that of the output layer of $\Phi^2$. Then, $\Phi^1 \odot \Phi^2$ is an NN with $L_1+L_2-1$ layers such that
\begin{equation*}
    \Phi^1 \odot \Phi^2 := ((A_1^2,b_1^2), \dots, (A_{L_2-1}^2,b_{L_2-1}^2), (A_1^1A_{L_2}^2, A_1^1b_{L_2}^2+b_1^1) , (A_2^1,b_2^1), \dots, (A_{L_1}^1,b_{L_1}^1)),
\end{equation*}
called the concatenation of $\Phi^1$ and $\Phi^2$.  Moreover, ${\rm R}(\Phi^1 \odot \Phi^2) = {\rm R}(\Phi^1) \circ {\rm R}(\Phi^2)$. 
\end{definition}

\begin{definition}[{\citep[Definition~2.7]{PetV2018OptApproxReLU}}]\label{def:Parallelization} 
Let $d, L \in \N$ and suppose 
\begin{equation*}
    \Phi^1 = ((A_1^1,b_1^1), \dots, (A_{L}^1,b_{L}^1))
    \quad\text{ and }\quad
    \Phi^2 = ((A_1^2,b_1^2), \dots, (A_{L}^2,b_{L}^2))
\end{equation*}
are NNs both with $L$ layers and $d$-dimensional input.
Then, ${\rm P}(\Phi^1,\Phi^2)$ is an NN with $L$ layers such that
\begin{equation*}
    {\rm P}(\Phi^1,\Phi^2):= ((\tilde{A}_1,\tilde{b}_1), \dots, (\tilde{A}_L,\tilde{b}_L) )
\end{equation*}
where 
\begin{equation*}
\widetilde{A}_{1}:=\begin{pmatrix}
	A_{1}^{1} \\
	A_{1}^{2}  
\end{pmatrix},\quad
\widetilde{b}_{1}:=\begin{pmatrix}
	b_{1}^{1} \\
	b_{1}^{2}  
\end{pmatrix},
\,\text{ and }\,
\widetilde{A}_{l}:=\begin{pmatrix}
	A_{l}^{1} & 0 \\
	0 & A_{l}^{2}  
\end{pmatrix},\quad
\widetilde{b}_{l}:=\begin{pmatrix}
	b_{l}^{1} \\
	b_{l}^{2}  
\end{pmatrix},
\, \text{ for } \,
l= 2, \dots, L,
\end{equation*}
called the parallelization of $\Phi^1$ and $\Phi^2$. Moreover, ${\rm R}({\rm P}(\Phi^1,\Phi^2))= ({\rm R}(\Phi^1), {\rm R}(\Phi^2))$.
\end{definition}

The following observation is immediate from the two preceding definitions. 

\begin{remark} \label{rem:NNdefsprops}
It can be easily seen that
\begin{equation*} 
    L(\Phi^1 \odot \Phi^2) \leq L(\Phi^1) + L(\Phi^2), \quad\text{ and }\quad
    W(\Phi^1 \odot \Phi^2) \leq W(\Phi^1) + W(\Phi^2),
\end{equation*}
and that 
\begin{equation*} 
    L({\rm P}(\Phi^1, \Phi^2)) = L(\Phi^1) = L(\Phi^2),
    \quad\text{ and }\quad
    W({\rm P}(\Phi^1, \Phi^2)) = W(\Phi^1) + W(\Phi^2).
\end{equation*} 
It is also readily confirmed that concatenation and parallelization can be extended naturally to involve multiple NNs, such that
\begin{align*}
    \Phi^1\odot \Phi^2 \odot \Phi^3 &= (\Phi^1\odot \Phi^2)\odot \Phi^3 = \Phi^1\odot (\Phi^2 \odot \Phi^3), \\
    {\rm P}(\Phi^1,\Phi^2,\Phi^3) &= {\rm P}({\rm P}(\Phi^1,\Phi^2),\Phi^3) = {\rm P}(\Phi^1,{\rm P}(\Phi^2,\Phi^3)).
\end{align*}
\end{remark}

A key result, shown below, establishes that the identity function can be effectively approximated by an NN with any chosen number of layers.

\begin{proposition}[{\citep[Proposition~2.4]{FEMNNsPetersenSchwab}}] \label{prop:ApproxIdent} 
For every $d,L \in \N$, if $I: \R^d \to \R^d$ denotes the identity map $I(x) = x$ for all $x \in \R^d$, then there exists an NN $\Phi_I$ with input dimension $d$ and $L$ layers, such that ${\rm R}(\Phi_I)(x) = I(x) = x$ for all $x \in \R^d$ and $W(\Phi_I) \leq 2dL$.
\end{proposition}

We provide another key result, which, as a culmination of the previous definitions and proposition, demonstrates that NNs can be linearly combined.

\begin{proposition}[{\citep[Lemma 5.4]{petersen2024mathematical}}]
\label{prop:NNLinComb}
For $d, K \in \N$, let $\{\Phi_j\}_{j = 1}^K$ be a collection of NNs where
\begin{equation*}
    \Phi^j = ((A_1^j,b_1^j), \dots, (A_{L_j}^j,b_{L_j}^j))
\end{equation*}
where $A^j_i \in \R^{d^j_i \times d^j_{i-1}}$, for $i=1,\dots, L_j$ and assume $d^1_{L_1} = d^2_{L_2} = \dots = d^K_{L_K}$. For a set of scalars $\{\alpha_j\}_{j=1}^K$, there is an NN $\Phi_{\Sigma}$ such that 
\begin{equation*}
     {\rm R} (\Phi_{\Sigma}) 
     = \sum_{j=1}^K \alpha_j {\rm R} (\Phi_{j}), 
\end{equation*}
and that
\begin{equation*}
    W  (\Phi_{\Sigma}) \leq 2 \sum_{j=1}^K W (\Phi_{j}) + 2 \sum_{j=1}^K (L_{\rm max} - L_j)d_{L_j}^j
    \quad\text{ and }\quad
    L  (\Phi_{\Sigma}) = \max_{1\leq j \leq K} L (\Phi_{j}). 
\end{equation*}
where $L_{\rm max} := \max_{1\leq j \leq K} L (\Phi_{j}) = \max_{1\leq j \leq K} L_j$.
\end{proposition}

Our final result addresses the polynomial approximability of NNs and is derived from \citep[Proposition 2.11]{opschoor2022exponential}.

\begin{proposition} 
\label{prop:AppLeg1D} 
Let $\varepsilon\in (0,1)$ and $P_k$ be a normalized Legendre polynomial \eqref{eq:NormLegendr} of degree $k \in \N_0$. For each $k\in\N$ there exists an NN $\Phi^k_{\varepsilon}$ with input and output of dimension one satisfying
\begin{equation*}
    \| P_k  - {\rm R}(\Phi^k_{\varepsilon}) \|_{L^{\infty}([-1,1])} \leq \varepsilon,
\end{equation*}
and for $k=0$, an NN $\Phi^0_{\varepsilon}$ satisfying $P_0 = {\rm R}(\Phi^0_{\varepsilon}) = 1$ on $[-1,1]$. 
Moreover, there exists a universal constant $C>0$ 
such that for $k\in\N$
\begin{equation*}
    L(\Phi^k_{\varepsilon}) \leq C(1+ \log_2 k)(k+ \log_2 (1/\varepsilon))
    \quad\text{ and }\quad
    W(\Phi^k_{\varepsilon}) \leq Ck(k+\log_2(1/\varepsilon)),
\end{equation*}
and in the case $k=0$, $L(\Phi^0_{\varepsilon}) = 2$ and $W(\Phi^0_{\varepsilon}) = 1$.
\end{proposition}

\end{document}